\documentclass[reqno]{amsart}
\usepackage[utf8]{inputenc}

\usepackage{amsfonts, amsmath,  amssymb, color,colortbl, enumerate, esint, fancyvrb, mathrsfs, tikz}  

\renewcommand{\r}{\color{red}}

\usepackage[normalem]{ulem}
\usepackage{marginnote}

\usepackage{pgfplots, subcaption}

\pgfplotsset{compat=1.10}
\usepgfplotslibrary{fillbetween}
\usetikzlibrary{patterns}

\pgfdeclarelayer{ft}
\pgfdeclarelayer{bg}
\pgfsetlayers{bg,main,ft}

\VerbatimFootnotes

\usetikzlibrary{positioning}
\VerbatimFootnotes

\newcommand{\C}{\mathbb{C}}
\newcommand{\R}{\mathbb{R}}
\newcommand{\Z}{\mathbb{Z}}
\newcommand{\N}{\mathbb{N}}

\newcommand{\T}{\mathbb{T}}

\renewcommand{\O}{\mathcal{O}}
\newcommand{\Sc}{\mathcal{S}}

\newcommand{\fh}{\mathfrak{h}}

\newcommand{\bZ}{\mathbf{Z}}
\newcommand{\bY}{\mathbf{Y}}

\newcommand{\cB}{\mathcal{B}}

\newcommand{\sE}{\mathscr{E}}

\newcommand{\ta}{\texttt{a}}
\newcommand{\bta}{\mbox{\small$\ta$}}
\newcommand{\sta}{\mbox{\scriptsize$\ta$}}
\newcommand{\tc}{\texttt{c}}
\newcommand{\btc}{\mbox{\small$\tc$}}
\newcommand{\stc}{\mbox{\scriptsize$\tc$}}

\newcommand{\BL}[1]{\mathrm{BL}_{#1}}
\newcommand{\tang}{\mathrm{tang}}
\newcommand{\trans}{\mathrm{trans}}
\newcommand{\cell}{\mathrm{cell}}
\newcommand{\alg}{\mathrm{alg}}

\renewcommand{\le}{\leqslant}
\renewcommand{\ge}{\geqslant}
\renewcommand{\leq}{\leqslant}
\renewcommand{\geq}{\geqslant}

\newcommand{\bz}{\mathbf{z}}
\newcommand{\bx}{\mathbf{x}}
\newcommand{\by}{\mathbf{y}}

\newcommand{\dist}{\operatorname{dist}}
\newcommand{\maxdist}{\operatorname{maxdist}}

\theoremstyle{plain}
\newtheorem{theorem}{Theorem}[section]
\newtheorem{lemma}[theorem]{Lemma}
\newtheorem{proposition}[theorem]{Proposition}
\newtheorem{corollary}[theorem]{Corollary}
\newtheorem{conjecture}[theorem]{Conjecture}
\newtheorem*{first key estimate}{First key estimate}
\newtheorem*{second key estimate}{Second key estimate}

\theoremstyle{definition}
\newtheorem{definition}[theorem]{Definition}

\newtheorem{remark}[theorem]{Remark}
\newtheorem*{acknowledgement}{Acknowledgement}
\newtheorem*{induction hypothesis}{Induction hypothesis}
\newtheorem*{notation}{Notation}

\title[Improved bounds for the Kakeya maximal conjecture]{Improved bounds for the Kakeya maximal conjecture in higher dimensions}

\author{Jonathan Hickman}
\address{School of Mathematics, James Clerk Maxwell Building, The King's Buildings, Peter Guthrie Tait Road
Edinburgh, EH9 3FD, UK}
\email{jonathan.hickman@ed.ac.uk}

\thanks{Supported by the MINECO grants SEV-2015-0554 and MTM2017-85934-C3-1-P and the ERC grant 834728}

\author{Keith M. Rogers}

\address{Instituto de Ciencias Matem\'aticas CSIC-UAM-UC3M-UCM, Madrid 28049, Spain}
\email{keith.rogers@icmat.es}

\author{Ruixiang Zhang}
\address{Department  of  Mathematics,  University  of  Wisconsin-Madison, 480 Lincoln Dr, Madison, WI-53706, USA}
\email{rzhang347@wisc.edu}

\begin{document}
\begin{abstract}  We adapt Guth's polynomial partitioning argument for the Fourier restriction problem to the context of the Kakeya problem. By writing out the induction argument as a recursive algorithm, additional multiscale geometric information is made available. To take advantage of this, we prove that direction-separated tubes satisfy a multiscale version of the polynomial Wolff axioms. Altogether, this yields improved bounds for the Kakeya maximal conjecture in~$\mathbb{R}^n$ with $n=5$ or $n\ge 7$ and improved bounds for the Kakeya set conjecture for an infinite sequence of dimensions. 
\end{abstract}

\maketitle




\section{Introduction}

For $n \geq 2$ and small $\delta>0$, a \emph{$\delta$-tube} is a cylinder $T\subset \mathbb{R}^n$ of unit height and radius~$\delta$, with arbitrary position and arbitrary orientation $\mathrm{dir}(T)\in S^{n-1}$. A family~$\T$ of $\delta$-tubes is \emph{direction-separated} if $\{\mathrm{dir}(T) : T \in \T\}$ forms a $\delta$-separated subset of the unit sphere.

\begin{conjecture}[Kakeya maximal conjecture]\label{maximal conjecture} Let  $p\ge \frac{n}{n-1}$. For all $\varepsilon > 0$, there exists a constant $C_{\varepsilon, n}>0$ such that
\begin{equation}\label{maximal inequality}\tag{ $\mathrm{K}_p$}
\Big\| \sum_{T\in\mathbb{T}} \chi_T\Big\|_{L^p(\mathbb{R}^n)}\,\leq\,  C_{\varepsilon, n} \delta^{-(n-1 - n/p)-\varepsilon} \Big(\sum_{T \in \T} |T| \Big)^{1/p}
\end{equation}
whenever $0 < \delta < 1$ and $\mathbb{T}$ is a direction-separated family of $\delta$-tubes. 
\end{conjecture} 

By an application of H\"older's inequality, one may readily verify that if \eqref{maximal inequality} holds for $p = \frac{n}{n-1}$, then, for all $\varepsilon>0$,  there exists a constant $c_{\varepsilon, n}>0$ such that
\begin{equation*}
    \big|\bigcup_{T \in \T} T \big| \,\geq\, c_{\varepsilon,n} \delta^{\varepsilon} \sum_{T \in \T} |T|.
\end{equation*}
This can be interpreted as the statement that any direction-separated family of  $\delta$-tubes is `essentially disjoint'. A more refined argument shows that if \eqref{maximal inequality} holds for a given $p$, then every Kakeya set in $\R^n$ (that is, every compact set that contains a unit line segment in every direction) has Hausdorff dimension at least~$p'$, the conjugate exponent of $p$. Thus, Conjecture~\ref{maximal conjecture} would imply the \emph{Kakeya set conjecture}, that Kakeya sets in~$\R^n$  have Hausdorff dimension~$n$; see, for instance, \cite{Bourgain1991a, Wolff1999, KT2002b}.

For $n=2$, the set conjecture was proven by Davies \cite{Davies1971} and the maximal conjecture was proven by C\'ordoba~\cite{Cordoba1977} in the seventies. Both conjectures remain challenging and important open problems in higher dimensions; for partial results, see~\cite{Drury1983, C1984, CDR1986, Bourgain1991a, Wolff1995, Schlag1998, TVV1998, Bourgain1999, KT1999, KLT2000, LT, KT2002,  BCT2006, Dvir2009, Guth2010, Guth2016b, GR, KZ2019} and references therein. 

In 1999, Bourgain~\cite{Bourgain1999} improved the state-of-the-art in higher dimensions using sum-difference theory from additive combinatorics. This technique was refined by Katz and Tao~\cite{KT1999, KT2002, KT2002b}, proving that Conjecture~\ref{maximal conjecture} is true in the range $p\ge 1+\frac{7}{4}\frac{1}{n-1}$. The purpose of the present article is to extend this range using a different approach.

\begin{theorem}\label{linear theorem} Conjecture \ref{maximal conjecture} is true in the range 
\begin{equation}\label{linear exponent}
p \geq 1+\min_{2\leq k\leq n}\max\Big\{\, \frac{2n}{(n-1)n + (k-1)k} , \,\frac{1}{n-k+1}\,\Big\}.
\end{equation}
\end{theorem}

When $k=n$, the first entry of the maximum of \eqref{linear exponent} takes the conjectured value; however, the second entry only reaches this value at the other extreme, when $k=2$. A reasonable compromise can be found by taking $k$ to be the closest integer to $(\sqrt{2}-1)n+1$, at which point we find, for instance, that the Kakeya maximal conjecture holds in the range\footnote{In all dimensions the range \eqref{linear exponent} is in fact strictly larger than \eqref{easy linear exponent}: the latter is included to provide a ready comparison with the maximal bounds from \cite{KT2002}.}
\begin{equation}\label{easy linear exponent}
p\ge 1+\frac{1}{2-\sqrt{2}}\frac{1}{n-1},
\end{equation}
which is an improvement over the Katz--Tao maximal bound \cite{KT2002}. See Figure~\ref{maximal} for the state-of-the-art in low dimensions. Theorem~\ref{linear theorem} also implies improved bounds for the Kakeya set conjecture in certain dimensions. Further discussion of the numerology is contained in the final section of the article.
\renewcommand{\arraystretch}{1.3}
 \begin{figure}
\begin{tabular}{ ||c|c|c||c|c|c|| } 
 \hline
  $n=$ & $p\ge$  &  & $n=$ & $p\ge $  &  \\ 
   \hline 
 2 &  \cellcolor{white} 2 &  \cellcolor{white} C\'ordoba~\cite{Cordoba1977}  &   9 &  \cellcolor{yellow!50} $6/5$ &  \cellcolor{yellow!50} Theorem~\ref{linear theorem}\\ 
  3 & \cellcolor{white} $5/3-\varepsilon$ &  \cellcolor{white} Katz--Zahl~\cite{KZ2019,KZ2}  &  10 & \cellcolor{yellow!50}   $13/11$ &  \cellcolor{yellow!50} Theorem~\ref{linear theorem}  \\
 4 &   \cellcolor{white} $1.4794...$ &    \cellcolor{white} Katz--Zahl~\cite{KZ2}  & 11 &  \cellcolor{yellow!50} $7/6$ &  \cellcolor{yellow!50} Theorem~\ref{linear theorem}\\  
 5 & \cellcolor{yellow!50} $18/13$ & \cellcolor{yellow!50} Theorem~\ref{linear theorem}    &  12 & \cellcolor{yellow!50}   $31/27$ &  \cellcolor{yellow!50} Theorem~\ref{linear theorem}   \\ 
 6 &  \cellcolor{white} $4/3$ &  \cellcolor{white} Wolff~\cite{Wolff1995}  &    13 &  \cellcolor{yellow!50} $106/93$ &  \cellcolor{yellow!50} Theorem~\ref{linear theorem}\\ 
 7 &  \cellcolor{yellow!50} $34/27$ &  \cellcolor{yellow!50} Theorem~\ref{linear theorem}  & 14 & \cellcolor{yellow!50}   $9/8$ &  \cellcolor{yellow!50} Theorem~\ref{linear theorem}   \\ 
 8 & \cellcolor{yellow!50}   $21/17$ &  \cellcolor{yellow!50} Theorem~\ref{linear theorem}   &   15 & \cellcolor{yellow!50}   $47/42$&  \cellcolor{yellow!50} Theorem~\ref{linear theorem}\\ 
  \hline
\end{tabular}
    \caption{The state-of-the-art for the Kakeya maximal conjecture in low dimensions. New results are \colorbox{yellow!50}{highlighted.}
    }\label{maximal}
    \end{figure}

The proof of Theorem~\ref{linear theorem} is based on the \textit{polynomial method}, which was introduced in the context of the Kakeya problem by Dvir in his celebrated proof~\cite{Dvir2009} of Wolff's finite field  Kakeya conjecture \cite{Wolff1999}. The polynomial method has been adapted to analyse Kakeya sets in Euclidean space in, for instance, works of  Guth \cite{Guth2010, Guth2016b} and Guth and Zahl \cite{GZ2018}. A key tool here is {\it polynomial partitioning}, introduced by Guth and Katz in their resolution of the two dimensional Erd\H{o}s distance conjecture~\cite{GK2015}. Of most relevance to the present article is the recent work of Guth~\cite{Guth2016,Guth2018} which adapted the partitioning technique to the context of the Fourier restriction problem. 

In \cite{Guth2016,Guth2018,GZ2018}, polynomial partitioning was used to study collections of direction-separated tubes. This led to the consideration of configurations of tubes that are partially contained in the neighbourhood of a real algebraic variety. Guth proved the following cardinality estimate for direction-separated tubes in three dimensions \cite[Lemma~4.9]{Guth2016} and conjectured that it should hold in higher dimensions \cite[Conjecture B.1]{Guth2018}. This was confirmed by Zahl~\cite{Zahl2018} in four dimensions and then in general by Katz and the second author~\cite{KR2018}.

\begin{theorem}[\cite{Guth2016, Zahl2018, KR2018}] \label{PWA} For all~$n\ge k\ge 1$, $d\ge 1$ and $\varepsilon>0$, there is a constant $C_{n,d,\varepsilon}>0$ such that
\begin{equation*}\#\Big\{\, T \in \mathbb{T} :  |T \cap B_{\lambda_k} \cap N_{\!\rho} \bZ_k|\ge \lambda_k|T| \,\Big\} \leq C_{n,d,\varepsilon} \Big(\frac{\rho}{\lambda_k}\Big)^{n-k}\delta^{-(n-1)-\varepsilon}
\end{equation*}
whenever $0<\delta\le \rho\le \lambda_k\le 1$, $\mathbb{T}$ is a direction-separated family of $\delta$-tubes and $\bZ_k\subset \mathbb{R}^n$ is a $k$-dimensional algebraic variety of degree $\le d$.
\end{theorem}
Here $N_{r}E$ denotes the $r$-neighbourhood of $E$ for any $r > 0$ and $E \subseteq \R^n$ and $B_{r}$ is a choice of ball in $\R^n$ of radius $r$. The relevant algebraic definitions are recalled in Section~\ref{Algebraic definitions section} below. In the language of \cite{Guth2018}, this theorem states that direction-separated tubes satisfy the \textit{polynomial Wolff axioms}; this terminology is recalled and discussed in further detail in the final section of the paper.

After adapting Guth's restriction argument \cite{Guth2016,Guth2018} to the context of the Kakeya maximal problem, one finds that Theorem~\ref{PWA} can be used to obtain improved bounds in certain intermediate dimensions: see the final section for more details. However, by rewriting Guth's induction argument as a recursive algorithm, one is readily able to take advantage of the the following strengthened version of Theorem~\ref{PWA}.

\begin{theorem}\label{mainThm}
For all $n\ge m\ge k\ge 1$, $d\ge 1$ and $\varepsilon>0$,  there is a constant $C_{n,d,\varepsilon}>0$ such that
\begin{equation*}
\#\bigcap_{j = k}^m \Big\{ T \in \mathbb{T} :  |T \cap B_{\lambda_j} \cap N_{\!\rho} \bZ_j|\ge \lambda_j|T|\Big\} \leq C_{n,d,\varepsilon}  \Big(\prod_{j = k}^{m-1} \frac{\rho}{\lambda_j}\Big) \Big(\frac{\rho}{\lambda_m}\Big)^{n-m}\delta^{-(n-1)-\varepsilon}
\end{equation*}
whenever $0<\delta\le \rho\le \lambda_k\le \ldots\le \lambda_m\le 1$, $\mathbb{T}$ is a direction-separated family of $\delta$-tubes,  $\bZ_j\subset \mathbb{R}^n$ are $j$-dimensional algebraic varieties of degree $\le d$ and the balls $B_{\lambda_j}$ are nested: $B_{\lambda_{k}}\subseteq \ldots \subseteq B_{\lambda_{m}}\subset \mathbb{R}^{n}$.
\end{theorem}

Taking the varieties~$\bZ_j$ to be nested $j$-planes reveals that the cardinality estimate  of Theorem~\ref{mainThm} is sharp up to the  factor of $ C_{n,d,\varepsilon}\delta^{-\varepsilon}$.  The proof will follow the argument of \cite{KR2018} once a relevant Wongkew-type volume bound (in the spirit of \cite{Wongkew1993}) has been established. The mixture of trigonometric and algebraic arguments involved in the proof of this volume bound constitutes the most novel part of the article. 

\begin{remark} In a late stage of the development of this project, the authors discovered that J. Zahl has proved the same maximal results as Theorem~\ref{linear theorem} using similar methods. In particular, J. Zahl has independently established Theorem~\ref{mainThm} and, moreover, was able to use this result to prove a strengthened version of Theorem~\ref{main theorem} involving $k$-linear (as opposed to $k$-broad) estimates.
\end{remark}

The remainder of the article is organised as follows:
\vspace{0.5em}
\begin{itemize}
\item In Section \ref{notation section} some notational conventions are fixed.
\item In Section \ref{yes} the proof of Theorem~\ref{mainThm} is presented after first  establishing the relevant Wongkew-type volume bound.
    \item In Section \ref{reduction section} the proof of  Theorem~\ref{linear theorem} is reduced to estimating the so-called \emph{$k$-broad norms} for the Kakeya maximal function, paralleling work on oscillatory integrals from~\cite{BG2011, Guth2016, Guth2018}.
     \item In Section~\ref{k-broad norms section} basic properties of $k$-broad norms are reviewed.
    \item In Section~\ref{polynomial partitioning section} the polynomial partitioning theorem from~\cite{Guth2018} is recalled and applied to the $k$-broad norms.
      \item In Section~\ref{structure lemma section} the recursive algorithm is described,  culminating in a structural statement of algebraic nature for the Kakeya maximal problem. 
    \item In Section~\ref{proof section} the structural statement is combined with Theorem~\ref{mainThm} to conclude the proof of Theorem~\ref{linear theorem}. 
     \item In Section \ref{final remarks} the applications to the Kakeya set conjecture and other related problems are discussed. 
     \item Appended is a review of some facts from real algbraic geometry used in Section~\ref{yes}.
\end{itemize}

\vspace{0.5em}

\begin{acknowledgement} The first author thanks both Larry Guth and Joe Karmazyn for helpful discussions during the development of this project.  
\end{acknowledgement}




\section{Notational conventions}\label{notation section}

 We call an $n$-dimensional ball $B_r$ of radius~$r$ an {\it $r$-ball}. The intersection of $S^{n-1}$ with a ball is called a \textit{cap}. The $\delta$-neighbourhood of a set $E$ will be denoted by $N_{\;\!\!\delta}E$.

The arguments will involve the {\it admissible parameters} $n$, $p$ and $\varepsilon$ and the constants in the estimates will be allowed to depend on these quantities. Moreover, any constant is said to be \textit{admissible} if it depends only on the admissible parameters. Given positive numbers $A, B \geq 0$ and a list of objects $L$, the notation $A \lesssim_L B$, $B \gtrsim_L A$ or $A = O_L(B)$ signifies that $A \leq C_L B$ where $C_L$ is a constant which depends only on the objects in the list and the admissible parameters. We write $A \sim_L B$ when both $A \lesssim_L B$ and $B \lesssim_L A$.

 The cardinality of a finite set $A$ is denoted by $\#A$. A set $A'$ is said to be a \emph{refinement} of $A$ if $A' \subseteq A$ and $\#A' \gtrsim \#A$. In many cases it will be convenient to \emph{pass to a refinement} of a set $A$, by which we mean that the original set $A$ is replaced with some refinement.




\section{Multiscale polynomial Wolff axioms: Proof of Theorem~\ref{mainThm}}\label{yes}

 In this section we prove Theorem~\ref{mainThm}. A minor modification of the argument used to prove Theorem~\ref{PWA} in \cite{KR2018} reduces matters to establishing a ``Wongkew-type lemma''. The details of this reduction are described in Section~\ref{multiscale subsection} below. In the simplest case where $k=m$ (which corresponds to Theorem~\ref{PWA}), after the reduction all that is needed is Wongkew's original lemma \cite{Wongkew1993}, which is used to bound the volume of the semialgebraic set $\bZ_k \cap B_{\lambda_k}$. In the general case the problem is to obtain bounds for the volume of other semialgebraic sets $S_m(I_m, \rho)$ which do not fall directly under the scope of \cite{Wongkew1993}. These sets arise from the multiscale hypotheses and are defined in Section~\ref{Wongkew section}.




 \subsection{Algebraic definitions}\label{Algebraic definitions section} Before continuing, it is perhaps useful to clarify some of the terminology featured in the statement of Theorem~\ref{mainThm} and also in the proof. 
 
 \begin{definition}  A set $\bZ \subseteq \R^n$ will be referred to as a \emph{variety} if it can be expressed as $\bZ = Z(P_1, \dots, P_{r})$ for a collection of polynomials $P_i \colon \R^n \to \R$ for $1 \leq i \leq r$ where\footnote{Note that here, in contrast with much of the algebraic geometry literature, the ideal generated by the $P_i$ is not required to be irreducible.}
\begin{equation}\label{variety}
    Z(P_1, \dots, P_{r}) := \{\bx \in \R^n : P_1(\bx) = \cdots = P_{r}(\bx) = 0\}.
\end{equation}
\end{definition}

For the case of interest (namely, where $\bZ$ is a \textit{transverse complete intersection}: see Definition \ref{transverse complete intersection definition} below), $\bZ$ will always be a real smooth submanifold of $\R^n$. Here the \textit{dimension} $\dim \bZ$ is defined to be the dimension of $\bZ$ as a real smooth manifold. The results of this section hold for more general varieties which potentially admit singular points, with a suitably generalised definition of dimension, although we will not discuss the details of this definition here (see, for instance, \cite{BCR}).

\begin{definition}
Given a variety $\bZ$ the \emph{degree} of $\bZ$ is
\begin{equation*}
\deg\,\bZ := \inf \sum_{j=1}^{r} \deg P_j,
\end{equation*}
where the infimum is taken over all possible representations of $\bZ$ of the form \eqref{variety}. 
\end{definition}

The proof of Theorem~\ref{mainThm} will involve the analysis of a more general class of sets. 

\begin{definition} A set $S \subset \R^n$ is \emph{semialgebraic} if there exists a finite collection of polynomials $P_{i,j}$, $Q_{i,j} \colon \R^n \to \R$ for $1 \leq i \leq r$, $1 \leq j \leq s$ such that
\begin{equation}\label{semialgebraic set}
    S = \bigcup_{i=1}^r \big\{\bx \in \R^n : P_{i,1}(\bx) = \cdots = P_{i,s}(\bx) = 0,\, Q_{i,1}(\bx) > 0, \dots, Q_{i,s}(\bx) > 0\big\}.
\end{equation}
\end{definition}

\begin{definition} Given a semialgebraic set $S \subset \R^n$ the \textit{complexity} of $S$ is 
\begin{equation*}
    \inf\Big( \sum_{i,j} \deg P_{i,j} + \deg Q_{i,j} \Big)
\end{equation*}
where the infimum is taken over all possible representations of $S$ of the form \eqref{semialgebraic set}.
\end{definition}
 
 A number of fundamental results in the theory of semialgebraic sets will be used in the proof of Theorem~\ref{mainThm}, including the Tarski--Seidenberg projection theorem and Gromov's algebraic lemma. For the reader's convenience, the relevant statements are recorded in the appendix.




 \subsection{A Wongkew-type lemma}\label{Wongkew section} The main new ingredient in the proof of Theorem~\ref{mainThm} will be a bound for the Lebesgue measure of certain semialgebraic sets $S_{m}(I_m,\rho)$ given by unions of line segments. Before defining these sets some basic reductions are made and some useful notion is introduced.
 
 We choose our coordinates in such a way that the $\lambda_{m}$-ball $B_{\lambda_m}$ is centred at the origin and a reasonably large  proportion of our direction-separated $\delta$-tubes have core lines which can be  parametrised by
 \begin{equation*}
   l_{\mathbf{a},\mathbf{d}}(t) := (\mathbf{a},0)+t(\mathbf{d},1), \qquad t\in\mathbb{R},  
 \end{equation*}
   for some $\mathbf{a},\mathbf{d}\in[-1,1]^{n-1}$. Then, for each $j=k,\ldots,m$, we partition the orthogonal projection of $B_{\lambda_j}$ onto the $t$-axis into $4nd$ disjoint intervals $I_j\subset [-1,1]$ of length~$\lambda_j/(2nd)$, where $d$ bounds the degree of our varieties $\bZ_k,\ldots,\bZ_m$.

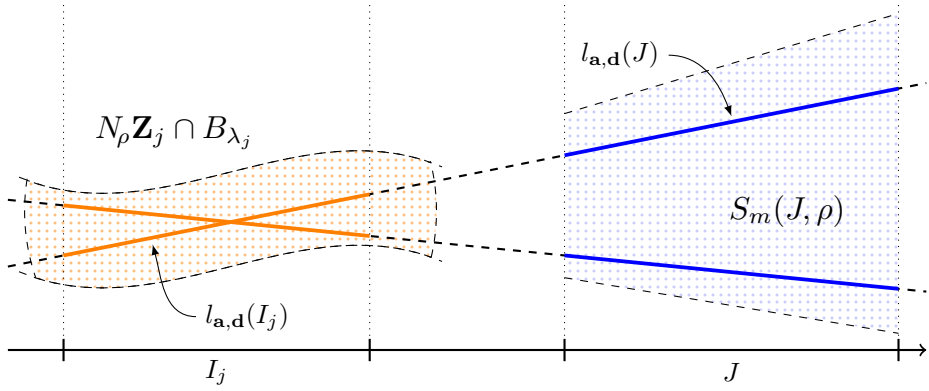
\begin{figure}
    \centering
 
 \begin{tikzpicture}[scale=0.74]

\clip (-4.2,-3.2) rectangle + (17,7);

\draw [black, dashed, domain=-4:4,name path=upper_A] plot (\x, {\x/4-\x/4*\x/4*\x/4 + 0.7}); 
\draw [black, dashed, domain=-4:4,name path=lower_A] plot (\x, {\x/4-\x/4*\x/4*\x/4 -1});

\tikzfillbetween[
    of=upper_A and  lower_A, on layer=bg] {pattern=dots, pattern color=orange!50};

\fill[fill=white] (3.7cm,0) arc [white,radius=3.7cm, start angle=0, delta angle=180]
                  -- (-4.2cm,0) arc [white, radius=4.2cm, start angle=178, delta angle=-180]
                  -- cycle;
\fill[fill=white] (3.7cm,0) arc [white,radius=3.7cm, start angle=0, delta angle=-180]
                  -- (-4.2cm,0) arc [white, radius=4.2cm, start angle=180, delta angle=180]
                  -- cycle;

\draw [black, dashed, domain=-13:14] plot ({3.7*cos(\x)}, {3.7*sin(\x)});
\draw [black, dashed, domain=171:199] plot ({3.7*cos(\x)}, {3.7*sin(\x)});

\draw [black, dashed, domain=-3.8:3.8] plot (\x, {\x/4-\x/4*\x/4*\x/4 + 0.7}); 
\draw [black, dashed, domain=-3.8:3.8] plot (\x, {\x/4-\x/4*\x/4*\x/4 -1}); 

\draw [black, thick, dashed, domain=-4:12.5] plot (\x, {\x/5 - 1/5});

\draw [orange, line width=0.5mm, domain=-3:2.5] plot (\x, {\x/5 - 1/5});
\draw [blue, line width=0.5mm, domain=6:12] plot (\x, {\x/5 - 1/5});

\draw [black, thick, dashed, domain=-4:12.5] plot (\x, {-\x/10 - 1/5});

\draw [orange, line width=0.5mm, domain=-3:2.5] plot (\x, {-\x/10 - 1/5});
\draw [blue, line width=0.5mm, domain=6:12] plot (\x, {-\x/10 - 1/5});

\draw[dotted] (-3, -2.7) -- (-3, 4);
\draw[dotted] (2.5, -2.7) -- (2.5, 4);
\draw[dotted] (6, -2.7) -- (6, 4);
\draw[dotted] (12, -2.7) -- (12, 4);

\draw[->, thick] (-4, -2.5) -- (12.5, -2.5);
\draw[-, thick] (-3, -2.7) -- (-3, -2.3);
\draw[-, thick] (2.5, -2.7) -- (2.5, -2.3);
\draw[-, thick] (6, -2.7) -- (6, -2.3);
\draw[-, thick] (12, -2.7) -- (12, -2.3);

\node[scale=1.2] at  (-1cm, 1.4cm) {$N_{\!\rho}\mathbf{Z}_j \cap B_{\lambda_j}$};

\node (I line) at  (0.3cm, -1.9cm) {$l_{\mathbf{a},\mathbf{d}}(I_j)$};
\draw[-latex] (I line) to[out=180,in=-75] (-1.4cm,-0.6cm);

\node at  (9cm, -2.9cm) {$J$};
\node at  (-0.25cm, -2.9cm) {$I_j$};

\draw [black, dashed, domain=6:12, name path=lower_S] plot (\x, {-\x/6 - 1/5});
\draw [black, dashed, domain=6:12, name path=upper_S] plot (\x, {0.3*\x - 0.05});
\tikzfillbetween[
    of=upper_S and  lower_S, on layer=bg] {pattern=dots, pattern color=blue!20};

\node[scale=1.2] at  (10cm, 0cm) {$S_{m}(J, \rho)$};
\node (J line) at  (7cm, 2.8cm) {$l_{\mathbf{a},\mathbf{d}}(J)$};
\draw[-latex] (J line) to[out=0,in=110] (9cm,1.7cm);
\end{tikzpicture}

    \caption{The set $S_m(J, \rho)$ is formed by a union of line segments $l_{\mathbf{a},\mathbf{d}}(J)$ which have the property that $l_{\mathbf{a},\mathbf{d}}(I_j) \subseteq N_{\!\rho}\mathbf{Z}_j \cap B_{\lambda_j}$ for $k \leq j \leq m$. }
    \label{set definition figure}
\end{figure}

Given any interval $J \subseteq \R$, we define
 \begin{equation*}
S_{m}(J,\rho):=\bigcap_{j=k}^m \big\{\, l_{\mathbf{a},\mathbf{d}}(t)  :    t \in J,\: (\mathbf{a},\mathbf{d})\in [-1,1]^{2(n-1)},\: l_{\mathbf{a},\mathbf{d}}(I_j)\subseteq N_{\!\rho} \bZ_j\cap B_{\lambda_j}\,\big\};  
\end{equation*}
see Figure~\ref{set definition figure} for a diagrammatic description of this set. The key problem will be to estimate the measure of these sets. Note that the measure of $S_{m}(I_m,\rho)$ depends on the specific choice of $I_k,\ldots,I_m${\r;} however{\r,} our bounds will be uniform over any choice and so we suppress this dependence in the notation. 
An example of such a bound follows from the $m$-dimensional version of Wongkew's theorem \cite{Wongkew1993} (see Theorem~\ref{Wongkew theorem} in the appendix), which immediately implies that \begin{equation}\label{wonk}
|S_m(I_m,\rho)| \le  |N_{\!\rho} \bZ_m \cap B_{\lambda_m}| \lesssim_d \lambda^m_m\rho^{n-m}.
\end{equation}
This estimate only uses the $m$-dimensional information, and our first task is to improve this bound using the additional lower dimensional information. 

In order to improve \eqref{wonk}, we will consider both $S_{\ell}(I_{\ell},\rho)$ and  $S_{\ell}(I_{\ell+1},\rho)$, the latter of which need not be contained in either $N_{\!\rho} \bZ_{\ell}\cap B_{\lambda_{\ell}}$ or $N_{\!\rho} \bZ_{\ell+1}\cap B_{\lambda_{\ell+1}}$. Roughly speaking, there are two steps to the argument:
\begin{itemize}
    \item[\textit{Step 1}:] We bound $|S_{\ell+1}(I_{\ell+1},\rho)|$ in terms of $|S_{\ell}(I_{\ell+1},2\rho)|$ using trigonometry and Wongkew's theorem \cite{Wongkew1993}.
    \item[\textit{Step 2}:] We bound $|S_{\ell}(I_{\ell+1},2\rho)|$ in terms of $|S_{\ell}(I_{\ell},4\rho)|$ using an algebraic argument that borrows ideas from \cite{KR2018}.  
\end{itemize}  
Iterating these steps yields a bound for $|S_m(I_m,\rho)|$ in terms of $|S_k(I_k,4^{m-k}\rho)|${\r,} at which point we can use the $k$-dimensional version of Wongkew's theorem rather than the $m$-dimensional version. 
The resulting bound is presented in the following lemma.

\begin{lemma}\label{mainlem} 
For all $n\ge m\ge k\ge 1$, $d\ge 1$ and $\varepsilon>0$, 
$$
\left|S_m(I_m,\rho)\right|\lesssim_d \rho^{-\varepsilon} \Big(\prod_{j=k}^{m-1} \frac{\rho}{\lambda_j}\Big) \lambda_m^{m}\rho^{n-m}$$
whenever $0<\rho/4\le \lambda_{k}\le\ldots \le\lambda_m\le1$, the $j$-dimensional varieties $\bZ_j\subset \mathbb{R}^n$ have degree $\le d$ and $B_{\lambda_{k}}\subset \ldots \subset B_{\lambda_{m}}\subset \mathbb{R}^{n}$.
\end{lemma}

 Taking the $j$-dimensional varieties~$\bZ_j$  to be nested $j$-planes reveals that the estimate is sharp up to the factor of $C_{ n,d,\varepsilon}\rho^{-\varepsilon}$.

 \begin{proof}[Proof (of Lemma \ref{mainlem})] The proof is somewhat involved and is broken into stages.
 
 \subsubsection*{Initial reductions} We may assume without loss of generality that 
\begin{equation}\label{non-degenerate rho}
    \rho\le 4^{k-m+1}\lambda_k. 
\end{equation}
Indeed, otherwise there exists a largest $k'$ such that $k + 1 \leq k' \leq m + 1$ and $\rho > 4^{k-m+1}\lambda_j$ for all $k \leq j \leq k' - 1$. If $k' = m+1$, then the result is trivial. If $k' < m+1$, then we may drop the $j$th condition in $S_m(I_m,\rho)$ for $k \leq j \leq k' - 1$. Relabelling $k'$ as $k$, \eqref{non-degenerate rho} now holds.

It will also be useful to assume that the intervals $I_j$ have lengths given by some dyadic number: that is,
\begin{equation}\label{dyadic intervals}
\frac{\lambda_j}{2nd} \in 2^{\Z}.
\end{equation}
This is possible by slightly enlarging the set by appropriately rounding up the $\lambda_j$'s.

 \subsubsection*{Setting up the induction} For all $k \leq \ell \leq m$ we will prove that
\begin{equation}\label{induction hypothesis}
\left|S_{\ell}(I_{\ell},\rho)\right|\lesssim_d \rho^{-\varepsilon}\Big(\prod_{j = k}^{\ell -1}\frac{\rho}{\lambda_j}\Big)\lambda_{\ell}^{\ell} \rho^{n - \ell}
\end{equation}
whenever $\rho\le 4^{k-\ell+1}\lambda_k$ and the $\lambda_j$ satisfy \eqref{dyadic intervals}. To do this, we induct on $\ell$. For technical reasons, it will be useful to slightly enlarge the sets by redefining 
\begin{equation*}
 S_{\ell}(J,\rho):=\bigcap_{j=k}^{\ell} \big\{\, l_{\mathbf{a},\mathbf{d}}(t)  :    t \in J,\: (\mathbf{a},\mathbf{d})\in Q^{2(n-1)}(\rho),\: l_{\mathbf{a},\mathbf{d}}(I_j)\subseteq N_{\!\rho} \bZ_j\cap B_{\lambda_j+\rho}\,\big\}
\end{equation*}
where $Q^{2(n-1)}(\rho) := [-1-\rho,1+\rho]^{2(n-1)}$. Clearly, any bound of the form \eqref{induction hypothesis} for these enlarged sets implies the same bound holds for the original $ S_{\ell}(I_{\ell},\rho)$.

By the $k$-dimensional version of Wongkew's theorem \cite{Wongkew1993} (see Theorem~\ref{Wongkew theorem}),
$$
\left|S_k(I_k,\rho)\right|\lesssim_d \lambda_{k}^{k}\rho^{n-k}
$$
whenever $\rho\le 4\lambda_k$ and this serves as the base case for the induction argument.

Assuming \eqref{induction hypothesis} holds for some $k \leq \ell \leq m-1$, it suffices to prove that
\begin{equation*}
\left|S_{\ell+1}(I_{\ell+1},\rho)\right|\lesssim_d \rho^{-\varepsilon}\Big(\frac{\lambda_{\ell+1}}{\lambda_{\ell}}\Big)^{{\ell+1}}\left|S_{\ell}(I_{\ell},4\rho)\right|
\end{equation*}
whenever $\rho\le 4^{k-\ell}\lambda_k$ and $\lambda_j/(2nd) \in 2^{\Z}$. We may also assume the non-degeneracy hypothesis that  
\begin{equation}\label{non-degenerate S}
|S_{\ell+1}(I_{\ell+1},\rho)|\ge 8 \Big(\prod_{j = k}^{\ell }\frac{\rho}{\lambda_j}\Big)\lambda_{\ell+1}^{\ell+1} \rho^{n - \ell-1}\ge  8\lambda_{\ell+1}\rho^{n-1},
\end{equation}
as otherwise the induction step would have closed already.

\subsubsection*{Dyadic decomposition} Recall from our initial reductions that the $I_j$ are dyadic intervals. To prove the induction step we partition $I_{\ell+1}$ into the part close to $I_\ell$,
\begin{equation}\label{close intervals}
   \{ t\in I_{{\ell+1}}\, :\, \dist(t,I_{\ell})\le |I_{\ell}|\}, 
\end{equation}
 and dyadic parts further from $I_\ell$,
 \begin{equation}\label{far intervals}
 \big\{\, t\in I_{{\ell+1}}\, :\, 2^{i}|I_\ell|\le\dist(t,I_{\ell})\le 2^{i+1}|I_{\ell}|\,\big\},\quad i\ge 0.     
 \end{equation}
Let $\mathcal{J}$ denote the collection of all maximal dyadic subintervals of the sets in \eqref{close intervals} or \eqref{far intervals}. We have
 \begin{equation*}
 S_{\ell+1}(I_{\ell+1},\rho)= \bigcup_{J \in \mathcal{J}} S_{\ell+1}(J,\rho) \subset \bigcup_{J \in \mathcal{J}} S_{\ell}(J,\rho)\cap N_{\!\rho} \bZ_{{\ell+1}},
\end{equation*}
 where the final inclusion follows directly from the definitions. Since the $J \in \mathcal{J}$ are contained in $I_{\ell+1}$ and are pairwise disjoint,
 \begin{equation*}
     |S_{\ell+1}(I_{\ell+1},\rho)| \leq 4\lambda_{\ell+1}\rho^{n-1} +  \sum_{\substack{J \in \mathcal{J} \\ | S_{\ell}(J,\rho)| \geq 4|J|\rho^{n-1} }}  |S_{\ell}(J,\rho)\cap N_{\!\rho} \bZ_{{\ell+1}}|.
 \end{equation*}
By \eqref{non-degenerate S}, the first term on the right-hand side of the above display is at most half the term on the left-hand side. Thus, it suffices to estimate the right-hand sum.
 
 Given that the balls are nested, $B_{\lambda_{k}}\subset \ldots \subset B_{\lambda_{m}}\subset \mathbb{R}^{n}$, we have   \begin{equation*}
 \maxdist(I_\ell,J):=\sup\big\{\, |t-t'|\, :\, t\in I_\ell,\ t'\in J\,\big\}\lesssim \lambda_{\ell+1},
 \end{equation*}
so there are no more than  $2\log(\lambda_{\ell+1}/|I_{\ell}|)\lesssim_d \log (\rho^{-1})$ intervals $J \in \mathcal{J}$. Thus, it will suffice to prove  that
 \begin{equation}\label{wanted}
\left|S_{\ell}(J,\rho)\cap N_{\!\rho} \bZ_{{\ell+1}}\right|\lesssim_d \rho^{-\varepsilon}\Big(\frac{|J|}{|I_{ \ell}|}\Big)^{{\ell+1}}\left|S_{\ell}(I_{ \ell},4\rho)\right|,
\end{equation}
whenever $J \in \mathcal{J}$ satisfies $|S_{\ell}(J,\rho)| \geq 4|J|\rho^{n-1}$.

\subsubsection*{Inductive step: the first bound} We now turn to the precise version of Step 1 from the proof sketch at the beginning of the section.

\begin{lemma}\label{step 1 lemma} If $J \in \mathcal{J}$ satisfies $\dist(I_{\ell}, J) \geq |J|$, then
\begin{equation*}
    |S_{\ell}(J,\rho)\cap N_{\!\rho} \bZ_{{\ell+1}}|\lesssim_d \Big(\frac{|J|}{|I_{ \ell}|}\Big)^{\ell+1-n}|S_{\ell}(J,2\rho)|.
\end{equation*}
\end{lemma}

\begin{proof} We first claim that it is possible to cover  $S_{\ell}(J,\rho)$ by a collection $\mathcal{B}$ of balls of radius $\rho|J|/|I_{ \ell}|$ with cardinality
\begin{equation}\label{cardinality bound}
\# \mathcal{B} \lesssim_d \Big(\frac{|I_{\ell}|}{\rho|J|}\Big)^n|S_{\ell}(J,2\rho)|.
\end{equation}
Temporarily assuming that this is so, one may argue as follows. For each of the balls $B \in \mathcal{B}$ one may apply Wongkew's theorem \cite{Wongkew1993} (see Theorem~\ref{Wongkew theorem}) to deduce that
\begin{equation*}
    |B \cap N_{\!\rho} \bZ_{{\ell+1}}| \lesssim_d \Big(\frac{\rho|J|}{|I_{\ell}|}\Big)^{{\ell+1}}\rho^{n-(\ell+1)}.
\end{equation*}
 Thus, by \eqref{cardinality bound}, altogether we find that 
\begin{align*}
|S_{\ell}(J,\rho)\cap N_{\!\rho} \bZ_{{\ell+1}}| & \le \sum_{B \in \mathcal{B}} |B\cap N_{\!\rho} \bZ_{{\ell+1}}| \\
&\lesssim_d \Big(\frac{\rho|J|}{|I_{ \ell}|}\Big)^{{\ell+1}}\rho^{n-(\ell+1)}\Big(\frac{|I_{ \ell}|}{\rho|J|}\Big)^n|S_{\ell}(J,2\rho)|,
\end{align*}
as desired.

\begin{figure}
\centering

\begin{tikzpicture}[scale=0.74]
 
\clip (-3.5,-3.6) rectangle + (17,6.3);

\draw [black, thick, dashed, domain=-2.5:12.5] plot (\x, {\x/5 - 1/5});

\draw [orange, line width=0.5mm, domain=-2:2.5] plot (\x, {\x/5 - 1/5});
\draw [blue, line width=0.5mm, domain=7:12] plot (\x, {\x/5 - 1/5});

\draw [black, thick, dashed, domain=-2.5:12.5] plot (\x, {-\x/10 - 1/10});

\draw [orange, line width=0.5mm, domain=-2:2.5] plot (\x, {-\x/10 - 1/10});
\draw [blue, line width=0.5mm, domain=7:12] plot (\x, {-\x/10 - 1/10});

\draw[black, dotted] (-2, -2.7) -- (-2, 4);
\draw[black, dotted] (2.5, -2.7) -- (2.5, 4);
\draw[black, dotted] (7, -2.7) -- (7, 4);
\draw[black, dotted] (12, -2.7) -- (12, 4);

\draw[->, black, thick] (-3, -2.5) -- (12.5, -2.5) node [right] {$x_n$};
\draw[-, black, thick] (-2, -2.7) -- (-2, -2.3);
\draw[-, black, thick] (2.5, -2.7) -- (2.5, -2.3);
\draw[-, black, thick] (7, -2.7) -- (7, -2.3);
\draw[-, black, thick] (12, -2.7) -- (12, -2.3);

\node at (-1cm,-1cm) {$l_{\mathbf{a},\mathbf{d}}(I_{\ell})$};
\node at (-1cm,0.5cm) {$l_{\mathbf{\tilde{a}},\mathbf{\tilde{d}}}(I_{\ell})$};

\node at  (9.5cm, -2.9cm) {$J$};
\node at  (0.25cm, -2.9cm) {$I_{\ell}$};

\draw[<->, black, thick] (-2.4,-0.5) -- (-2.4,0) node [scale=1.2, midway, left] {$\lesssim \rho$};
\draw[<->,black, thick] (9.5,-0.9) -- (9.5,1.5) node [scale=1.2, midway, right] {$\lesssim \rho \displaystyle \frac{|J|}{|I_{\ell}|}$};
\draw[<->, black, thick] (2.7,-2.9) -- (6.8,-2.9) node [scale=1.2, midway, below] {$\sim |J|$};

\node[circle,fill=black,inner sep=0pt,minimum size=4pt, label=below:{$\mathbf{y} \in l_{\mathbf{\tilde{a}},\mathbf{\tilde{d}}}(J)$}] at (9.5, -9.5*1/10 - 1/10) {};
\node[circle,fill=black,inner sep=0pt,minimum size=4pt, label=above:{$\mathbf{x} \in l_{\mathbf{a},\mathbf{d}}(J)$}] at (9.5, 9.5*1/5 - 1/5) {};

\node[circle,fill=black,inner sep=0pt,minimum size=4pt, label=below:{$\mathbf{z}$}] at (1/3, 1/15 - 1/5) {};

\end{tikzpicture}
\caption{The trigonometric argument. } 
\label{trigonometric figure}
\end{figure}
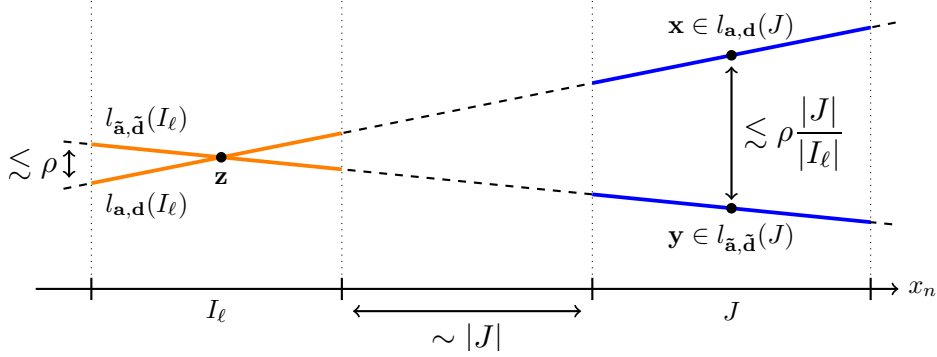

It remains to verify the claim. Letting $r_{\ell} := \rho|J|/(4nd|I_{\ell}|)$, by an elementary covering argument it suffices to show that 
\begin{equation}\label{covering argument}
    N_{r_{\ell}}S_{\ell}(J, \rho) \cap \big( \R^{n-1} \times J \big) \subseteq S_{\ell}(J, 2\rho).
\end{equation}
Fix a point $\by \in N_{r_{\ell}}S_{\ell}(J, \rho) \cap \big( \R^{n-1} \times J \big)$ so there exists some $\bx \in S_{\ell}(J, \rho)$ with $|\bx - \by| < r_{\ell}$. Furthermore, by the definition of $S_{\ell}(J, \rho)$, there exists some $(\mathbf{a}, \mathbf{d}) \in Q^{2(n-1)}(\rho)$ and $t_0 \in J$ such that $\bx = l_{\mathbf{a}, \mathbf{d}}(t_0)$ and $l_{\mathbf{a},\mathbf{d}}(I_j) \subseteq N_{\!\rho} \bZ_j \cap B_{\lambda_j}$.  Let $\bz$ denote the midpoint of the line segment $l_{\mathbf{a},\mathbf{d}}(I_{\ell})$ and $\theta$ the angle $\angle \bx\bz\by$; see Figure~\ref{trigonometric figure}. The separation between $J$ and $I_{\ell}$ implies that $|\bx - \bz|, |\by - \bz| > |J|$ and therefore 
\begin{equation}\label{angle bound}
    |\tan \theta| \leq \frac{r_{\ell}}{|J|} = \frac{1}{4nd} \cdot \frac{\rho}{|I_{\ell}|}.
\end{equation}
The line passing through $\bz$ and $\by$ can be parametrised by $t \mapsto l_{\mathbf{\tilde{a}},\mathbf{\tilde{d}}}(t)$ for some choice of $(\mathbf{\tilde{a}},\mathbf{\tilde{d}}) \in Q^{2(n-1)}(2\rho)$ and $\by = l_{\mathbf{\tilde{a}}, \mathbf{\tilde{d}}}(t_1)$ for some $t_1 \in J$. Moreover, the angle bound \eqref{angle bound} implies that the segment $l_{\mathbf{\tilde{a}},\mathbf{\tilde{d}}}(I_j)$ is contained in a $\rho$-neighbourhood of $l_{\mathbf{a},\mathbf{d}}(I_j)$ for $k \leq j \leq \ell$. Thus, $l_{\mathbf{\tilde{a}},\mathbf{\tilde{d}}}(I_j) \subseteq N_{2\rho}\bZ_j \cap B_{\lambda_j + 2\rho}$ for $k \leq j \leq \ell$ and, consequently,
\begin{equation*}
    \by \in l_{\mathbf{\tilde{a}},\mathbf{\tilde{d}}}(J) \subseteq S_{\ell}(J, 2\rho).
\end{equation*}
This establishes \eqref{covering argument} and concludes the proof.
\end{proof}

\subsubsection*{Inductive step: the second bound} We now turn to the precise version of Step 2 from the proof sketch at the beginning of the section. Loosely speaking, the following lemma tells us that our line segments can never expand at an unexpectedly fast  rate, even after leaving the constricted region.

\begin{lemma}\label{32} If $J \in \mathcal{J}$ satisfies  $|S_{\ell}(J,\rho)|\ge 4|J|\rho^{n-1}$, then
\begin{align*}
|S_{\ell}(J,\rho)| \lesssim_{d}\rho^{-\varepsilon}\Big(\frac{|J|}{|I_{\ell}|}\Big)^{n}|S_{\ell}(I_{\ell},2\rho)|.
\end{align*}
\end{lemma}

To prove Lemma \ref{32}, we will apply the following elementary lemma which states that, although it is not possible to bound  a polynomial at a point in terms of the value that it takes at another point (which could be a root), such a bound holds on average. 

\begin{lemma}\label{elem} Let  $P:\mathbb{R}\to\mathbb{R}$ be a polynomial of degree $m$, $I\subset \mathbb{R}$ be an interval and $t\in \mathbb{R}$. Then $$
|P(t )|\le \Big(8m\frac{\max\{|I|, \dist(t,I)\}}{{|I|}}\Big)^m\frac{1}{|I|}\int_I |P(t')|\, d t'. 
$$
\end{lemma}

The simple proof of this result is postponed until the end of the subsection.

At this point it is also worth recalling that the $\rho$-neighbourhoods~$N_{\!\rho} \bZ_j$ of algebraic  varieties~$\bZ_j=Z(P_1, \dots, P_{n-j})$ are semialgebraic sets. To see this we consider the auxiliary set
$$
Y_j=\Big\{\, (\bx,\by)\in\mathbb{R}^{2n}\ :\  P_1(\bx), \ldots , P_{n-j}(\bx)=0,\ |\by-\bx|< \rho \,\Big\}$$ 
which is clearly semiaglebraic. Then the Tarski--Seidenberg theorem (see Theorem~\ref{Tarski}) tells us that the orthogonal projection $\Pi(Y_j)=N_{\!\rho} \bZ_j$, where $\Pi: (\bx,\by)\mapsto \by$, is also semialgebraic with compexity bounded in terms of $n$ and $d$.

\begin{proof}[Proof (of Lemma \ref{32})]
Consider slices of $S_{\ell}(J,\rho)$ of the form
$$
S_{\ell}(J,\rho)_{t}:= S_{\ell}(J,\rho) \cap \big(\R^{n-1} \times \{t\}\big), \qquad t \in \R,
$$
so that, by Fubini's theorem,
\begin{equation*}
    |S_{\ell}(J,\rho)| \leq 2|J|\rho^{n-1} + \int_{\{ t \in J : |S_{\ell}(J,\rho)_t| \geq 2\rho^{n-1}\}} |S_{\ell}(J,\rho)_t|\,d t.
\end{equation*}
By the hypothesis of the lemma, the first term on the right-hand side is at most half the left-hand term. Therefore, is suffices to prove that  \begin{equation}\label{second}
|S_{\ell}(J,\rho)_{t_{\ell}}|\lesssim_d \rho^{-\varepsilon}\Big(\frac{|J|}{|I_{\ell}|}\Big)^{n-1}\frac{|S_{\ell}(I_{\ell},2\rho)|}{|I_{\ell}|}
\end{equation}
whenever $t_{\ell}\in J$ and 
$|S_{\ell}(J,\rho)_{t_{\ell}}|\ge 2\rho^{n-1}.$ 

In order to prove \eqref{second}, we write $\mathbf{a}'=\mathbf{a} +t_\ell \mathbf{d}$ and $l_{\mathbf{a}',\mathbf{d}}'(t) := (\mathbf{a}' + (t-t_{\ell})\mathbf{d}, t)$ so that $l_{\mathbf{a}',\mathbf{d}}'(t) = l_{\mathbf{a},\mathbf{d}}(t)$ and $S_{\ell}(J,\rho)$ can be rewritten as
$$
\bigcap_{j=k}^\ell\big\{\, l_{\mathbf{a}',\mathbf{d}}'(t) \,:\, t \in J,\, (\mathbf{a}'-t_{\ell}\mathbf{d},\mathbf{d})\in Q^{2(n-1)}(\rho),\, l_{\mathbf{a}',\mathbf{d}}'(I_j) \subseteq N_{\!\rho} \bZ_j\cap B_{\lambda_j+\rho}\,\big\}.
$$
Consider the associated sets of lines 
$L_\ell(\rho,t_\ell)\equiv L_\ell(\rho,t_\ell,I_k,\ldots,I_m)$ defined by
$$
L_\ell(\rho,t_\ell):=\bigcap_{j=k}^\ell\big\{ (\mathbf{a}',\mathbf{d})\,:\, (\mathbf{a}'-t_{\ell}\mathbf{d},\mathbf{d})\in  Q^{2(n-1)}(\rho), \,   l_{\mathbf{a'},\mathbf{d}}'(I_j) \subseteq  N_{\!\rho} \bZ_j\cap B_{\lambda_j  +\rho}\big\}.
$$
From the definitions, 
\begin{equation}\label{unpack definitions}
    (\mathbf{a}',\mathbf{d}) \in L_\ell(\rho,t_\ell) \quad \textrm{if and only if} \quad l_{\mathbf{a}',\mathbf{d}}'(J) \subseteq S_{\ell}(J,\rho)
\end{equation}
and, in particular, if either of these equivalent statements holds, then $\mathbf{a}' \in S_{\ell}(J,\rho)_{t_{\ell}}$.

\begin{figure}
    \centering
 
 \begin{tikzpicture}[scale=0.8]


\draw[dotted] (0, -1) -- (0, 8);

\draw[blue, thick] (0, 1) -- (0, 3);
\draw[blue, thick] (0, 4) -- (0, 7);
\draw[black!70, name path=lower_A] (-2, 0) -- (0,1) -- (1.5, 0.5) -- (2, 1);
\draw[black!70, name path=upper_A] (-2, 7.5) -- (2,6.5);
\filldraw[white] (-1, 3.5) -- (0,3) -- (1, 3.5) -- (0,4) -- (-1, 3.5);
\draw[black!70] (-1, 3.5) -- (0,3) -- (1, 3.5) -- (0,4) -- (-1, 3.5);
\draw[dotted] (0, 3) -- (0, 4);

\tikzfillbetween[
    of=upper_A and  lower_A, on layer=bg] {pattern=dots, pattern color=blue!30};

\foreach \d in {-0.5,-0.4,...,-0.1}
{
\draw [black, domain=-2:2] plot (\x, {1.2 -\d*\x});
\draw [black!50,dashed, domain=2.1:3] plot (\x, {1.2 -\d*\x});
\draw [black!50,dashed, domain=-3:-2.1] plot (\x, {1.2 -\d*\x});
}

\foreach \d in {-0.7,-0.6,...,0.3}
{
\draw [black, domain=-2:2] plot (\x, {1.8 -\d*\x});
\draw [black!50,dashed, domain=2.1:3] plot (\x, {1.8 -\d*\x});
\draw [black!50,dashed, domain=-3:-2.1] plot (\x, {1.8 -\d*\x});
}
\foreach \d in {-0.6,-0.5,...,0.6}
{
\draw [black, domain=-2:2] plot (\x, {5 -\d*\x});
\draw [black!50,dashed, domain=2.1:3] plot (\x, {5 -\d*\x});
\draw [black!50,dashed, domain=-3:-2.1] plot (\x, {5 -\d*\x});
}

\foreach \d in {-0.2,-0.1,...,0.7}
{
\draw [black, domain=-2:2] plot (\x, {6 -\d*\x});
\draw [black!50,dashed, domain=2.1:3] plot (\x, {6 -\d*\x});
\draw [black!50,dashed, domain=-3:-2.1] plot (\x, {6 -\d*\x});
}

\node[circle,fill=blue,inner sep=0pt,minimum size=4pt] at (0,1.2) {};
\node[circle,fill=blue,inner sep=0pt,minimum size=4pt] at (0,1.8) {};
\node[circle,fill=blue,inner sep=0pt,minimum size=4pt] at (0,5) {};
\node[circle,fill=blue,inner sep=0pt,minimum size=4pt] at (0,6) {};

\draw[->, thick] (-3,-1) -- (3, -1);
\draw[-, thick] (0,-0.8) -- (0, -1.2) node [below] {$t_{\ell} \in J$};
\draw[-, thick] (-2,-0.8) -- (-2, -1.2);
\draw[-, thick] (2,-0.8) -- (2, -1.2);
\draw[dotted] (-2, -1) -- (-2, 8);
\draw[dotted] (2, -1) -- (2, 8);


\draw[dotted] (0+8.5, 0) -- (0+8.5, 8);
\draw[blue, thick] (0+8.5, 1) -- (0+8.5, 3);
\draw[blue, thick] (0+8.5, 4) -- (0+8.5, 7);
\draw[black!70, name path=lower_B] (-2+8.5, 0) -- (0+8.5,1) -- (1.5+8.5, 0.5) -- (2+8.5, 1);
\draw[black!70, name path=upper_B] (-2+8.5, 7.5) -- (2+8.5,6.5);
\filldraw[white] (-1+8.5, 3.5) -- (0+8.5,3) -- (1+8.5, 3.5) -- (0+8.5,4) -- (-1+8.5, 3.5);
\draw[black!70] (-1+8.5, 3.5) -- (0+8.5,3) -- (1+8.5, 3.5) -- (0+8.5,4) -- (-1+8.5, 3.5);
\draw[dotted] (0+8.5, 3) -- (0+8.5, 4);

\tikzfillbetween[
    of=upper_B and  lower_B, on layer=bg] {pattern=dots, pattern color=blue!30};

\foreach \d in {-0.4}
{
\draw [black, thick, domain=-2:2] plot (\x+8.5, {1.2 -\d*\x});
\draw [black!50,dashed, domain=2:3] plot (\x+8.5, {1.2 -\d*\x});
\draw [black!50,dashed, domain=-3:-2] plot (\x+8.5, {1.2 -\d*\x});
}

\foreach \d in {-0.3}
{
\draw [black, thick, domain=-2:2] plot (\x+8.5, {1.8 -\d*\x});
\draw [black!50,dashed, domain=2:3] plot (\x+8.5, {1.8 -\d*\x});
\draw [black!50,dashed, domain=-3:-2] plot (\x+8.5, {1.8 -\d*\x});
}
\foreach \d in {0.5}
{
\draw [black, thick, domain=-2:2] plot (\x+8.5, {5 -\d*\x});
\draw [black,dashed, domain=2:3] plot (\x+8.5, {5 -\d*\x}) node [right] {$l_{\mathbf{x}}'$};
\draw [black,dashed, domain=-3:-2] plot (\x+8.5, {5 -\d*\x});
}

\foreach \d in {0.3}
{
\draw [black, thick, domain=-2:2] plot (\x+8.5, {6 -\d*\x});
\draw [black!50,dashed, domain=2:3] plot (\x+8.5, {6 -\d*\x});
\draw [black!50,dashed, domain=-3:-2] plot (\x+8.5, {6 -\d*\x});
}

\node[circle,fill=blue,inner sep=0pt,minimum size=4pt] at (0+8.5,1.2) {};
\node[circle,fill=blue,inner sep=0pt,minimum size=4pt] at (0+8.5,1.8) {};
\node[circle,fill=blue,inner sep=0pt,minimum size=4pt] at (0+8.5,5) {};
\node[circle,fill=blue,inner sep=0pt,minimum size=4pt] at (0+8.5,6) {};

\draw[->, thick] (-3+8.5,-1) -- (3+8.5, -1);
\draw[-, thick] (0+8.5,-0.8) -- (0+8.5, -1.2) node [below] {$t_{\ell} \in J$};
\draw[-, thick] (-2+8.5,-0.8) -- (-2+8.5, -1.2);
\draw[-, thick] (2+8.5,-0.8) -- (2+8.5, -1.2);
\draw[dotted] (-2+8.5, -1) -- (-2+8.5, 8);
\draw[dotted] (2+8.5, -1) -- (2+8.5, 8);

\node (F def) at  (5cm, 4.6cm) {$(F(\mathbf{x}) , t_{\ell})
$};
\draw[-latex] (F def) to[out=335,in=220] (8.4,4.8);

\end{tikzpicture}

    \caption{Forming a semialgebraic section of the lines. Roughly speaking, the slice $S_{\ell}(J, \rho)_{t_{\ell}}$ (shown as a {\color{blue} \textbf{blue}} vertical line above) is parametrised by a polynomial mapping $F \colon \R^{n-1} \to \R^{n-1}$. We can find another polynomial mapping $G \colon \R^{n-1} \to \R^{n-1}$ which ``selects'' a single line through each point $(F(\mathbf{x}),t_{\ell}) \in S_{\ell}(J, \rho)_{t_{\ell}}$. Indeed, the line $l_{\mathbf{x}}' := \{(F(\mathbf{x}) + (t-t_{\ell})G(\mathbf{x}), t) : t \in \R\}$ has this property. } 
\end{figure}
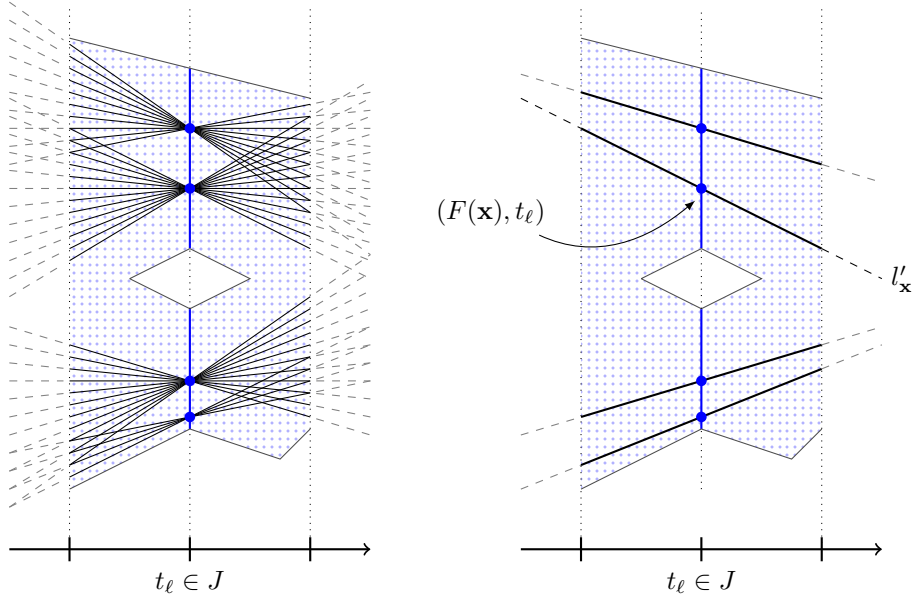

Recall from our earlier discussion that the sets $N_{\!\rho}\bZ_j \cap B_{\lambda_j+ \rho}$ are semialgebraic. By quantifier elimination (that is, the Tarski--Seidenberg theorem),  the sets $L_{\ell}(\rho, t_{\ell})$ are also semialgebraic (see  \cite[Lemma 1.1]{KR2018} for an argument of this type). By an application of Lemma~2.2 of~\cite{KR2018} (see also Corollary~\ref{Tarski corollary} of the appendix), we can take a semialgebraic section of $L_\ell(\rho,t_\ell)$ with complexity bounded by $C(n,d)$, so that there is only one direction~$\mathbf{d}$ for each possible  position~$\mathbf{a}'$ (this is in  contrast with~\cite{KR2018}, where the section was taken to leave only one position for each direction). Calling this section $L_{\ell}'(\rho,t_{\ell})$, we may use Gromov's algebraic lemma (see Lemma \ref{Gromov}), as in \cite[Section 3]{KR2018},  to parametrise $L_{\ell}'(\rho,t_{\ell})$. In particular, taking $s$ to be the first integer larger than $2n^2/\varepsilon$, there exists some $N \in \N$, depending only on the dimension $n$, degree $d$ and $\varepsilon$, and a collection of $C^s$ functions $F^i, G^i:[0,1]^{n-1}\to\mathbb{R}^{n-1}$ for $1 \leq i \leq N$ such that:
\begin{enumerate}[i)]
    \item $\displaystyle \bigcup_{i=1}^N (F^i,G^i)([0,1]^{n-1})=L_{\ell}'(\rho,t_{\ell})$,
    \item $\displaystyle\sup_{|\alpha|\le s}\|\partial^\alpha F^i\|_{\infty},\ \sup_{|\alpha|\le s}\|\partial^\alpha G^i\|_{\infty}\le 1,\quad i=1,\ldots,N$.
\end{enumerate}

Again following \cite[Section 3]{KR2018}, we partition $[0,1]^{n-1}$ into cubes~$Q$ of small diameter $c\rho^{\varepsilon/n}$, with $c$ to be chosen below. On each cube $Q$, we approximate the $C^s$ functions $F^i,G^i:[0,1]^{n-1}\to \mathbb{R}^{n-1}$ by polynomials $F^i_{\!Q},G^i_{\!Q}:\mathbb{R}^{n-1}\to \mathbb{R}^{n-1}$ of degree~$s$ using Taylor's theorem. Indeed, letting $\by_{\!Q}$ denote the centre of $Q$, Taylor's theorem yields polynomials that satisfy
 \begin{equation}\label{Taylor approximation}
 |F^i(\by)-F^i_{\!Q}(\by)|,\ |G^i(\by)-G^i_{\!Q}(\by)|\le  \frac{1}{s!}|\by-\by_{\!Q}|^s\le c^s\!\rho^{2n},\quad \by\in Q.
 \end{equation}
  Using \eqref{unpack definitions} and unpacking all the definitions, 
 \begin{equation*}
     S_{\ell}(J,\rho)_{t_{\ell}} \subseteq \bigcup_{i=1}^N \bigcup_Q F^i(Q).
 \end{equation*}
Furthermore, by \eqref{Taylor approximation}, the boundary of $F^i(Q)$ belongs to the $c^s\!\rho^{2n}$-neighbourhood of the boundary of $F_Q^i(Q)$ and, in particular,
\begin{equation*}
    F^i(Q) \subseteq N_{c^s\!\rho^{2n}} F^i_Q(\partial Q) \cup F^i_Q(Q).
\end{equation*}
The set $F^i_Q(\partial Q)$ is contained in a union of $2^n$ algebraic hypersurfaces so that, by Wongkew's theorem~\cite{Wongkew1993} (see Theorem~\ref{Wongkew theorem}),
\begin{equation*}
 |F^i(Q)| \leq C(n,s)c^s\!\rho^{2n} + |F_Q^i(Q)|.  
\end{equation*}
By taking $c$ sufficiently small, depending only on $n$, $d$ and $\varepsilon$, 
 \begin{equation}\label{Taylor approximation 2}
     |S_{\ell}(J,\rho)_{t_{\ell}}| \leq \rho^{n-1} + \sum_{i=1}^N \sum_{Q}|F_{Q}^i(Q)|
 \end{equation}
 and, by the nondegeneracy hypothesis $|S_{\ell}(J,\rho)_{t_{\ell}}|\ge 2 \rho^{n-1}$, we have
 \begin{equation}\label{sum}
 |S_{\ell}(J,\rho)_{t_{\ell}}|\le 2  \sum_{i=1}^N \sum_{Q}|\mathcal{S}_{Q}^i(J)_{t_{\ell}}|
 \end{equation}
 where 
 \begin{equation*}
\mathcal{S}_{Q}^i(J):=\Big\{\, (F^i_{\!Q}(\by) +(t-t_{\ell}) G^i_{\!Q}(\by), t)\in \mathbb{R}^{n-1}\times J\,:\,  \by\in Q\,\Big\}.
\end{equation*}

On the other hand, we also have that $\mathcal{S}_{Q}^i(I_{\ell})\subseteq S_{\ell}(I_{\ell},2\rho)$. Indeed, fixing $\by \in Q$, it follows from the definition of the $F^i$ and $G^i$, \eqref{unpack definitions} and \eqref{Taylor approximation} that 
\begin{equation*}
    (F^i_{\!Q}(\by) +(t-t_{\ell}) G^i_{\!Q}(\by), t) \in N_{2\rho}\bZ_j \cap B_{\lambda_j+2\rho} \qquad \textrm{for all $t \in I_j$ and $k \leq j \leq \ell$.}
\end{equation*} 
In particular, if $t \in I_{\ell}$ then $(F^i_{\!Q}(\by) +(t-t_{\ell}) G^i_{\!Q}(\by), t) \in S_{\ell}(I_{\ell},2\rho)$. Given that there are fewer than $C(n,d,\varepsilon)\rho^{-\varepsilon}$ summands in \eqref{sum}, it therefore suffices to show
\begin{equation}\label{isit special}
    |\mathcal{S}_{Q}^i(J)_{t_{\ell}}| \lesssim_d \Big(\frac{|J|}{|I_{\ell}|}\Big)^{n-1} \frac{|S^i_Q(I_{\ell})|}{|I_{\ell}|}
\end{equation}
for any fixed choice of $i$ and $Q$. Suppose $F,G \colon \mathbb{R}^{n-1}\to\mathbb{R}^{n-1}$ are polynomials of degree at most $s$ such that  $\det DF$ is not the zero polynomial, where $DF$ denotes the $(n-1)\times(n-1)$ Jacobian matrix of $F$. It thus suffices to prove, more generally, that
 \begin{equation}\label{isit}
|\mathcal{S}(J)_{t_{\ell}}|\le (8(n-1)s)^{n-1}\Big(\frac{\max\{|I|, \maxdist(I,J)\}}{|I|}\Big)^{n-1}\frac{|\mathcal{S}(I)|}{|I|}
\end{equation}
where $I, J \subseteq \R$ are arbitrary intervals, $Q\subset [0,1]^{n-1}$ is any measureable set and  
\begin{equation*}
\mathcal{S}(I):=\Big\{\, (F(\by) +(t-t_{\ell}) G(\by), t)\in \mathbb{R}^{n-1}\times I\,:\,  \by\in Q\,\Big\}.
\end{equation*}
Indeed, it follows from \eqref{Taylor approximation 2} that the polynomials $\det DF_Q^i$ are not zero and for the choice of intervals $I_{\ell}$ and $J$ above we have $\max\{|I_{\ell}|, \mathrm{maxdist}(I_{\ell},J)\} \leq 3|J|$. Hence \eqref{isit special} follows as a special case of \eqref{isit}.

Now, by B\'ezout's theorem, $F+(t-t_{\ell})G$ is at most $s^{n-1}$-to-one on 
$$
Q_t=\big\{\,\by\in Q\,:\,    \det(DF+(t-t_{\ell})DG)(\by)\neq 0\,\big\}.
$$
Furthermore, since, by hypothesis, the polynomial $(\by, t) \mapsto \det(DF + tDG)(\by)$ is non-zero, it follows by Fubini's theorem that $Q\setminus Q_t$ is a Lebesgue null set for almost every $t \in \R$. Consequently,
$$
\frac{1}{s^{n-1}}\int_{I}\int_{Q} | \det(DF+(t-t_{\ell})DG)(\by)|\,d\by dt\le \int_{I}|(F+(t-t_{\ell})G)(Q_t)|\,dt \leq |\mathcal{S}(I)|.
$$
On the other hand, by an application of Lemma~\ref{elem}, we have that
\begin{align*}
&\,|\mathcal{S}(J)_{t_{\ell}}| =|F(Q)|\le \int_{Q} | \det DF(\by)|\,d\by \\
\le&\, \big(8(n-1)\big)^{n-1}\frac{\max\{|I|, \maxdist(I,J)\}^{n-1}}{|I|^{n}}\!\!\int_{Q} \int_{I} |\det (DF+(t-t_{\ell})DG)(\by)|\,d t d\by.
\end{align*}
Combining these displayed inequalities, via an application of Fubini's theorem, yields \eqref{isit} which completes the proof.
\end{proof}

\subsubsection*{Closing the induction} By the initial reductions, to close the inductive step (and thereby finish the proof of Lemma~\ref{mainlem}), it suffices to show \eqref{wanted}. There are two cases to consider:
\begin{itemize}
    \item If $J$ is a subinterval of \eqref{close intervals}, then $|J| = |I_{\ell}|$ and $\mathrm{maxdist}(I_{\ell}, J) \leq 2|I_{\ell}|$. In this case, \eqref{wanted} immediately follows from Lemma~\ref{32}.
    \item If $J$ is a subinterval of one of the sets in \eqref{far intervals}, then $\dist(I_{\ell}, J) = |J|$ and $\maxdist(I_{\ell}, J) \leq 3|J|$.  In this case, \eqref{wanted} follows from a successive application of Lemma~\ref{step 1 lemma} and Lemma~\ref{32}.
\end{itemize}  
This concludes the proof of Lemma~\ref{mainlem}.
\end{proof}

\subsubsection*{The elementary polynomial bound} It remains to prove the elementary Lemma \ref{elem}, which was used in the proof of Lemma \ref{32}.

\begin{proof}[Proof (of Lemma \ref{elem})] By  translating so that $I=[-\lambda,\lambda]$ for some $\lambda>0$, factorising the resulting polynomial, scaling $t\to t/\lambda$ and using the fact that the resulting inequality is symmetric over the origin, this reduces to  proving
$$
|(t -z_1)\cdots(t -z_{m})|\le \big(8m\max\{|t-1|,2\}\big)^m\int_{-1}^{1} |(t'-z_1)\cdots(t'-z_{m})| \,dt'$$
whenever $z_1,\ldots,z_{k}\in\mathbb{C}$.  Supposing that $|z_1|,\ldots,|z_k|\ge 2$ and $|z_{k+1}|,\ldots,|z_m|<2$, as we may, we first note that
\begin{align}\nonumber
\int_{-1}^{1} |(t'-z_1)\cdots(t'-z_{m})| \,dt' &\ge \Big(\frac{1}{2}\Big)^k|z_1|\cdots|z_k|\int_{-1}^{1} |(t'-z_{k+1})\cdots(t'-z_{m})| \,dt'\\
&\ge \Big(\frac{1}{2}\Big)^{k}|z_1|\cdots|z_k|\Big(\frac{{1}}{2(m-k)}\Big)^{m-k},\label{lol}
\end{align}
where the second inequality follows because most values of $t'\in[-1,1]$ must be reasonably far from the roots. Now the small roots, when $j=k+1,\ldots,m$, satisfy
$$
|t -z_{j}|\le |t-1|+|1-z_{j}|\le 4\max\{|t-1|,2\},
$$
and the large roots, when $j=1,\ldots,k$, satisfy
$$
\frac{|z_j|}{|t -z_j|}\ge \frac{|z_j|}{|t-1| +|1-z_j|}\ge \min\Big\{\frac{|z_j|}{2|t-1|},\frac{|z_j|}{2|1-z_j|}\Big\}\ge \frac{{1}}{2\max\{|t-1|,2\}}.
$$
Together we find that
$$
|z_1|\cdots|z_k| \ge \Big(\frac{1}{4\max\{|t-1|,2\}}\Big)^{m}|(t -z_1)\cdots(t -z_{m})|
$$
which can be plugged into \eqref{lol} to complete the proof.
\end{proof}




\subsection{Proof of Theorem~\ref{mainThm}}\label{multiscale subsection} Theorem~\ref{mainThm} now follows by a minor adaptation of the argument from~\cite{KR2018}, applying Lemma~\ref{mainlem} in one key step.

\begin{proof}[Proof (of Theorem~\ref{mainThm})]
Note first that when 
\begin{equation}\label{thisthing}
|T \cap B_{\lambda_j}\cap N_{\!\rho} \bZ_j|\ge \lambda_{j}|T|, 
\end{equation}
 there necessarily exists a line in the direction of $T$ for which  the one-dimensional Lebesgue measure of the line intersected with $B_{\lambda_j}\cap N_{\!\rho} \bZ_j$ is greater than or equal to $\lambda_j$. 
 By B\'ezout's theorem, this line can cross $\bZ_j$ at most $d$ times, so that if $T \cap B_{\lambda_j}\cap N_{\!\rho} \bZ_j$  satisfies \eqref{thisthing}, it must contain a line segment in the direction of $T$ of length $\lambda_{j}/(d+1)$. Fattening this line segment, we obtain a truncated $\delta$-tube contained in $B_{\lambda_j}\cap N_{2\rho} \bZ_j$ that projects onto an interval in the $t$-axis of length~$\ge\lambda_{j}/(nd)$. This interval must contain one of the intervals $I_j$ of length~$\lambda_{j}/(2nd)$ with which we partitioned the orthogonal projection of $B_{\lambda_j}$. Recalling that 
 $$
 L_m(2\rho,0,I_k,\ldots,I_m) := \bigcap_{j=k}^m\Big\{\, (\mathbf{a},\mathbf{d})\in [-1,1]^{2(n-1)}\,:\,   l_{\mathbf{a},\mathbf{d}}(I_j) \subseteq N_{2\rho} \bZ_j\cap B_{\lambda_j}\,\Big\},
$$
 we find that
 \begin{equation*}
\delta^{n-1}\#\bigcap_{j=k}^m\left\{ T \in \mathbb{T}  : |T \cap B_{\lambda_j} \cap N_{\!\rho} \bZ_j|\ge \lambda_j|T|\right\}\,\lesssim\!\!\sum_{I_k,\ldots,I_m}\!\!\big|\Pi\big(L_m(2\rho,0,I_k,\ldots,I_m)\big)\big|,
\end{equation*}
where  $\Pi:(\mathbf{a},\mathbf{d})\mapsto \mathbf{d}$ denotes the orthogonal projection onto the directions.
This is because, for each of the $\delta$-tubes of the original discrete set,  there is a whole $\delta$-ball's worth of different directions contained in one of $\Pi\big(L_{m}(2\rho,0,I_k,\ldots,I_m)\big)$, and these balls finitely overlap due to the fact that $\mathbb{T}$ is direction-separated.

Now by the Tarski--Seidenberg projection theorem, we can take another semialgebraic section of $L_m(2\rho,0,I_k,\ldots,I_m)$, this time leaving only one position $\mathbf{a}$ for each $\mathbf{d}$ as in \cite[Lemma 1.2]{KR2018} (see Corollary~\ref{Tarski corollary}). Following the notation of \cite{KR2018}, we call this section~$L'(I_k,\ldots,I_m)$, and so we also have
\begin{equation}\label{it}
\delta^{n-1}\#\bigcap_{j=k}^m\left\{ T \in \mathbb{T}  : |T \cap B_{\lambda_j} \cap N_{\!\rho} \bZ_j|\ge \lambda_j|T|\right\}\,\lesssim\sum_{I_k,\ldots,I_m}\big|\Pi\big(L'(I_k,\ldots,I_m)\big)\big|.
\end{equation}
Noting that there are no more than $(4nd)^{m-k+1}$ summands in this sum, it remains to bound $|\Pi(L'(I_k,\ldots,I_m))|$ independently of the choice of $I_k,\ldots,I_m$.
For this we use Gromov's algebraic lemma as in the previous section to parametrise $L'(I_k,\ldots,I_m)$ with $C^s$ functions $F^i$ and $G^i$;
$$
\bigcup_{i=1}^N (F^i,G^i)([0,1]^{n-1})=L'(I_k,\ldots,I_m).
$$ 
Then we partition $[0,1]^{n-1}$ into cubes $Q$ again, this time of diameter $c\delta^{\varepsilon/n}$, and approximate the functions $F^i$ and $G^i$ by polynomials $F^i_{\!Q}$ and $G^i_{\!Q}$ of degree $s\le C(n,\varepsilon)$ using Taylor's theorem.  Assuming that $|\Pi(L'(I_k,\ldots,I_m)| \ge \delta^{n-1}$, as we may, these polynomial approximations do not alter the total measure significantly and we find that
\begin{align*}
\big|\Pi\big(L'(I_k,\ldots,I_m)\big)\big|& \le 2\sum_{i=1}^N\sum_{Q} |G^i_{\!Q}(Q)|\le 2\sum_{i=1}^N\sum_{Q} \int_{\!Q}|\det DG^i_{\!Q}(\by)|\, d\by.
\end{align*}
For any fixed $\by \in \R^{n-1}$, provided $\det DG_{\!Q}^i(\by) \neq 0$, the polynomial $ t \mapsto \det(DF^i_{\!Q}+tDG_{\!Q}^i)(\by)$ can be expressed as 
\begin{equation*}
    \det DG_{\!Q}^i (\by)\cdot \prod_{j=1}^{n-1}(t-z_j)
\end{equation*}
for some family of complex roots $z_1, \dots, z_{n-1} \in \C$. There exists a subset of $I_m$ of measure at least $\lambda_m/2$ upon which
\begin{equation*}
    |t - z_j| \geq \frac{\lambda_m}{4(n-1)} \quad \textrm{for $j=1, \dots, n-1$.}
\end{equation*}
On this set, it follows that $|\det(DF^i_{\!Q}+tDG_{\!Q}^i)(\by)| \lesssim \lambda_m^{n-1}|\det DG_{\!Q}^i (\by)|$ and, consequently,
\begin{align}\label{reto}
\big|\Pi\big(L'(I_k,\ldots,I_m)\big)\big|& \lesssim \sum_{i=1}^N\sum_{Q} \lambda_m^{-n}\int_{I_m} \int_{\!Q}|\det(DF^i_{\!Q}+tDG_{\!Q}^i)(\by)|\, d\by d t.
\end{align}

Now by an application of B\'ezout's theorem as in the previous section, the polynomials  $F^i_{\!Q}+tG_{\!Q}^i$ are at most $s^{n-1}$-to-one, so that each of the integrals on the right-hand side of \eqref{reto} can be bounded by $$s^{n-1}\int_{I_m}|(F^i_{\!Q}+tG_{\!Q}^i)(Q)|\,d t\le s^{n-1}|S_m(I_m,4\rho)|.$$ Given that there are fewer than $C(n,d,\varepsilon)\delta^{-\varepsilon}$ summands in \eqref{reto}, this yields
\begin{align*}
\big|\Pi\big(L'(I_k,\ldots,I_m)\big)\big| \lesssim_d \delta^{-\varepsilon}\lambda_m^{-n} |S_m(I_m,4\rho)|.
\end{align*}
Then the proof is completed by combining this with \eqref{it}, bounding $|S_m(I_m,4\rho)|$ by an application of  Lemma~\ref{mainlem}.
\end{proof}





\section{Reduction to $k$-broad estimates}\label{reduction section}

Rather than attempt to prove~\eqref{maximal inequality} directly, it is useful to work with a class of weaker inequalities known as \emph{$k$-broad estimates}. This type of inequality was introduced by Guth~\cite{Guth2016, Guth2018} in the context of oscillatory integral operators (and, in particular, the Fourier restriction conjecture) and was inspired by the earlier multilinear theory developed in~\cite{BCT2006} (see also \cite{Bennett2014} for a detailed discussion of multilinear Kakeya inequalities or Proposition \ref{k-broad vrs k-linear} below for a precise statement relating the $k$-broad and $k$-linear theory). 

In order to introduce the $k$-broad estimates, we decompose the unit sphere $S^{n-1}$ into finitely-overlapping caps $\tau$ of diameter~$\beta$, an admissible constant satisfying $\delta \ll \beta \ll 1$. We then perform a corresponding decomposition of $\T$ by writing the family as a disjoint union of subcollections
\begin{equation*}
    \T = \bigcup_{\tau} \T[\tau]
\end{equation*} 
where each $\T[\tau]$ satisfies $\mathrm{dir}(T) \in \tau$ for all $T \in \T[\tau]$. The ambient euclidean space is also decomposed into tiny balls $B_{\;\!\!\delta}$ of radius~$\delta$. In particular, fix $\mathcal{B}_{\delta}$ a collection of finitely-overlapping $\delta$-balls which cover $\R^n$. For $B_{\;\!\!\delta} \in \mathcal{B}_{\delta}$ define
\begin{equation*}
\mu_{\T}(B_{\;\!\!\delta}) := \min_{V_1,\dots,V_{\!A} \in \mathrm{Gr}(k-1, n)} \Bigg(\max_{\substack{\tau : \angle(\tau, V_a) > \beta \\ \textrm{for } 1 \leq a \leq A}} \Big\|\sum_{T \in \T[\tau]} \chi_T\Big\|_{L^p(B_{\;\!\!\delta})}^p \Bigg),
\end{equation*}
where $A \in \N$ and $\mathrm{Gr}(k-1, n)$ is the Grassmannian manifold of all $(k-1)$-dimensional subspaces in $\R^n$. Here $\angle(\tau, V_a)$ denotes the infimum of the (unsigned) angles $\angle(v,v')$  over all pairs of non-zero vectors $v \in \tau$ and $v' \in V_a $. For $U\subseteq \R^n$ the \emph{$k$-broad norm over $U$} is then defined to be
\begin{equation*}
\Big\|\sum_{T \in \T} \chi_T\Big\|_{\mathrm{BL}^p_{k,A}(U)} := \Bigg(\sum_{\substack{B_{\;\!\!\delta} \in \mathcal{B}_{\delta} }}\frac{|B_{\;\!\!\delta}\cap U|}{|B_{\;\!\!\delta}|} \mu_{\T}(B_{\;\!\!\delta}) \Bigg)^{1/p}.
\end{equation*}
The $k$-broad norms are \emph{not} norms in any familiar sense, but they do satisfy weak analogues of various properties of $L^p$-norms. The basic properties of these objects are described in Section~\ref{k-broad norms section} below.

The main ingredient in the proof of Theorem~\ref{linear theorem} is the following estimate for $k$-broad norms.

\begin{theorem}\label{main theorem} Let   $p\ge 1+\frac{2n}{(n-1)n + (k-1)k}$. For all $\varepsilon>0$,  there is an $A\sim 1$ such that
\begin{equation}\label{broad estimate}\tag{$\BL{k}^p$}
    \Big\|\sum_{T\in \T}\chi_{T}\Big\|_{\mathrm{BL}^p_{k,A}(\R^n)}\,\lesssim\, \delta^{-(n-1 - n/p)-\varepsilon}\Big(\sum_{T \in \T} |T| \Big)^{1/p}\end{equation}
whenever $0<\delta<1$ and $\mathbb{T}$ is a direction-separated family of $\delta$-tubes. 
\end{theorem}

The proof of Theorem~\ref{main theorem}, which is based on the polynomial partitioning method and closely follows the arguments of \cite{Guth2016, Guth2018, HR2018}, will be presented in Sections~\ref{k-broad norms section}--\ref{proof section}. 

The key feature which distinguishes the $k$-broad norm from its $L^p$ counterpart is that the former vanishes whenever the tubes of $\T$ cluster around a $(k-1)$-dimensional set (see Lemma~\ref{vanishing lemma} for a precise statement of this property). Owing to this special behaviour, the inequality~\eqref{broad estimate} is substantially weaker than~\eqref{maximal inequality}. Nevertheless, a mechanism introduced by Bourgain and Guth~\cite{BG2011} allows one to pass from $k$-broad to linear estimates, albeit under a rather stringent condition on the exponent. 

\begin{proposition}[Bourgain--Guth~\cite{BG2011}, Guth~\cite{Guth2018}]\label{Bourgain--Guth} 
Let  
$
 p \ge \frac{n-k+2}{n-k+1}$,  $\varepsilon > 0$ and $A \sim 1$. Suppose that
\begin{equation}\label{broad again}\tag{$\BL{k}^p$}
    \Big\|\sum_{T\in \T}\chi_{T}\Big\|_{\mathrm{BL}^p_{k,A}(\R^n)}\,\lesssim\, \delta^{-(n-1 - n/p)-\varepsilon}\Big(\sum_{T \in \T} |T| \Big)^{1/p}
    \end{equation}
whenever $0<\delta<1$ and $\mathbb{T}$ is a direction-separated family of $\delta$-tubes. Then
 \begin{equation}\label{linear again}\tag{$\mathrm{K}_p$}
\Big\| \sum_{T\in\mathbb{T}} \chi_T\Big\|_{L^p(\mathbb{R}^n)}\,\lesssim\, \delta^{-(n-1 - n/p)-\varepsilon}\Big(\sum_{T \in \T} |T| \Big)^{1/p}
\end{equation}
whenever $0<\delta<1$ and $\mathbb{T}$ is a direction-separated family of $\delta$-tubes.
 \end{proposition}
 
 Thus, combining Theorem~\ref{main theorem} and Proposition~\ref{Bourgain--Guth} yields Theorem~\ref{linear theorem}. In contrast with the range of Lebesgue exponents in Theorem~\ref{main theorem}, the range in which Proposition~\ref{Bourgain--Guth} applies shrinks as $k$ increases. The optimal compromise between the constraints in Theorem~\ref{main theorem} and Proposition~\ref{Bourgain--Guth} is given by \eqref{linear exponent}.

We end this section with a proof of Proposition~\ref{Bourgain--Guth}, which is a minor modification of the argument in~\cite{BG2011} (see also~\cite{Guth2018}).

\begin{proof}[Proof (of Proposition~\ref{Bourgain--Guth})] The proof is by an induction-on-scale argument. 

For the base case, fix $\delta\sim 1$ and let $\T$ be a family of direction-separated $\delta$-tubes. If $\mathcal{B}$ is a cover of $\R^n$ by finitely-overlapping balls of radius 1, then
$$
 \Bigl\|\sum_{T\in \T}\chi_{T}\Big\|_{L^p(\R^n)}^p\,\leq\,     \sum_{B \in \mathcal{B}} \Bigl\|\sum_{\substack{T\in \T\\ T \cap B \neq \emptyset}}\chi_{T}\Big\|_{L^p(B)}^p \,\lesssim\,  \sum_{B \in \mathcal{B}} \#\{T \in \T : T \subset 3B\}^p.
$$
The direction separation condition implies that $\# \T \lesssim 1$ and, consequently, \eqref{linear again} follows from H\"older's inequality and the fact that any tube $T \in \T$ can belong to at most $O(1)$ of the balls $3B$. 

 Now let $\mathbf{C}$ be a fixed constant, chosen sufficiently large so as to satisfy the requirements of the forthcoming argument, and fix some small $\delta>0$.

\vspace{0.5em}
\noindent{\it Induction hypothesis:} Suppose the inequality
\begin{equation*}
    \Bigl\|\sum_{\widetilde{T}\in \widetilde{\T}}\chi_{\widetilde{T}}\Big\|_{L^p(\R^n)}\,\leq\,    \mathbf{C} \tilde{\delta}^{-(n-1-n/p)-\varepsilon}\Big(\sum_{\widetilde{T} \in \widetilde{\T}} |\widetilde{T}|\Big)^{1/p}
\end{equation*}
holds whenever $\tilde{\delta}\in[2\delta ,1)$ and $\widetilde{\T}$ is a direction-separated family of $\tilde{\delta}$-tubes.
\vspace{0.5em}

\noindent

Let $\T$ be a direction-separated family of $\delta$-tubes. Fix a $\delta$-ball $B_{\;\!\!\delta} \in \mathcal{B}_{\delta}$ and subspaces $V_1, \dots, V_{\!A} \in \mathrm{Gr}(n,k-1)$ which obtain the minimum in the definition of $\mu_{\T}(B_{\;\!\!\delta})$; thus
\begin{equation*}
    \mu_{\T}(B_{\;\!\!\delta}) = \max_{\substack{\tau : \angle(\tau, V_a) > \beta \\ \textrm{for } 1 \leq a \leq A}} \Big\|\sum_{T \in \T[\tau]} \chi_T\Big\|_{L^p(B_{\;\!\!\delta})}^p .
\end{equation*}
Since $A \sim 1$ and $\#\{\tau : \angle(\tau, V_a) \leq \beta\} \sim \beta^{-(k-2)}$, by the triangle inequality followed by H\"older's inequality, 
\begin{align*}
    \int_{B_{\;\!\!\delta}}\big|\sum_{T \in \T} \chi_T\big|^p &\,\lesssim\, \int_{B_{\;\!\!\delta}}\big|\sum_{\substack{\tau : \angle(\tau, V_a) > \beta \\ \textrm{for } 1 \leq a \leq A}}\sum_{T \in \T[\tau]} \chi_T\big|^p + \sum_{a = 1}^A \int_{B_{\;\!\!\delta}}\big|\sum_{\tau : \angle(\tau, V_a) \leq \beta}\sum_{T \in \T[\tau]} \chi_T\big|^p \\
    &\,\lesssim\, \beta^{-(n-1)p} \mu_{\T}(B_{\;\!\!\delta}) + \beta^{-(k-2)(p-1)}\sum_{\tau}\int_{B_{\;\!\!\delta}}\big|\sum_{T \in \T[\tau]} \chi_T\big|^p.
\end{align*}
Summing the estimate over all the balls $B_{\;\!\!\delta} \in \mathcal{B}_{\delta}$, we find  that
\begin{equation*}
    \Big\|\sum_{T \in \T} \chi_T\Big\|_{L^p(\R^n)}^p \!\lesssim \beta^{-(n-1)p} \Big\|\sum_{T \in \T} \chi_T\Big\|_{\BL{k,A}^p(\R^n)}^p\! + \beta^{-(k-2)(p-1)}\sum_{\tau}\Big\|\sum_{T \in \T[\tau]} \!\!\chi_T\Big\|_{L^p(\R^n)}^p.
\end{equation*}

The first term on the right-hand side of the above display is estimated using the hypothesised broad estimate. For the second term, we apply a linear rescaling $L \colon \R^n \to \R^n$ so that
\begin{equation}\label{this}
    \Big\|\sum_{T \in \T[\tau]} \chi_T\Big\|^p_{L^p(\R^n)} = \beta^{n-1}\Big\|\sum_{T \in \T[\tau]} \chi_{L(T)}\Big\|^p_{L^p(\R^n)}
\end{equation}
where $\{L(T) : T \in \T[\tau]\}$ is essentially a collection of $\tilde{\delta}$-tubes with $\tilde{\delta} := \beta^{-1}\delta$. To be more precise, let $\omega \in S^{n-1}$ denote the centre of the cap $\tau$ and choose $L$ so that it fixes the 1-dimensional space spanned by $\omega$ and acts as a dilation by a factor of $\beta^{-1}$ on the orthogonal complement $\omega^{\perp}$. Writing $x \in \R^n$ as $x = (x',x_n)$ with $x' \in \omega^{\perp}$, for any $T \in \T[\tau]$ with $v := \mathrm{dir}(T)$ there exists some $u \in \R^n$ such that 
\begin{equation*}
    T \subseteq \big\{ x \in \R^n : |x' - u' - tv'| \lesssim \delta \textrm{ for some $|t| \leq 1$ and } |x_n - u_n| \leq 1/2 \big\},
\end{equation*}
Applying $L$ one obtains
\begin{equation*}
   L(T) \subseteq \big\{ y \in \R^n : |y' - \beta^{-1}u' - t\beta^{-1}v'| \lesssim \beta^{-1}\delta \textrm{ for some $|t| \leq 1$ and } |y_n - u_n| \leq 1/2 \big\}
\end{equation*}
and the right-hand side can be covered by a bounded number of $\tilde{\delta}$-tubes. Furthermore, the family of $\tilde{\delta}$-tubes $L(T)$ is also direction-separated.

Combining~\eqref{this} with the induction hypothesis we find that
\begin{equation*}
 \Big\|\sum_{T \in \T[\tau]} \chi_T\Big\|^p_{L^p(\R^n)}   \,\lesssim\,   \beta^{n-1}\mathbf{C}^p (\beta^{-1}\delta)^{-(n-1)p+n-p\varepsilon}(\beta^{-1}\delta)^{n-1}\#\T[\tau].
\end{equation*}
Recalling  that $\sum_{\tau}\#\T[\tau]=\#\T$, by plugging the preceding estimate into our $L^p(\R^n)$-norm bound,
\begin{equation*}
 \Big\|\sum_{T \in \T} \chi_T\Big\|_{L^p(\R^n)}^p  \leq\,  C \Big(C_{\mathrm{b}}(\beta) + \mathbf{C}^p\beta^{e(p,n,k)+p\varepsilon} \Big) \delta^{-(n-1)p+n-p\varepsilon}\Big(\sum_{T \in \T} |T|\Big); 
\end{equation*}
here $C_{\mathrm{b}}(\beta)$ depends, amongst other things, on the implied constant in \eqref{broad again} whilst $C$ is a constant depending only on $n$ and $p$ (and, in particular, is independent of the choice of $\beta$) and 
\begin{equation*}
    e(p,n,k) := (n -k + 1)p - (n - k + 2).
\end{equation*}

By assumption,  $p\ge \frac{n-k+2}{n-k+1}$ and therefore $e(p,n,k) \ge0$. Consequently, $\beta$ may be chosen sufficiently small, depending only on the admissible parameters $n$, $p$ and~$\varepsilon$, so that
\begin{equation*}
    C \beta^{e(p,n,k)+p\varepsilon}  \,\leq\,  \frac{1}{2}.
\end{equation*}
Moreover, if $\mathbf{C}$ is chosen sufficiently large from the outset, it follows that 
\begin{equation*}
       \Big\|\sum_{T \in \T} \chi_T\Big\|_{L^p(\R^n)}^p \,\leq\,   \mathbf{C}^p \delta^{-(n-1)p+n - p\varepsilon}\Big(\sum_{T \in \T} |T|\Big),
\end{equation*}
which closes the induction and completes the proof. 
\end{proof}




\section{Basic properties of the $k$-broad norms} \label{k-broad norms section}




\subsubsection*{Vanishing property} The proof of Theorem~\ref{main theorem} will involve analysing collections of tubes which enjoy certain tangency properties with respect to algebraic varieties.

\begin{definition}\label{transverse complete intersection definition} Given any collection of polynomials $P_1, \dots, P_{n-m} \colon \R^n \to \R$, recall that the common zero set
\begin{equation*}
    Z(P_1, \dots, P_{n-m}) := \{\bx \in \R^n : P_1(\bx) = \cdots = P_{n-m}(\bx) = 0\}
\end{equation*}
is referred to as a variety. It will often be convenient to work with varieties which satisfy the additional property that
\begin{equation}\label{non singular variety}
\bigwedge_{j=1}^{n-m}\nabla P_j(\bz) \neq 0 \qquad \textrm{for all $\bz \in \bZ = Z(P_1, \dots, P_{n-m})$.}
\end{equation}
In this case the zero set forms a smooth $m$-dimensional submanifold of $\R^n$ with a (classical) tangent space $T_{\bz}\bZ$ at every point $\bz \in \bZ$. A variety $\bZ$ which satisfies~\eqref{non singular variety} is said to be an \emph{$m$-dimensional transverse complete intersection}.
\end{definition}

\begin{definition}\label{tangent definition} Let $0 < \delta < r < 1$, $\bx_0 \in \R^n$ and $\bZ \subseteq \R^n$ be a transverse complete intersection. A $\delta$-tube $T \subset \R^n$ is \emph{tangent to $\bZ$ in $B(\bx_0, r)$} if 
\begin{enumerate}[i)]
\item $T \cap B(\bx_0, r) \cap N_{\;\!\!\delta}\bZ \neq \emptyset$ ;
\item If $\bx \in T$ and $\bz \in \bZ \cap B(\bx_0, 2r)$ satisfy $|\bz-\bx| \leq 8\delta$, then
\begin{equation*}
    \angle(\mathrm{dir}(T), T_{\bz}\bZ) \,\leq\,  c_{\tang} \frac{\delta}{r}. 
\end{equation*}
\end{enumerate}
\end{definition}

Here $0 < c_{\tang}$ is an admissible constant which is chosen small enough to ensure that, whenever i) and ii) hold, 
\begin{equation}\label{inclusion}
T \cap B(\bx_0, 2 r) \subseteq N_{\;\!\! 4 \delta}\bZ.
\end{equation} The fact that such a choice is possible follows from a simple calculus exercise (see, for instance,~\cite[Proposition~9.2]{GHI} for details of an argument of this type).

The raison d'\^etre for the $k$-broad norms is the following lemma, which roughly states that the broad norms vanish if the tubes in $\T$ cluster around a low dimensional variety.

\begin{lemma}[Vanishing property]\label{vanishing lemma} Given $\varepsilon_{\;\!\!\circ} > 0$ and $0<\beta<1$ there exists some $0 < c < 1$ such that the following holds. Let $0 < \delta < c$, $r > \delta^{1- \varepsilon_{\;\!\!\circ}}$, $\bx_0 \in \R^n$ and $\bZ \subseteq \R^n$ be a transverse complete intersection of dimension at most $k-1$. Then
\begin{equation*}
    \Big\| \sum_{T \in \T} \chi_T \Big\|_{\mathrm{BL}^p_{k,A}(B(\bx_0, r))} = 0
\end{equation*}
whenever $\T$ is a family of $\delta$-tubes which are tangent to $\bZ$ in $B(\bx_0, r)$.\footnote{Here the parameter $\beta$ appears implicitly in the definition of the $k$-broad norm.}
 \end{lemma}
 
 \begin{proof} Fix $B_{\;\!\!\delta} \in \mathcal{B}_{\delta}$ with $B_{\;\!\!\delta} \cap B(\bx_0, r) \neq \emptyset$. Recalling the definition of the $k$-broad norm, it suffices to show that there exists some $V \in \mathrm{Gr}(k-1,n)$ such that 
 \begin{equation*}
     \max_{\tau  : \angle(\tau, V) > \beta} \int_{B_{\;\!\!\delta}}  \big| \sum_{T \in \T[\tau]} \chi_T \big|^p = 0.
 \end{equation*}
This would follow if $V$ has the property that
\begin{equation}\label{vanishing 1}
\textrm{if $T \in \T$ satisfies $T \cap B_{\;\!\!\delta} \neq \emptyset$, then $\angle(\mathrm{dir}(T), V) \leq \beta$.}
\end{equation}

Without loss of generality, one may assume there exists some $T_0 \in \T$ such that $T_0 \cap B_{\;\!\!\delta} \neq \emptyset$ (otherwise~\eqref{vanishing 1} vacuously holds for any choice of $(k-1)$-dimensional subspace). By the containment property resulting from the tangency hypothesis,
\begin{equation*}
  T_0 \cap  B_{\;\!\!\delta}  \subseteq T_0 \cap B(\bx_0,2r) \subseteq N_{\;\!\!4 \delta} \bZ
\end{equation*}
and therefore there exists some $\bz_0 \in \bZ$ such that $|\bz_0 - \by_0| < 4 \delta$ for some $\by_0 \in T_0 \cap B_{\;\!\!\delta}$. Let $V$ be a $(k-1)$-dimensional subspace containing $T_{\bz_0} \bZ$. Given any $T \in \T$, if $\bx \in T \cap B_{\;\!\!\delta}$ then $|\bx - \bz_0| <  8 \delta$ and property ii) of the tangency hypothesis implies 
\begin{equation*}
     \angle(\mathrm{dir}(T), V) \,\lesssim\, \frac{\delta}{r} . 
\end{equation*}
Since $r > \delta^{1- \varepsilon_{\;\!\!\circ}}$, it follows that $ \angle(\mathrm{dir}(T), V) \leq \beta$ provided $\delta$ is sufficiently small depending only on $\varepsilon_{\;\!\!\circ}$ and $\beta$, which completes the proof.  
 \end{proof}



\subsubsection*{Triangle and logarithmic convexity inequalities} The $k$-broad norms satisfy weak  variants of certain key properties of $L^p$-norms.

\begin{lemma}[Finite subadditivity] Let $U_1, U_2 \subseteq \R^n$,  $1 \leq p < \infty$ and $A\in \mathbb{N}$. Then
\begin{equation*}
\Big\|\sum_{T \in \T} \chi_T\Big\|_{\mathrm{BL}^p_{k,A}(U_1 \cup U_2)}^p \,\leq\,  \Big\|\sum_{T \in \T} \chi_T\Big\|_{\mathrm{BL}^{p}_{k,A}(U_1)}^p + \Big\|\sum_{T \in \T} \chi_T\Big\|_{\mathrm{BL}^{p}_{k,A}(U_2)}^p
\end{equation*} 
whenever $\T$ is a family of $\delta$-tubes.
\end{lemma}

\begin{lemma}[Triangle inequality]\label{triangle inequality lemma} Let $U \subseteq \R^n$, $1 \leq p < \infty$ and $A\in\mathbb{N}$. Then
\begin{equation*}
\Big\|\sum_{T \in \T_1 \cup \T_2} \chi_T \Big\|_{\mathrm{BL}^p_{k,2A}(U)} \,\lesssim\, \Big\|\sum_{T \in \T_1} \chi_T\Big\|_{\mathrm{BL}^p_{k,A}(U)} + \Big\|\sum_{T \in \T_2} \chi_T\Big\|_{\mathrm{BL}^p_{k,A}(U)}
\end{equation*} 
whenever $\T_1$ and $\T_2$ are families of $\delta$-tubes.
\end{lemma}

\begin{lemma}[Logarithmic convexity]\label{logarithmic convexity inequality lemma} Let $U \subseteq \R^n$, $1 \leq p, p_0, p_1 < \infty$ and $A\in\mathbb{N}$. Suppose that $\theta\in[0,1]$ satisfies
\begin{equation*}
\frac{1}{p} = \frac{1-\theta}{p_0} + \frac{\theta}{p_1}.
\end{equation*}
Then
\begin{equation*}
\Big\|\sum_{T \in \T} \chi_T\Big\|_{\mathrm{BL}^p_{k,2A}(U)} \,\lesssim\, \Big\|\sum_{T \in \T} \chi_T\Big\|_{\mathrm{BL}^{p_0}_{k,A}(U)}^{1-\theta} \Big\|\sum_{T \in \T} \chi_T\Big\|_{\mathrm{BL}^{p_1}_{k,A}(U)}^{\theta}
\end{equation*} 
whenever $\T$ is a family of $\delta$-tubes.
\end{lemma}

These estimates are entirely elementary. The proofs are identical to those used to analyse broad norms in the context of the Fourier restriction problem~\cite{Guth2018}. It is remarked that the parameter $A$ appears in the definition of the $k$-broad norm to allow for these weak triangle and logarithmic convexity inequalities.

\subsubsection*{$k$-broad versus $k$-linear estimates}

Although not required for the proof of Theorem~\ref{linear theorem}, it is perhaps instructive to note the relationship between the $k$-broad norms and the multilinear expressions appearing in the work of Bennett--Carbery--Tao~\cite{BCT2006}.

\begin{proposition}\label{k-broad vrs k-linear} Let $\T$ be a collection of $\delta$-tubes in $\R^n$. Then
\begin{equation*}
\Big\|\sum_{T\in \T}\chi_{T}\Big\|_{\mathrm{BL}^p_{k,A}(\R^n)}\,\lesssim\, \Bigg(\sum_{\substack{(\tau_1, \dots, \tau_k) \\ \sim \,\beta^{k-1}\mathrm{\!\!-trans.}}}\Bigl\|\prod_{j=1}^{k}\Big(\sum_{T_{j}\in\T[\tau_{j}]}\chi_{N_{\;\!\!2\delta}T_{j}}\Big)^{1/k}\Bigr\|_{L^{p}(\R^n)}^p\Bigg)^{1/p}
\end{equation*}
where the sum is over all $k$-tuples $(\tau_1, \dots, \tau_k)$ of caps of diameter $\beta$ which are $\sim \beta^{k-1}$-transversal in the sense that $|\bigwedge_{j=1}^{k} \omega_j | \gtrsim \beta^{k-1}$ for all $\omega_j \in \tau_j$.
\end{proposition}

Thus, any $k$-linear inequality of the type featured in~\cite{BCT2006, Guth2010, BG2011} is stronger than the corresponding $k$-broad estimate (given that $\beta$ is admissible).

The proof of Proposition~\ref{k-broad vrs k-linear} is a simple exercise and is omitted (see~\cite{GHI} for similar results in the (more complicated) context of oscillatory integral operators). 


%




\section{Polynomial partitioning}\label{polynomial partitioning section}

In this section the algebraic and topological ingredients for the proof of Theorem~\ref{main theorem} are reviewed. In particular, the key polynomial partitioning theorem is recalled, which is adapted from~\cite{Guth2016, Guth2018} (see also~\cite{Wang}) 
and previously appeared explicitly in~\cite{HR2018}. 

 Given a polynomial $P \colon \R^n \to \R$ consider the collection  $\cell(P)$  
 of connected components of $\R^n \setminus Z(P)$. Each $O' \in \cell(P)$ is referred to as a \emph{cell} cut out by the variety $Z(P)$ and the cells are thought of as partitioning the ambient euclidean space into a finite collection of disjoint regions. 

In order to account for the choice of scale $\delta > 0$ appearing in the definition of the $\delta$-tubes, it will be useful to consider the family of \emph{$\delta$-shrunken cells} defined by
\begin{equation}\label{shrunken cells}
    \O := \big\{ O'\setminus N_{\;\!\!\delta}Z(P) : O' \in \cell(P) \big\}.
\end{equation}
An important consequence of this definition is the following simple observation:
\begin{quote}
    A $\delta$-tube $T$ can enter at most $\deg P + 1$ of the shrunken cells~$O \in \O$.
\end{quote}
Indeed, this is a simple and direct consequence of the fundamental theorem of algebra (or B\'ezout's theorem) applied to the core line of $T$.

\begin{theorem}[Guth~\cite{Guth2018}]\label{partitioning theorem} Fix $0 < \delta < r$, $x_0 \in \R^n$ and suppose $F \in L^1(\R^n)$ is non-negative and supported on $B(\bx_0,r) \cap N_{\;\!\! 4\delta}\bZ$ where $\bZ$ is an $m$-dimensional transverse complete intersection with $\deg\bZ \leq d$. At least one of the following cases holds:\\

\paragraph{\underline{Cellular case}} There exists a polynomial $P \colon \R^n \to \R$ of degree $O(d)$ with the following properties:
\begin{enumerate}[i)]
    \item $\#\cell(P) \sim d^m$ and each $O \in \cell(P)$ has diameter at most $r/2$.
    \item One may pass to a refinement of $\cell(P)$ such that if $\O$ is defined as in~\eqref{shrunken cells}, then 
    \begin{equation*}
        \int_{O} F \sim d^{-m}\int_{\R^n} F \qquad \textrm{for all $O \in \O$.}
    \end{equation*}
\end{enumerate}  
\paragraph{\underline{Algebraic case}} There exists an $(m-1)$-dimensional  transverse complete intersection $\bY$ of  degree at most $O(d)$ such that
    \begin{equation*}
        \int_{B(\bx_0,r) \cap N_{\;\!\! 4\delta}\bZ} F \,\lesssim\,\log d \int_{B(\bx_0,r) \cap N_{\;\!\!\delta}\bY} F.
    \end{equation*}
\end{theorem}

This theorem is based on an earlier discrete partitioning result which played a central role in the resolution of the Erd\H{o}s distance conjecture~\cite{GK2015}. The proof is essentially topological, involving the polynomial ham sandwich theorem of Stone--Tukey~\cite{Stone1942}, which is itself a consequence of the Borsuk--Ulam theorem (see, for instance,~\cite{Matousek}), combined with a pigeonholing argument.

The theorem  is applied to $k$-broad norms by taking\begin{equation*}
   F= \sum_{B_{\;\!\!\delta} \in \cB_{\;\!\!\delta}} \mu_{\T}(B_{\;\!\!\delta}) \frac{1}{|B_{\;\!\!\delta}|} \chi_{B_{\;\!\!\delta}  }.
\end{equation*}
 \begin{itemize} 
 \item If the cellular case holds, then it follows that
 \begin{equation*}
    \Big\|\sum_{T\in \T}\chi_{T}\Big\|^p_{\mathrm{BL}^p_{k,A}(B(\bx_0,r) \cap N_{\;\!\!4\delta}\bZ)} \,\lesssim\, d^{m}\Big\|\sum_{T\in \T}\chi_{T}\Big\|_{\BL{k,A}^p(O)}^{p} \ \textrm{for all $O \in \O$}
 \end{equation*}
 where $\O$ is the collection of cells produced by Theorem~\ref{partitioning theorem}.
 \item If the algebraic case holds, then it follows that
 \begin{equation*}
    \Big\|\sum_{T\in \T}\chi_{T}\Big\|^p_{\mathrm{BL}^p_{k,A}(B(\bx_0,r) \cap N_{\;\!\!4\delta}\bZ)} \,\lesssim\, 
     \log d\,\Big\|\sum_{T\in \T}\chi_{T}\Big\|^p_{\mathrm{BL}^p_{k,A}(B(\bx_0,r) \cap N_{\;\!\!\delta}\bY)}
\end{equation*}
where $\bY$ is the variety produced by Theorem~\ref{partitioning theorem}.
\end{itemize}




\section{Finding polynomial structure}\label{structure lemma section}

In this section, the recursive argument used to study the Fourier restriction problem in~\cite{HR2018} (which, in turn, is adapted from~\cite{Guth2018}) is reformulated so as to apply to the Kakeya problem. As in~\cite{HR2018}, the argument will be presented as two separate algorithms:
\begin{itemize}
    \item \texttt{[alg 1]} effects a dimensional reduction, essentially passing from an $m$-dimensional to an $(m-1)$-dimensional situation. 
    \item  \texttt{[alg 2]} consists of repeated application of the first algorithm to reduce to a minimal dimensional case. 
\end{itemize}

The final outcome is a method of decomposing any given $k$-broad norm into pieces which are either easily controlled or enjoy special algebraic structure. This decomposition applies to \emph{arbitrary} families of $\delta$-tubes. In the following section, we will specialise to the case where the tube family is direction-separated and use this additional information to prove Theorem~\ref{main theorem}. 

\subsection*{The first algorithm}  Throughout this section let $p \geq 1$ and $0< \varepsilon_{\;\!\!\circ} \ll \varepsilon \ll 1$ be fixed.\\

\paragraph{\underline{\texttt{Input}}} \texttt{[alg 1]} will take as its input:
\begin{itemize}
    \item A choice of small scale $0<\delta\ll 1$ and large scale $r_0\in[\delta^{1-\varepsilon_{\;\!\!\circ}},\delta^{\varepsilon_{\;\!\!\circ}}]$.
    \item A transverse complete intersection $\bZ$ of dimension $m\in\{2,\ldots,n\}$.
    \item A family  $\T$ of $\delta$-tubes which are tangent to $\bZ$ on a ball $B_{\;\!\!r_0}$ of radius $r_0$.
    \item A large integer $A \in \N$. 
\end{itemize}

%
%
%
\paragraph{\underline{\texttt{Output}}} \texttt{[alg 1]} will output a finite sequence of sets $(\sE_j)_{j=0}^J$, which are constructed via a recursive process. Each $\sE_j$ is referred to as an \emph{ensemble} and contains all the relevant information coming from the $j$th step of the algorithm. In particular, the ensemble $\sE_j$ consists of:
\begin{itemize}
    \item A word $\fh_j$ of length $j$ in the alphabet $\{\ta, \tc\}$, referred to as a {\it history}. 
    The  $\ta$ is an abbreviation of ``algebraic'' and~$\tc$ ``cellular''. The words $\fh_j$ are recursively defined by successively adjoining a single letter. Each $\fh_j$ records how the cells $O_j \in \O_j$ were constructed via repeated application of the polynomial partitioning theorem. 
    \item A large scale $r_j\in[\delta^{1-\varepsilon_{\;\!\!\circ}},\delta^{\varepsilon_{\;\!\!\circ}}]$. The $r_j$ will in fact be completely determined by the initial scales and the history $\fh_j$. In particular, 
let $\sigma_{k} \colon [0,1] \to [0,1]$ be given by
    \begin{equation*}
        \sigma_{k}(r) := \left\{ \begin{array}{ll}
        \frac{r}{2} &  \textrm {if the $k$th letter of $\fh_j$ is $\tc$} \\[6pt]
        r^{1+\varepsilon_{\;\!\!\circ}} & \textrm{if the $k$th letter of $\fh_j$ is $\ta$}
        \end{array}\right. 
    \end{equation*}
for each $1 \leq k \leq j$. With these definitions, 
\begin{equation*}
r_j := \sigma_{j} \circ \cdots \circ \sigma_{1}(r_0).
\end{equation*}
Note that each $\sigma_{k}$ is a decreasing function and 
\begin{equation}\label{radius bounds}
r_j \leq \delta^{\varepsilon_{\;\!\!\circ}(1+\varepsilon_{\;\!\!\circ})^{\#_{\sta}(j)}}   \quad \textrm{and} \quad r_j \leq 2^{-\#_{\stc}(j)}\delta^{\varepsilon_{\;\!\!\circ}}  
\end{equation}
  where $\#_{\bta}(j)$ and $\#_{\btc}(j)$ denote the number of occurrences of $\ta$ and $\tc$ in the history $\fh_j$, respectively. 
  \item A family of subsets $\O_j$ of $\R^n$ which will be referred to as \emph{cells}. Each cell $O_j \in \O_j$ is contained in $B_{\;\!\!r_0}$ and will have diameter at most $2r_j$. 
\item An assignment of a subfamily $\T[O_j]$ of $\delta$-tubes to each of the cells $O_j$. 
  \item A large integer $d \in \N$ which depends only on $\deg\bZ$ and the admissible parameters $n$, $p$ and $\varepsilon$.
  \end{itemize}

Moreover, the components of the ensemble are defined so as to ensure that, for certain coefficients
\begin{equation*}
   C_{j}(d):= d^{\#_{\stc}(j)\varepsilon_{\;\!\!\circ}} d^{\#_{\sta}(j)(n+\varepsilon_{\;\!\!\circ})} 
\end{equation*}
and $A_j := 2^{-\#_{\sta}(j)}A \in \N$, the following properties hold:\\

\paragraph{\underline{Property I}} The  function $\sum_{T\in \T} \chi_T$ on $B_{\;\!\!r_0}$ can be compared with functions defined over the $\T[O_j]$:
\begin{equation} \tag*{$(\mathrm{I})_j$}
      \Big\|\sum_{T\in \T} \chi_T\Big\|_{\BL{k,A}^p(B_{\;\!\!r_0})}^p \leq \,\,C_{j}(d)\! \sum_{O_{j} \in \O_{j}} \Big\|\sum_{T\in \T[O_j]} \chi_T\Big\|_{\BL{k,A_j}^p(O_j)}^p .
\end{equation}
\\

\paragraph{\underline{Property II}} The tube families $\T[O_j]$ satisfy
\begin{equation}\tag*{$(\mathrm{II})_j$}
    \sum_{O_{j} \in \O_{j}} \# \T[O_j] \,\leq\, C_{j}(d)d^{\#_{\stc}(j)}  \#\T.
\end{equation}
\paragraph{\underline{Property III}} Furthermore, each individual $\T[O_j]$ satisfies
\begin{equation}\tag*{$(\mathrm{III})_j$}
     \#\T[O_{j}] \,\leq\, C_{j}(d)  d^{-\#_{\stc}(j)(m-1)}  \#\T.
\end{equation}

%
%
%

\subsection*{The initial step}

The initial ensemble $\sE_0$ is defined by taking:
\begin{itemize}
    \item $\fh := \emptyset$ to be the empty word;
    \item $r_0$ to be the large scale;
    \item $\O_0$ the collection consisting of the single ball $O_0 := B_{\;\!\!r_0}$;
    \item $\T[O_0] := \T$.
\end{itemize} 
All the desired properties then vacuously hold. 

At this point it is also convenient to fix some large $d \in \N$, to be determined later, which depends only on  $\deg\bZ$ and the admissible parameters $n$, $p$ and $\varepsilon$. 

With these definitions, it is trivial to verify that Properties I, II and III hold.

%
%
%

\subsection*{The recursive step} Assume the ensembles $\sE_0, \dots, \sE_j$ have been constructed for some $j \in \N_0$ and that they all satisfy the desired properties. \\

%
%
%

\paragraph{\underline{\texttt{Stopping conditions}}} The algorithm has two stopping conditions which are labelled \texttt{[tiny]} and \texttt{[tang]}. 
\begin{itemize}
\item[\texttt{Stop:[tiny]}] The algorithm terminates if $r_j \leq \delta^{1-\varepsilon_{\;\!\!\circ}}$.
\end{itemize}

\begin{itemize}
\item[\texttt{Stop:[tang]}]Let $C_{\textrm{\texttt{tang}}}$ and $C_{\alg}$ be fixed constants, chosen large enough to satisfy the forthcoming requirements of the proof. The algorithm terminates if the inequalities
\begin{equation*}
    \sum_{O_j \in \O_j} \Big\|\sum_{T \in \T[O_j]} \chi_T\Big\|_{\BL{k,A_j}^p(O_j)}^p \,\leq\,  C_{\textrm{\texttt{tang}}} \log d  \sum_{S \in \Sc} \Big\|\sum_{T \in \T[S]} \chi_T\Big\|_{\BL{k,A_j/2}^p(B[S])}^p
\end{equation*}
and
\begin{align*}
   \sum_{S \in \Sc}  \#\T[S]  &\,\leq \,C_{\textrm{\texttt{tang}}}\delta^{-n\varepsilon_{\;\!\!\circ}}\!\!\sum_{O_j \in \O_j}  \#\T[O_j]; \\ 
  \nonumber
  \max_{S \in \Sc}  \#\T[S]  &\,\leq \,C_{\textrm{\texttt{tang}}}\max_{O_j \in \O_j}  \#\T[O_j]
\end{align*}
hold for some choice of:
\end{itemize}
\begin{itemize}
    \item $\Sc$ a collection of transverse complete intersections in $\R^n$ all of equal dimension $m-1$ and degree at most $C_{\alg}d$;
    \item An assignment of a subfamily $\T[S]$ of $\T$ and a $\max\{r_j^{1 +\varepsilon_{\;\!\!\circ}}, \delta^{1 -\varepsilon_{\;\!\!\circ}}\}$-ball $B[S]$ to each $S \in \Sc$ with the property that each $T \in \T[S]$ is tangent to $S$ in $B[S]$ in the sense of Definition~\ref{tangent definition}.
\end{itemize}

The stopping condition \texttt{[tang]} can be roughly interpreted as forcing the algorithm to terminate if one can pass to a lower dimensional situation. Indeed, by the inclusion property \eqref{inclusion}, the broad norm over $B[S]$ could instead be taken over a $4\delta$-neighbourhood of $S$.

If either of the above conditions hold, then the stopping time is defined to be $J := j$. Recalling~\eqref{radius bounds}, the stopping condition \texttt{[tiny]} implies that the algorithm must terminate after finitely many steps and, moreover, 
\begin{equation*}
  \#_{\bta}(J) \,\lesssim\, \varepsilon_{\;\!\!\circ}^{-1} \log(\varepsilon_{\;\!\!\circ}^{-1}) \quad \textrm{and} \quad  \#_{\btc}(J) \,\lesssim\, \log \delta^{-1}.
\end{equation*}
Note that there can be relatively few algebraic steps $\#_{\bta}(j)$ but there can many cellular steps $\#_{\btc}(j)$. The first of the above estimates can also be used to show that $C_{j}(d)\lesssim_{d,\varepsilon_{\;\!\!\circ}}d^{\#_{\stc}(j)\varepsilon_{\;\!\!\circ}}$ always holds. Furthermore, by choosing $A \geq 2^{\varepsilon_{\;\!\!\circ}^{-2}}$, say, one may ensure that the $A_j$ defined above are indeed integers.
\\

%
%
%

\paragraph{\underline{\texttt{Recursive step}}}

Suppose that neither stopping condition \texttt{[tiny]} nor \texttt{[tang]} is met. One proceeds to construct the ensemble $\sE_{j+1}$ as follows. 

Given $O_j \in \O_j$, apply the polynomial partitioning theorem with degree $d$ to 
\begin{equation*}
    \Big\|\sum_{T \in \T[O_j]} \chi_T\Big\|_{\BL{k,A_j}^p(O_j\cap N_{\;\!\!4\delta}\bZ)}^p = \Big\|\sum_{T \in \T[O_j]} \chi_T\Big\|_{\BL{k,A_j}^p(O_j)}^p.
\end{equation*}
 For each $O_j \in \O_j$ either the cellular or the algebraic case holds, as defined in Theorem~\ref{partitioning theorem}. Let $\O_{j,\cell}$ denote the subcollection of $\O_j$ consisting of all cells for which the cellular case holds and $\O_{j,\alg} := \O_j \setminus \O_{j,\cell}$. Thus, by $(\mathrm{I})_j$, one may bound $\|\sum_{T \in \T} \chi_T\|_{\BL{k,A}^p(B_{\;\!\!r_0})}^p$ by
\begin{equation*}
    C_{j}(d)  \Big[ \sum_{O_j \in \O_{j,\cell}} \Big\|\sum_{T \in \T[O_j]} \chi_T\Big\|_{\BL{k,A_j}^p(O_j)}^p + \sum_{O_j \in \O_{j,\alg}} \Big\|\sum_{T \in \T[O_j]} \chi_T\Big\|_{\BL{k,A_j}^p(O_j)}^p \Big];
\end{equation*}
the analysis is splits into two cases depending on which term in the above sum dominates.




\subsection*{$\blacktriangleright$ Cellular-dominant case} Suppose that the inequality
\begin{equation*}
     \sum_{O_j \in \O_{j,\alg}} \Big\|\sum_{T \in \T[O_j]} \chi_T\Big\|_{\BL{k,A_j}^p(O_j)}^p \leq \sum_{O_j \in \O_{j,\cell}} \Big\|\sum_{T \in \T[O_j]} \chi_T\Big\|_{\BL{k,A_j}^p(O_j)}^p 
\end{equation*}
holds so that 
\begin{equation}\label{cellular dominant case}
   \Big\|\sum_{T \in \T} \chi_T\Big\|_{\BL{k,A}^p(B_{\;\!\!r_0})}^p\leq \,2C_{j}(d)\!\! \sum_{O_j \in \O_{j,\cell}} \Big\|\sum_{T \in \T[O_j]} \chi_T\Big\|_{\BL{k,A_j}^p(O_j)}^p.
\end{equation}

\paragraph{\underline{Definition of $\sE_{j+1}$}} Define $\fh_{j+1}$ by adjoining the letter $\tc$ to the word $\fh_j$. Thus, it follows from the definitions that 
\begin{equation}\label{cellular word}
    r_{j+1} =  \tfrac{1}{2}r_j, \quad \#_{\btc}(j+1) = \#_{\btc}(j)+1 \quad \textrm{and} \quad \#_{\bta}(j+1) = \#_{\bta}(j).
\end{equation}

The next generation of cells $\O_{j+1}$ arise from the cellular decomposition guaranteed by Theorem~\ref{partitioning theorem}. Fix $O_j \in \O_{j, \cell}$ so that there exists some polynomial $P \colon \R^n \to \R$ of degree $O(d)$ with the following properties:
\begin{enumerate}[i)]
    \item $\#\cell(P) \sim d^m$ and each $O \in \cell(P)$ has diameter at most $2r_{j+1}$. 
    \item One may pass to a refinement of $\cell(P)$ such that if 
    \begin{equation*}
        \O_{j+1}(O_j) := \big\{ O\setminus N_{\;\!\!\delta}Z(P) : O \in \cell(P)\}
    \end{equation*}
    denotes the corresponding collection of $\delta$-shrunken cells, then 
    \begin{equation*}
    \Big\|\sum_{T \in \T[O_j]} \chi_T\Big\|_{\BL{k,A_j}^p(O_j)}^p \,\lesssim\, d^{m}\Big\|\sum_{T \in \T[O_j]} \chi_T\Big\|_{\BL{k,A_j}^p(O_{j+1})}^p
\end{equation*}
for all $O_{j+1} \in \O_{j+1}(O_j)$.
    \end{enumerate}
 Given $O_{j+1} \in \O_{j+1}(O_j)$, define
 \begin{equation*}
 \T[O_{j+1}] := \big\{ T \in \T[O_j] : T \cap O_{j+1} \neq \emptyset \big\}.
\end{equation*}
 Recall that, by the fundamental theorem of algebra (or B\'ezout's theorem), any $\delta$-tube~$T$ can enter at most $O(d)$ cells $O_{j+1} \in \O_{j+1}(O_j)$ and, consequently, 
\begin{equation}\label{cellular 1}
    \sum_{O_{j+1} \in \O_{j+1}(O_j)} \#\T[O_{j+1}] \,\lesssim\, d \cdot \#\T[O_j].
\end{equation}
By the pigeonhole principle, one may pass to a refinement of $\O_{j+1}(O_j)$ such that
 \begin{equation}\label{cellular 2}
    \#\T[O_{j+1}] \,\lesssim\, d^{-(m-1)} \#\T[O_j] \qquad \textrm{for all $O_{j+1} \in \O_{j+1}(O_j)$.}
\end{equation}
 Finally, define 
 \begin{equation*}
 \O_{j+1} := \bigcup_{O_j \in \O_{j, \cell}} \O_{j+1}(O_j).
\end{equation*}
This completes the construction of $\sE_{j+1}$ and 
it remains to check that the new ensemble satisfies the desired properties. In view of this, it is useful to note that
\begin{equation}\label{cellular coefficients}
C_{j}(d) = d^{-\varepsilon_{\;\!\!\circ}} C_{j+1}(d) \quad \textrm{and}\quad A_j = A_{j+1},
\end{equation}
which follows immediately from~\eqref{cellular word} and the definition of the $C_{j}(d)$ and $A_j$. \\

%
%
%
\paragraph{\underline{Property I}} Fix $O_j \in \O_{j,\cell}$ and observe that $\#\O_{j+1}(O_j) \sim d^m$ and
\begin{equation*}
    \Big\|\sum_{T \in \T[O_j]} \chi_T\Big\|_{\BL{k,A_j}^p(O_j)}^p  \,\lesssim\, d^{m}\Big\|\sum_{T \in \T[O_{j+1}]} \chi_T\Big\|_{\BL{k,A_j}^p(O_{j+1})}^p
\end{equation*}
for all $O_{j+1} \in \O_{j+1}(O_j)$. Averaging,
\begin{equation*}
    \Big\|\sum_{T \in \T[O_j]} \chi_T\Big\|_{\BL{k,A_j}^p(O_j)}^p \,\lesssim \sum_{O_{j+1} \in \O_{j+1}(O_j)}  \Big\|\sum_{T \in \T[O_{j+1}]} \chi_T\Big\|_{\BL{k,A_j}^p(O_{j+1})}^p
\end{equation*}
and, recalling~\eqref{cellular dominant case} and~\eqref{cellular coefficients}, one deduces that
\begin{equation*}
    \Big\|\sum_{T \in \T} \chi_T\Big\|_{\BL{k,A}^p(B_{\;\!\!r_0})}^p \leq \,\,C d^{-\varepsilon_{\;\!\!\circ}} C_{j+1}(d)\!\! \sum_{O_{j+1} \in \O_{j+1}}  \Big\|\sum_{T \in \T[O_{j+1}]} \chi_T\Big\|_{\BL{k,A_{j+1}}^p(O_{j+1})}^p.
\end{equation*}
Provided $d$ is chosen large enough so as to ensure that the additional $d^{-\varepsilon_{\;\!\!\circ}}$ factor absorbs the unwanted constant $C$, one deduces $(\mathrm{I})_{j+1}$. This should be compared
with the approach of Solymosi and Tao to polynomial partitioning \cite{TS2012}.
\\

%
%
%
\paragraph{\underline{Property II}} By the construction, 
\begin{align*}
    \sum_{O_{j+1} \in \O_{j+1}} \#\T[O_{j+1}] &\,=\, \sum_{O_j \in \O_j}  \sum_{O_{j+1} \in \O_{j+1}(O_j)} \#\T[O_{j+1}]  \\
    &\,\lesssim\, d \sum_{O_j \in \O_j} \#\T[O_j], 
\end{align*}
where the inequality follows from a term-wise application of~\eqref{cellular 1}. Thus, $(\mathrm{II})_j$,~\eqref{cellular word} and~\eqref{cellular coefficients} imply that
\begin{equation*}
    \sum_{O_{j+1} \in \O_{j+1}} \#\T[O_{j+1}] \,\lesssim\, d^{-\varepsilon_{\;\!\!\circ}} C_{j+1}(d)  d^{\#_{\stc}(j+1)} \#\T.
\end{equation*}
Provided $d$ is chosen sufficiently large, one deduces $(\mathrm{II})_{j+1}$.\\

%
%
%
\paragraph{\underline{Property III}} Fix $O_{j+1} \in \O_{j+1}(O_j)$ and recall from~\eqref{cellular 2} that
\begin{equation*}
    \#\T[O_{j+1}] \,\lesssim\, d^{-(m-1)}\#\T[O_{j}].
\end{equation*}
Thus, $(\mathrm{III})_j$,~\eqref{cellular word} and~\eqref{cellular coefficients} imply that
\begin{equation*}
    \#\T[O_{j+1}]\,\lesssim\, d^{-\varepsilon_{\;\!\!\circ}} C_{j+1}(d) d^{-\#_{\stc}(j+1)(m-1)} \#\T[O_{j}].
\end{equation*}
Provided $d$ is chosen sufficiently large as before, one deduces $(\mathrm{III})_{j+1}$.




\subsection*{$\blacktriangleright$ Algebraic-dominant case} Suppose the hypothesis of the cellular-dominant case fails so that
\begin{equation}\label{algebraic dominant case}
   \Big\|\sum_{T \in \T} \chi_T\Big\|_{\BL{k,A}^p(B_{\;\!\!r_0})}^p \leq \,\,2C_{j}(d)\!\!  \sum_{O_j \in \O_{j,\alg}} \Big\|\sum_{T \in \T[O_j]} \chi_T\Big\|_{\BL{k,A_j}^p(O_j)}^p.
\end{equation} 
Each cell in $\O_{j,\alg}$ satisfies the condition of the algebraic case of Theorem~\ref{partitioning theorem}; this information is used to construct the $(j+1)$-generation ensemble. \\

\paragraph{\underline{Definition of $\sE_{j+1}$}} Define $\fh_{j+1}$ by adjoining the letter $\ta$ to the word $\fh_j$. Thus, it follows from the definitions that 
\begin{equation}\label{algebraic word}
    r_{j+1} =  r_j^{1+\varepsilon_{\;\!\!\circ}}, \quad \#_{\btc}(j+1) = \#_{\btc}(j) \quad \textrm{and} \quad \#_{\bta}(j+1) = \#_{\bta}(j)+1.
\end{equation} 
 The next generation of cells is constructed from the varieties  which arise from the algebraic case in Theorem~\ref{partitioning theorem}. Fix $O_j \in \O_{j,\alg}$ so that there exists a transverse complete intersection $\bY_{\!j}$ of dimension $m-1$ and $\deg \bY_{\!j} \leq C_{\mathrm{alg}} d$ such that
\begin{equation*}
    \Big\|\sum_{T \in \T[O_j]} \chi_T\Big\|_{\BL{k,A_j}^p(O_j)}^p \,\lesssim\,  \log d\,\Big\|\sum_{T \in \T[O_j]} \chi_T\Big\|_{\BL{k,A_{j}}^p(O_j \cap N_{\;\!\!\delta}\bY_{\!j})}^p.
\end{equation*}
Let $\cB(O_j)$ be a  cover of $O_j \cap N_{\;\!\!\delta}\bY_{\!j}$ consisting of finitely-overlapping balls of radius $\max\{r_{j+1},\delta^{1 -\varepsilon_{\;\!\!\circ}}\}$. For each $B \in \cB(O_j)$ let $\T_{\!B}$ denote the family of $T \in \T[O_j]$ for which $T \cap B \cap N_{\;\!\!\delta}\bY_{\!j}  \neq \emptyset$. This set is partitioned into the subsets
\begin{equation*}
    \T_{\!B, \tang} := \big\{ T \in \T_{\!B} : \textrm{$T$ is tangent to $\bY_{\!j}$ on $B$} \big\}, \quad \T_{\!B, \trans} := \T_{\!B} \setminus \T_{\!B, \tang};
\end{equation*}
here the notion of tangency is that given in Definition~\ref{tangent definition}.

 By hypothesis, \texttt{[tang]} fails and, consequently, one may deduce that
\begin{equation}\label{algebraic 1}
    \sum_{O_j\in \O_{j,\mathrm{alg}}}\Big\|\sum_{T \in \T[O_j]} \chi_T\Big\|_{\BL{k,A_j}^p(O_j)}^p \lesssim\,\, \log d\!\! \sum_{\substack{O_j \in \O_{j,\mathrm{alg}}\\B \in \cB(O_j)}} \Big\|\sum_{T \in \T_{\!B, \trans}} \chi_T\Big\|_{\BL{k,A_{j+1}}^p(B_j)}^p
\end{equation}
where, for notational convenience, $B_j := B \cap N_{\;\!\!\delta}\bY_{\!j}$. Indeed, provided $C_{\textrm{\texttt{tang}}} > 0$ is sufficiently large,
\begin{align}\label{tangent bounds} \nonumber
   \sum_{O_j \in \O_{j,\alg}} \sum_{B \in \cB(O_j)}  \#\T_{\!B, \tang}  &\,\leq\,\, C_{\textrm{\texttt{tang}}}\delta^{-n\varepsilon_{\;\!\!\circ}}\!\!\sum_{O_j \in \O_j}  \#\T[O_j]; \\ 
  \max_{O_j \in \O_{j,\alg}} \max_{B \in \cB(O_j)}  \#\T_{\!B, \tang}  &\,\leq\,\, \max_{O_j \in \O_j}  \#\T[O_j].
\end{align}
Consequently, the failure of the stopping condition \texttt{[tang]} forces
\begin{equation*}
    \log d\!\!\sum_{O_j \in \O_j} \sum_{B \in \cB(O_j)} \Big\|\sum_{T \in \T_{\!B, \tang}} \!\!\!\!\chi_T\Big\|_{\BL{k,A_{j+1}}^p(B)}^p < \frac{1}{C_{\mathrm{tang}}}  \sum_{O_j \in \O_j} \Big\|\sum_{T \in \T[O_j]} \!\!\!\chi_T\Big\|_{\BL{k,A_j}^p(O_j)}^p 
\end{equation*}
(since the estimates in~\eqref{tangent bounds} show all other conditions for \texttt{[tang]} are met for $\mathcal{S}$, $\T[S]$ and $B[S]$ appropriately defined). On the other hand, by the triangle inequality for broad norms (Lemma~\ref{triangle inequality lemma}), using the fact that $A_{j+1} = A_j/2$, the left-hand side of~\eqref{algebraic 1} is dominated by
\begin{equation*}
 \log d\!\! \sum_{O_j \in \O_{j,\mathrm{alg}}} \sum_{B \in \cB(O_j)} \Big[\big\|\sum_{T \in \T_{\!B, \tang}}\!\!\chi_T\big\|_{\BL{k,A_{j+1}}^p(B_j)}^{p} + \big\|\sum_{T \in \T_{\!B, \trans}}\!\!\chi_T\big\|_{\BL{k,A_{j+1}}^p(B_j)}^p\Big].
\end{equation*}
For a suitable choice of constant $C_{\mathrm{tang}}$, combining the information in the two previous displays yields~\eqref{algebraic 1}.

For $O_j \in \O_{j,\alg}$ define 
\begin{equation*}
    \O_{j+1}(O_j) := \big\{ B \cap N_{\;\!\!\delta}\bY_{\!j} : B \in \cB(O_j) \big\}
\end{equation*}
and let $\T[O_{j+1}] := \T_{\!B, \trans}$ for $O_{j+1} = B \cap N_{\;\!\!\delta}\bY_{\!j} \in \O_{j+1}(O_j)$. The collection of cells $\O_{j+1}$ is then given by
\begin{equation*}
    \O_{j+1} := \bigcup_{O_j \in \O_{j, \alg}} \O_{j+1}(O_j). 
\end{equation*}
It remains to verify that the ensemble $\sE_{j+1}$ satisfies the desired properties. In view of this, it is useful to note that
\begin{equation}\label{algebriac coefficients}
C_{j}(d) = d^{-(n+\varepsilon_{\;\!\!\circ})}C_{j+1}(d),
\end{equation}
which follows directly from the definition of $C_{j}(d)$ and \eqref{algebraic word}.\\

%
%
%
\paragraph{\underline{Property I}} By combining~\eqref{algebraic 1} together with the various definitions one obtains 
\begin{equation*}
  \sum_{O_{j} \in \O_{j,\mathrm{alg}}} \Big\|\sum_{T \in \T[O_j]} \chi_T\Big\|_{\BL{k,A_j}^p(O_{j})}^p \,\lesssim\,\, \log d\!\!\! \sum_{O_{j+1} \in \O_{j+1}} \Big\|\sum_{T \in \T[O_{j+1}]} \chi_T\Big\|_{\BL{k,A_{j+1}}^p(O_{j+1})}^p.
\end{equation*}
Recalling~\eqref{algebraic dominant case} and~\eqref{algebriac coefficients}, if $c(d) := Cd^{-(n+\varepsilon_{\;\!\!\circ})}\log d$ for an appropriate choice of admissible constant $C$, then
\begin{equation*}
\Big\|\sum_{T \in \T} \chi_T\Big\|_{\BL{k,A}^p(B_{\;\!\!r_0})}^p \,\leq \,\,c(d)C_{j+1}(d)\!\!  \sum_{O_{j+1} \in \O_{j+1}}\Big\|\sum_{T \in \T[O_{j+1}]} \chi_T\Big\|_{\BL{k,A_{j+1}}^p(O_{j+1})}^p.
\end{equation*}
 Provided $d$ is sufficiently large, $c(d) \leq 1$ and one thereby deduces $(\mathrm{I})_{j+1}$. 
\\
\paragraph{\underline{Property II}} Fix $O_j \in \O_{j,\mathrm{alg}}$ and note that
\begin{equation}\label{algbraic case property II}
    \sum_{O_{j+1} \in \O_{j+1}(O_j)} \# \T[O_{j+1}] = \sum_{B \in \cB(O_j)}  \#\T_{\!B, \trans} 
\end{equation}
by the definition of $\T[O_{j+1}]$. To estimate the latter sum one may invoke the following algebraic-geometric result of Guth, which appears in Lemma 5.7 of~\cite{Guth2018}. 

\begin{lemma}[\cite{Guth2018}]\label{transversal intersection lemma}
Suppose $T$ is an infinite cylinder in $\R^n$ of radius $\delta$ and central axis $\ell$ and $\bY$ is a transverse complete intersection. For $\alpha > 0$ let
\begin{equation*}
    \bY_{\!> \alpha} := \big\{\, \by \in \bY \,:\, \angle(T_{\by}\bY, \ell) > \alpha \,\big\}.
\end{equation*}
The set $\bY_{\!> \alpha} \cap T$ is contained in a union of $O\big((\deg \bY)^n\big)$ balls of radius $\delta\alpha^{-1}$.
\end{lemma}

Since $T \cap B \cap N_{\;\!\!\delta}\bY \neq \emptyset$ by the definition of $\T_{\!B}$, a tube $T \in \T_{\!B}$ belongs to $\T_{\!B,\trans}$ if and only if the angle condition ii) from Definition~\ref{tangent definition} fails to be satisfied. Thus, given any $T \in \bigcup_{B \in \cB} \T_{\!B, \trans}$, it follows from the definitions that 
\begin{equation*}
    \angle(\mathrm{dir}(T), T_{\by} \bY) \,\gtrsim\, \frac{\delta}{r_{j+1}}
\end{equation*}
for some $\by \in \bY\cap 2B$ with $|\by - \bx| \lesssim \delta$ for some $\bx\in T$. This implies that 
\begin{equation*}
    N_{\;\!\!C\delta}T \cap 2B \cap \bY_{\!> \alpha_{j+1}} \neq \emptyset
\end{equation*}
where $\alpha_{j+1} \sim \delta/r_{j+1}$. Consequently, by Lemma~\ref{transversal intersection lemma}, any $T \in \bigcup_{B \in \cB(O_j)} \T_{\!B, \trans}$ lies in at most $O(d^n)$ of the sets~$\T_{\!B,\mathrm{trans}}$ and so
\begin{equation*}
    \sum_{B \in \cB(O_j)}  \#\T_{\!B, \trans} \,\lesssim\, d^n \#\T[O_j].
\end{equation*}
Combining this inequality with~\eqref{algbraic case property II} and summing over all $O_j \in \O_{j,\mathrm{alg}}$,
\begin{equation*}
    \sum_{O_{j+1} \in \O_{j+1}}  \# \T[O_{j+1}] \,\lesssim\, d^n \sum_{O_{j} \in \O_{j}}\#\T[O_j].
\end{equation*}
Applying $(\mathrm{II})_j$,~\eqref{algebraic word} and~\eqref{algebriac coefficients}, one concludes that
\begin{equation*}
    \sum_{O_{j+1} \in \O_{j+1}} \#\T [O_{j+1}] \,\lesssim\, d^{-\varepsilon_{\;\!\!\circ}}C_{j+1}(d)d^{\#_{\stc}(j+1)} \#\T.
\end{equation*}
Provided $d$ is chosen to be sufficiently large to absorb the implicit constant, one deduces $(\mathrm{II})_{j+1}$. 
\\

\paragraph{\underline{Property III}} Fix $O_j \in \O_{j,\mathrm{alg}}$ and $O_{j+1} \in \O_{j+1}(O_j)$. By definition, $\T[O_{j+1}] \subseteq \T[O_{j}]$ and so \begin{equation*}
   \#\T[O_{j+1}] \,\leq\, \#\T[O_{j}] \,\leq\, C_{j+1}(d) d^{-\#_{\stc}(j+1)(m-1)}\#\T,
\end{equation*}
by $(\mathrm{III})_j$ and~\eqref{algebraic word}.

\subsection*{The second algorithm} The algorithm \texttt{[alg 1]} is now applied repeatedly in order to arrive at a final decomposition of the $k$-broad norm. This process forms part of a second algorithm, referred to as \texttt{[alg 2]}.

Throughout this section  let $p_{\ell}$, with $k \leq \ell \leq n$, denote some choice of Lebesgue exponents satisfying $p_k \geq p_{k+1} \geq \dots \geq p_n =: p \geq 1.$ The numbers $0 \leq  \Theta_{\ell} \leq 1$ are then defined in terms of the $p_{\ell}$ by
\begin{equation*}
   \Theta_\ell :=\Big(1-\frac{1}{p_\ell}\Big)^{-1}\Big(1-\frac{1}{p}\Big)
\end{equation*}
so that $\Theta_n=1$. Also fix $0 < \varepsilon_{\;\!\!\circ}  \ll \varepsilon \ll 1$ as in the previous section. 

There are two stages to \texttt{[alg 2]}, which can roughly be described as follows:
\begin{itemize}
    \item \textbf{The recursive stage}: $\sum_{T \in \T}\chi_T$ is repeatedly decomposed into pieces with favourable tangency properties with respect to varieties of progressively lower dimension.  
    \item \textbf{The final stage}: $\sum_{T \in \T}\chi_T$ is further decomposed into very small scale pieces.
\end{itemize}
To begin, the recursive stage of \texttt{[alg 2]} is described. \\

\paragraph{\underline{\texttt{Input}}} \texttt{[alg 2]} will take as its input:
\begin{itemize}
    \item A choice of small scale $0 < \delta  \ll 1$.
    \item A large integer $A \in \N$. 
    \item A family of $\delta$-tubes $\T$ which are non-degenerate in the sense that
\begin{equation}\label{non degeneracy hypothesis}
\Big\|\sum_{T \in \T} \chi_T\Big\|_{\BL{k,A}^{p}(\R^n)} \neq 0.
\end{equation}
\end{itemize}
Note that the process applies to essentially arbitrary families of $\delta$-tubes (in particular, the direction-separated hypothesis does not appear at this stage).\\

\paragraph{\underline{\texttt{Output}}} The $(n+1-\ell)$th step of the recursion will produce:
\begin{itemize}
    \item An $(n+1-\ell)$-tuple of:
    \begin{itemize}
        \item scales $\vec{\delta}_{\ell} = (\delta_n, \dots, \delta_{\ell})$ satisfying $\delta^{\varepsilon_{\;\!\!\circ}} = \delta_n >  \dots > \delta_{\ell} \ge \delta^{1 - \varepsilon_{\;\!\!\circ}}$;
        \item large and (in general) non-admissible parameters $\vec{D}_{\ell} = (D_n, \dots, D_{\ell})$;
        \item integers $\vec{A} = (A_n, \dots, A_{\ell})$ satisfying $A = A_n > A_{n-1} > \dots > A_{\ell}$.
    \end{itemize}
    Each of these $(n+1-\ell)$-tuples is formed by adjoining a component to the corresponding $(n-\ell)$-tuple from the previous stage.
    \item A family $\vec{\Sc}_{\ell}$ of $(n+1-\ell)$-tuples of transverse complete intersections $\vec{S}_{\ell} = (S_n, \dots, S_{\ell})$ satisfying $\dim S_i = i$ and $\deg S_i = O(1)$ for $\ell \leq i \leq n$. 
    \item An assignment of a $\delta_{\ell}$-ball $B[\vec{S}_{\ell}]$ and a subfamily $\T[\vec{S}_{\ell}]$ of $\delta$-tubes to each $\vec{S}_{\ell} \in \vec{\Sc}_{\ell}$ with the property that the tubes $T \in \T[\vec{S}_{\ell}]$ are tangent to $S_{\ell}$ in $B[\vec{S}_{\ell}]$ (here $S_{\ell}$ is the final component of $\vec{S}_{\ell}$).
\end{itemize}
This data is chosen so that the following properties hold:

\begin{notation} Throughout this section a large number of harmless $\delta^{-\varepsilon_{\;\!\!\circ}}$-factors appear in the inequalities. For notational convenience, given $A, B \geq 0$ let $A \lessapprox B$ or $B \gtrapprox A$ denote $A \lesssim \delta^{-c\varepsilon_{\;\!\!\circ}} B$ for some $c>0$ depending only on $n$ and $p$.
\end{notation}

\paragraph{\underline{Property 1}} The inequality
\begin{equation}\label{l step}
      \Big\|\sum_{T \in \T} \chi_T\Big\|_{\BL{k,A}^p(\R^n)} \lessapprox C(\vec{D}_{\ell}; \vec{\delta}_{\ell}) [\delta^n \#\T]^{1-\Theta_{\ell}} \Big( \sum_{\vec{S}_{\ell} \in \vec{\Sc}_{\ell}} \Big\|\sum_{T \in \T[\vec{S}_{\ell}]} \!\!\!\chi_T\Big\|_{\BL{k,A_{\ell}}^{p_{\ell}}(B[\vec{S}_{\ell}])}^{p_{\ell}}\Big)^{\frac{\Theta_{\ell}}{p_{\ell}}}
\end{equation}
holds for 
\begin{equation*}
    C(\vec{D}_{\ell}; \vec{\delta}_{\ell}) := \prod_{i=\ell}^{n-1} \Big(\frac{\delta_i}{\delta}\Big)^{\Theta_{i+1}-\Theta_{i}}D_i^{(1 + \varepsilon_{\;\!\!\circ})(\Theta_{i+1} - \Theta_{\ell}) + \varepsilon_{\;\!\!\circ}}.
\end{equation*}
\paragraph{\underline{Property 2}}  For $\ell \leq n-1$, the inequality
\begin{equation*}
    \sum_{\vec{S}_{\ell} \in \vec{\Sc}_{\ell}} \# \T[\vec{S}_{\ell}] \,\lessapprox\,\,D_{\ell}^{1 + \varepsilon_{\;\!\!\circ}} \!\!\!\!\sum_{\vec{S}_{\ell+1} \in \vec{\Sc}_{\ell+1}} \# \T[\vec{S}_{\ell+1}] 
\end{equation*}
holds.\\

\paragraph{\underline{Property 3}}    For $\ell \leq n-1$, the inequality
\begin{equation*}
    \max_{\vec{S}_{\ell} \in \vec{\Sc}_{\ell}} \# \T[\vec{S}_{\ell}] \,\lessapprox\,\,D_{\ell}^{-\ell + \varepsilon_{\;\!\!\circ}} \!\!\!\!\max_{\vec{S}_{\ell+1} \in \vec{\Sc}_{\ell+1}} \# \T[\vec{S}_{\ell+1}]
\end{equation*}
holds.\\

By the inclusion property \eqref{inclusion}, the broad norms over $B[\vec{S}_{\ell}]$ on the right-hand side of \eqref{l step} could be replaced by broad norms over $4\delta$-neighbourhoods of $S_{\ell}$.\\

\paragraph{\underline{\texttt{First step}}} Vacuously, the tubes belonging to $\T$ are tangent to the $n$-dimensional variety $\R^n$. Let $\mathcal{B}_{\circ}$ denote a collection of finitely-overlapping balls of radius $\delta^{\varepsilon_{\circ}}$ which cover $\bigcup_{T \in \T} T$ and define
\begin{itemize}
        \item $\delta_n := \delta^{\varepsilon_{\;\!\!\circ}}$; $D_n := 1$ and $A_n:= A$;
        \item $\Sc_n$ is the collection consisting of repeated copies of the 1-tuple $(\R^n)$, with one copy for each ball in $\mathcal{B}_{\circ}$;
        \item For each $\vec{S}_n \in \Sc_n$ assign a ball $B[\vec{S}_n] \in \mathcal{B}_{\circ}$ and let 
        \begin{equation*}
            \T[\vec{S}_n] := \big\{T \in \T : T \cap B[\vec{S_n}] \neq \emptyset\big\}.
        \end{equation*}
\end{itemize}
By a straightforward orthogonality argument (identical to that used to establish the base case in the proof of Proposition \ref{Bourgain--Guth}), Property 1 can be shown to hold with $C(\vec{D}_{n}; \vec{\delta}_{n})=1$ and $\Theta_n=1$.\\

\paragraph{\underline{\texttt{($n+2-\ell$)th step}}} Let $\ell \geq 1$ and suppose that the recursive algorithm has ran through $n+1-\ell$ steps. Since each family $\T[\vec{S}_{\ell}]$ consists of $\delta$-tubes which are tangent to $S_{\ell}$ on $B[\vec{S}_{\ell}]$, one may apply \texttt{[alg 1]} to bound the $k$-broad norm 
\begin{equation*}
    \Big\|\sum_{T \in \T[\vec{S}_{\ell}]} \chi_T\Big\|_{\BL{k,A_{\ell}}^{p_{\ell}}(B[\vec{S}_{\ell}])}.
\end{equation*} 
One of two things can happen: either \texttt{[alg 1]} terminates due to the stopping condition \texttt{[tiny]} or it terminates due to the stopping condition \texttt{[tang]}. The current recursive process terminates if the contributions from terms of the former type dominate:\\

\paragraph{\underline{\texttt{Stopping condition}}} The recursive stage of \texttt{[alg 2]} has a single stopping condition, which is denoted by \texttt{[tiny-dom]}. 
\begin{itemize}
    \item[\texttt{Stop:[tiny-dom]}] Suppose that the inequality
  \begin{equation}\label{tiny case 1}
      \sum_{\vec{S}_{\ell} \in \vec{\Sc}_{\ell}} \Big\|\sum_{T \in \T[\vec{S}_{\ell}]} \chi_T\Big\|_{\BL{k,A_{\ell}}^{p_{\ell}}(B[\vec{S}_{\ell}])}^{p_{\ell}} \,\leq\, 2 \sum_{\vec{S}_{\ell} \in \vec{\Sc}_{\ell,\textrm{\texttt{tiny}}}}  \Big\|\sum_{T \in \T[\vec{S}_\ell]} \chi_T \Big\|_{\BL{k,A_{\ell}}^{p_{\ell}}(B[\vec{S}_{\ell}])}^{p_{\ell}}
\end{equation}
holds, where the right-hand summation is restricted to those $S_{\ell} \in \vec{\Sc}_{\ell}$ for which \texttt{[alg 1]} terminates owing to the stopping condition \texttt{[tiny]}. Then \texttt{[alg 2]} terminates.
\end{itemize}

Assume that the condition \texttt{[tiny-dom]} is not met. Necessarily,
  \begin{equation}\label{tangent case 1}
      \sum_{\vec{S}_{\ell} \in \vec{\Sc}_{\ell}}\Big\|\sum_{T \in \T[\vec{S}_{\ell}]} \chi_T\Big\|_{\BL{k,A_{\ell}}^{p_{\ell}}(B[\vec{S}_{\ell}])}^{p_{\ell}} \leq \,2 \sum_{\vec{S}_{\ell} \in \vec{\Sc}_{\ell,\textrm{\texttt{tang}}}} \Big\|\sum_{T \in \T[\vec{S}_{\ell}]} \chi_T\Big\|_{\BL{k,A_{\ell}}^{p_{\ell}}(B[\vec{S}_{\ell}])}^{p_{\ell}},
\end{equation}
where the right-hand summation is restricted to those $S_{\ell} \in \vec{\Sc}_{\ell}$ for which \texttt{[alg 1]} does not terminate owing to \texttt{[tiny]} and therefore terminates owing to \texttt{[tang]}. Consequently, for each $\vec{S}_{\ell} \in \vec{\Sc}_{\ell,\textrm{\texttt{tang}}}$ the inequalities
  \begin{equation}\label{recursive step property 1}
      \Big\|\sum_{T \in \T[\vec{S}_{\ell}]} \chi_T\Big\|_{\BL{k,A_{\ell}}^{p_{\ell}}(B[\vec{S}_{\ell}])}^{p_{\ell}} \,\lessapprox  D_{\ell - 1}^{\varepsilon_{\;\!\!\circ}} \sum_{S_{\ell-1} \in \Sc_{\ell-1}[\vec{S}_{\ell}]} \Big\|\sum_{T \in \T[\vec{S}_{\ell-1}]}\chi_T\Big\|_{\BL{k,2A_{\ell-1}}^{p_{\ell}}(B[\vec{S}_{\ell-1}])}^{p_{\ell}},
\end{equation}
and
\begin{align}
\label{recursive step property 2}
    \sum_{S_{\ell-1} \in \Sc_{\ell-1}[\vec{S}_{\ell}]}  \#\T[S_{\ell-1}] &\,\lessapprox\, D_{\ell-1}^{1 + \varepsilon_{\;\!\!\circ}} \#\T[S_{\ell}]; \\
\label{recursive step property 3}
\max_{S_{\ell-1} \in \Sc_{\ell-1}[\vec{S}_{\ell}]} \#\T[S_{\ell-1}] &\,\lessapprox\,  D_{\ell-1}^{-(\ell-1) + \varepsilon_{\;\!\!\circ}} \#\T[S_{\ell}]
\end{align}
hold for some choice of:
\begin{itemize}
    \item Scale $\delta_{\ell-1}$ satisfying $\delta_{\ell} > \delta_{\ell-1} \geq \delta^{1-\varepsilon_{\;\!\!\circ}}$; non-admissible number $D_{\ell-1}$ and large integer $A_{\ell-1}$ satisfying $A_{\ell-1} \sim A_{\ell}$;
    \item Family $\Sc_{\ell-1}[\vec{S}_{\ell}]$ of $(\ell-1)$-dimensional transverse complete intersections of degree $O(1)$;
    \item Assignment of a subfamily $\T[\vec{S}_{\ell-1}] = \T[\vec{S}_{\ell}][S_{\ell-1}]$ of $\delta$-tubes for every $S_{\ell-1} \in \Sc_{\ell-1}[\vec{S}_{\ell}]$ such that each $T \in \T[\vec{S}_{\ell-1}]$ is tangent to $S_{\ell-1}$ on $B[\vec{S}_{\ell-1}]$.
\end{itemize}
Each inequality~\eqref{recursive step property 1},~\eqref{recursive step property 2} and~\eqref{recursive step property 3} is obtained by combining the definition of the stopping condition \texttt{[tang]} with Properties I, II and III from \texttt{[alg 1]}, respectively. Indeed, we take $$r_0 := \delta_{\ell},\quad \delta_{\ell-1} := \max\{r_{\!J}^{1+\varepsilon_{\;\!\!\circ}},\delta^{1-\varepsilon_{\;\!\!\circ}}\},\quad \text{and}\quad D_{\ell-1} := d^{\#_{\stc}(J)},$$ using the notation from \texttt{[alg 1]}.

The $\delta_{\ell-1}$, $D_{\ell-1}$ and $A_{\ell-1}$ can depend on the choice of $\vec{S}_{\ell}$, but this dependence can be essentially removed by pigeonholing. In particular,  $\#_{\stc}(J)$ depends on $\vec{S}_{\ell}$, but satisfies ${\#_{\stc}(J)}= O(\log \delta^{-1})$. Thus, since there are only logarithmically many possible different values, one may find a subset of the $\Sc_{\ell,\textrm{\texttt{tang}}}$ over which the $D_{\ell-1}$ all have a common value and, moreover, the inequality~\eqref{tiny case 1} still holds except that the constant $1/2$ is now replaced with, say,~$\delta^{-\varepsilon_{\;\!\!\circ}}$. A brief inspection of \texttt{[alg 1]} shows that both $\delta_{\ell-1}$ and $A_{\ell-1}$ are determined by~$D_{\ell-1}$ and so the desired uniformity is immediately inherited by these parameters. 

Letting $\vec{\Sc}_{\ell-1}$ denote the structured set 
\begin{equation*}
    \vec{\Sc}_{\ell-1} := \big\{ (\vec{S}_{\ell}, S_{\ell-1}) : \vec{S}_{\ell} \in \vec{\Sc}_{\ell,\textrm{\texttt{tang}}} \textrm{ and } S_{\ell-1} \in \Sc_{\ell-1}[\vec{S}_{\ell}] \big\},
\end{equation*}
where $\vec{\Sc}_{\ell,\textrm{\texttt{tang}}}$ is understood to be the refined collection described in the previous paragraph, it remains to verify that the desired properties hold for the newly constructed data. Property 2 follows immediately from~\eqref{recursive step property 2} and Property 3 from~\eqref{recursive step property 3}, so it remains only to verify Property 1.

By combining the inequality~\eqref{l step} from the previous stage of the algorithm with~\eqref{tangent case 1} and~\eqref{recursive step property 1}, one deduces that
   \begin{equation*}
    \Big\|\sum_{T \in \T} \chi_T\Big\|_{\BL{k,A}^p(\R^n)}  \,\lessapprox\,  D_{\ell - 1}^{\varepsilon_{\;\!\!\circ}} C(\vec{D}_{\ell};\vec{\delta}_{\ell}) [\delta^n\#\T]^{1-\Theta_{\ell}} \Big\|\sum_{T \in \T[\vec{S}_{\ell-1}]} \chi_T\Big\|_{\ell^{p_{\ell}}\BL{k,2A_{\ell-1}}^{p_{\ell}}(\vec{\Sc}_{\ell-1})}^{\Theta_{\ell}}
\end{equation*}
where, for any $1 \leq q < \infty$ and $M \in \N$, we write
\begin{equation*}
    \Big\|\sum_{T \in \T[\vec{S}_{\ell-1}]} \chi_T\Big\|_{\ell^{q}\BL{k,M}^{q}(\vec{\Sc}_{\ell-1})} := \Bigg(\sum_{\vec{S}_{\ell-1}\in \vec{\Sc}_{\ell-1}}\!\Big\|\sum_{T \in \T[\vec{S}_{\ell-1}]} \chi_T\Big\|_{\BL{k,M}^{q}(B[\vec{S}_{\ell - 1}])}^{q}\Bigg)^{1/q}.
\end{equation*}
Taking $q = p_{\ell}$ and $M = 2A_{\ell-1}$, the logarithmic convexity inequality (Lemma~\ref{logarithmic convexity inequality lemma}) dominates the preceding expression by
\begin{equation*}
  \Big\|\sum_{T \in \T[\vec{S}_{\ell-1}]} \chi_T\Big\|_{\ell^{1}\BL{k,A_{\ell-1}}^{1}(\vec{\Sc}_{\ell-1})}^{1-\Theta_{\ell-1}/\Theta_\ell}\,  \Big\|\sum_{T \in \T[\vec{S}_{\ell-1}]} \chi_T\Big\|_{\ell^{p_{\ell-1}}\BL{k,A_{\ell-1}}^{p_{\ell-1}}(\vec{\Sc}_{\ell-1})}^{\Theta_{\ell-1}/\Theta_{\ell}}.
\end{equation*}
Observe that, trivially, one has
\begin{equation*}
 \Big\|\sum_{T \in \T[\vec{S}_{\ell-1}]} \chi_T\Big\|_{\ell^{1}\BL{k,A_{\ell-1}}^{1}(\vec{\Sc}_{\ell-1})}\, \lesssim\,\,\Big(\frac{\delta_{\ell-1}}{\delta}\Big) \delta^n\!\!\! \sum_{\vec{S}_{\ell-1} \in \vec{\Sc}_{\ell-1}} \# \T[\vec{S}_{\ell-1}].
\end{equation*}
and, by Property 2 for the tube families $\{\T[\vec{S}_{i}] : \vec{S}_i \in \vec{\Sc}_i\}$ for $\ell - 1 \leq i \leq n-1$, it follows that
\begin{equation*}
 \Big\|\sum_{T \in \T[\vec{S}_{\ell-1}]} \chi_T\Big\|_{\ell^{1}\BL{k,A_{\ell-1}}^{1}(\vec{\Sc}_{\ell-1})} \,\lesssim\, \Big(\frac{\delta_{\ell-1}}{\delta}\Big)  \Big( \prod_{i=\ell-1}^{n-1}D_{i}^{1 + \varepsilon_{\;\!\!\circ}}\Big) \delta^n \#\T.
\end{equation*}
One may readily verify that
\begin{equation*}
    D_{\ell - 1}^{\varepsilon_{\;\!\!\circ}} C(\vec{D}_{\ell}; \vec{\delta}_{\ell})\cdot  \Big( \frac{\delta_{\ell-1}}{\delta} \prod_{i=\ell-1}^{n-1}D_{i}^{1 + \varepsilon_{\;\!\!\circ}}\Big)^{\Theta_\ell-\Theta_{\ell-1}} = C(\vec{D}_{\ell-1};\vec{\delta}_{\ell-1})
\end{equation*}
and so, combining the above estimates, 
   \begin{equation*}
     \Big\|\sum_{T \in \T} \chi_T\Big\|_{\BL{k,A}^p(\R^n)} \,\lessapprox\, C(\vec{D}_{\ell-1};\vec{\delta}_{\ell-1}) [
      \delta^n\#\T]^{1-\Theta_{\ell-1}} \Big\|\!\!\sum_{T \in \T[\vec{S}_{\ell-1}]} \!\! \chi_T  \Big\|_{\ell^{p_{\ell-1}}\BL{k,A_{\ell-1}}^{p_{\ell-1}}(\vec{\Sc}_{\ell-1})}^{\Theta_{\ell-1}},
\end{equation*}
which is Property 1.
\\

\paragraph{\underline{\texttt{The final stage}}} If the algorithm has not stopped by the $k$th step, then it necessarily terminates at the $k$th step. Indeed, otherwise~\eqref{l step} would hold for $\ell = k-1$ and families $\T[\vec{S}_{{k-1}}]$ of $\delta_{k-1}$-tubes which are tangent to some transverse complete intersection of dimension $k-1$. By the vanishing property of the $k$-broad norms as described in Lemma~\ref{vanishing lemma}, one would then have
\begin{equation*}
    \Big\|\sum_{T \in \T[\vec{S}_{k-1}]} \chi_T\Big\|_{\BL{k,A_{k-1}}^{p_{k-1}}(B[\vec{S}_{k-1}] )} = 0,
\end{equation*}
which, by \eqref{l step}, would contradict the non-degeneracy hypothesis~\eqref{non degeneracy hypothesis}.

Suppose the recursive process terminates at step $m$, so that $m \geq k$. For each $\vec{S}_m \in \vec{\Sc}_{m,\textrm{\texttt{tiny}}}$ let $\O[\vec{S}_m]$ denote the final collection of cells output by \texttt{[alg 1]} (that is, the collection denoted by $\O_{\!J}$ in the notation of the previous subsection) when applied to estimate the broad norm $\|\sum_{T \in \T[\vec{S}_m]}\chi_T\|_{\BL{k,A_m}^{p_m}(B[\vec{S}_m])}$. 
By Properties I, II and III of \texttt{[alg 1]} one has 
   \begin{equation*}
      \Big\|\sum_{T \in \T[\vec{S}_m]}\chi_T\Big\|_{\BL{k,A_m}^{p_m}(B[\vec{S}_m])}^{p_m} \,\lesssim \sum_{O \in \O[\vec{S}_m]}  \Big\|\sum_{T \in \T[O]} \chi_T \Big\|_{\BL{k,A_{m-1}}^{p_m}(O)}^{p_m},
\end{equation*}
for some $A_{m-1} \sim A_m$ where the families $\T[O]$ satisfy
\begin{equation}\label{tiny property 2}
    \sum_{O \in \O[\vec{S}_m]} \#\T[O] \,\lesssim\, D_{m-1}^{1+ \varepsilon_{\circ}}  \#\T[\vec{S}_m]
    \end{equation}
and
\begin{equation}\label{tiny property 3}
    \max_{O \in \O[\vec{S}_m]} \#\T[O] \,\lesssim\, D_{m-1}^{-(m-1) + \varepsilon_{\circ}} \#\T[\vec{S}_m]
\end{equation}
for $D_{m-1}$ a large and (in general) non-admissible parameter. Once again, by pigeonholing, one may pass to a subcollection of $\Sc_{m, \textrm{\texttt{tiny}}}$ and thereby assume that the $D_{m-1}$ (and also the $A_{m-1}$) all share a common value. 

If $\O$ denotes the union of the $\O[\vec{S}_m]$ over all $\vec{S}_m$ belonging to subcollection of $\Sc_{m, \textrm{\texttt{tiny}}}$ described above, then \texttt{[alg 2]} outputs the following inequality.
\begin{first key estimate} 
   \begin{equation*}
     \Big\|\sum_{T \in \T} \chi_T\Big\|_{\BL{k,A}^p(\R^n)} \lessapprox\,  C(\vec{D}_m; \vec{\delta}_m) [\delta^n \#\T]^{1-\Theta_m}  \Bigg(\sum_{O\in\O}\Big\|\sum_{T \in \T[O]} \chi_T\Big\|^{p_m}_{\BL{k,A_{m-1}}(O)}\Bigg)^{\frac{\Theta_m}{p_m}}\!\!.
\end{equation*}
\end{first key estimate}

\section{Proof of Theorem~\ref{main theorem}}\label{proof section}

Henceforth, fix $\T$ to be a direction-separated family of $\delta$-tubes in $\R^n$. Without loss of generality, we may assume that~$\T$ satisfies the non-degeneracy hypothesis~\eqref{non degeneracy hypothesis}. The algorithms described in the previous section can be applied to this tube family, leading to the final decomposition of the broad norm described in the first key estimate. One therefore wishes to show, using the direction-separated hypothesis, that the quantity on the right-hand side of the first key estimate can be effectively bounded, provided that the exponents $p_k, \dots, p_n$ are suitably chosen. 

Since each $O \in \O$ is contained in a ball of radius at most $\delta^{1-\varepsilon_{\;\!\!\circ}}$, trivially one may bound
\begin{equation*}
    \Big\|\sum_{T \in \T[O]} \chi_T\Big\|_{\BL{k,A_{m-1}}^{p_m}(O)}^{p_m} \,\lessapprox\, \delta^{n} \big(\#\T[O]\big)^{p_m}.
\end{equation*}
Recalling that 
$
  \Theta_{m}(1-\frac{1}{p_{m}}) = 1-\frac{1}{p},
$ this yields
\begin{equation*}
\Bigg(\sum_{O\in\O}\Big\|\sum_{T \in \T[O]} \chi_T\Big\|^{p_m}_{\BL{k,A_{m-1}}^{p_m}(O)}\Bigg)^{\frac{\Theta_m}{p_m}}\lessapprox\,   \big(\max_{O \in \O}\#\T[O]\big)^{1-\frac{1}{p}}\Big( \delta^{n}\sum_{O \in \O}  \#\T[O]\Big)^{\frac{\Theta_m}{p_m}}.
\end{equation*}
Now \eqref{tiny property 2} and repeated application of Property 2 from \texttt{[alg 2]} imply 
\begin{equation*}
 \sum_{O \in \O} \#\T[O] \,\lessapprox\, \Big(\prod_{i = m-1}^{n-1} D_{i}^{1 + \varepsilon_{\;\!\!\circ}}\Big)\#\T.  
\end{equation*}
 Combining this with the first key estimate and the definition of $C(\vec{D}_m;\vec{\delta}_m)$, one concludes that 
  \begin{equation}\label{reduction to max bound}
    \Big\|\sum_{T \in \T} \chi_T\Big\|_{\BL{k,A}^{p}(\R^n)}  \,\lessapprox\,  \mathbf{C}(\vec{D}; \vec{\delta}\,)  \big(\max_{O \in \O}\#\T[O]\big)^{1 - \frac{1}{p}} \Big(  \delta\sum_{T \in \T} |T| \Big)^{\frac{1}{p}}
\end{equation}
where, taking $\delta_{m-1} := \delta$, the constant takes the form
\begin{equation*}
    \mathbf{C}(\vec{D}; \vec{\delta}\,) := \prod_{i = m-1}^{n-1} \Big(\frac{\delta_{i}}{\delta}\Big)^{\Theta_{i+1}-\Theta_{i}}D_{i}^{\Theta_{i+1} - (1 - \frac{1}{p}) + O(\varepsilon_{\;\!\!\circ})}.
\end{equation*}

In order to bound the maximum appearing on the right-hand side of~\eqref{reduction to max bound}, by~\eqref{tiny property 3} and repeated application of Property 3 of \texttt{[alg 2]}, it follows that 
\begin{equation*}
\max_{O \in \O} \#\T[O] \,\lessapprox\, \Big( \prod_{i=m-1}^{\ell-1}D_{i}^{-i + \varepsilon_{\;\!\!\circ}}\Big) \max_{\vec{S}_{\ell} \in \vec{\Sc}_{\ell}}\#\T[\vec{S}_{\ell}]
\end{equation*}
whenever $m \leq \ell \leq n$. Recall, for each tube family $\T[\vec{S}_{\ell}]$ produced by \texttt{[alg 2]} and each $\ell \leq i \leq n-1$ there exists a $\delta_i$-ball $B_{\delta_{i}} := B[\vec{S}_{i}]$ such that every $\delta$-tube $T \in \T[\vec{S}_{\ell}]$ is tangent to~$S_{i}$ in~$B_{\;\!\!\delta_{i}}$; in particular, 
\begin{equation*}
    T \cap B_{\;\!\!\delta_i} \cap N_{\;\!\!\delta}S_i \neq \emptyset \quad \textrm{and} \quad T \cap 2 B_{\;\!\!\delta_i} \subseteq N_{\;\!\!4\delta} S_i \qquad \textrm{for $\ell \leq i \leq n-1$}.
\end{equation*}
Here $S_i$ is a transverse complete intersection of dimension $i$ and $\deg S_i$ depends only on the admissible parameters $n$, $p$ and $\varepsilon$. Thus, Theorem~\ref{mainThm} implies that
\begin{equation*}
    \#\T[\vec{S}_{\ell}] \leq \# \bigcap_{i = \ell}^{n-1}\Big\{ T \in \T : |T \cap 2B_{\;\!\!\delta_i} \cap N_{\;\!\!4\delta} S_i| \geq 2 \delta_i |T| \Big\} \lessapprox  \delta^{-(n-1)}\prod_{i=\ell}^{n-1} \Big( \frac{\delta_i}{\delta}\Big)^{-1},
\end{equation*}
where the first inequality follows from elementary geometric considerations. Combining these observations,
\begin{equation*}
    \max_{O \in \O}\#\T[O] \,\lessapprox\, \Big( \prod_{i = m-1}^{\ell-1}D_{i}^{-i + \varepsilon_{\;\!\!\circ}}\Big)\delta^{-(n-1)}\prod_{i=\ell}^{n-1} \Big( \frac{\delta_i}{\delta}\Big)^{-1}.
\end{equation*}
for all $m \leq \ell \leq n$. Finally, these $n-m+1$ different estimates can be combined into a single inequality by taking a weighted geometric mean, yielding:

\begin{second key estimate} Let $0 \leq \gamma_m, \dots, \gamma_n \leq 1$ satisfy $\sum_{j = m}^{n} \gamma_{j} = 1$. Then
\begin{equation*}
    \max_{O \in \O} \#\T[O] \,\lessapprox\, \Big(\prod_{i=m-1}^{n-1}\Big(\frac{\delta_i}{\delta}\Big)^{-\sum_{j=m}^i\gamma_{j}} D_{i}^{-i(1-\sum_{j=m}^{i} \gamma_j) + O(\varepsilon_{\;\!\!\circ})} \Big) \delta^{-(n-1)}.
\end{equation*}
\end{second key estimate} 
When $i = m-1$ the $(\delta_i/\delta)^{- \sum_{j=m}^i\gamma_{j}}$ factor is understood to be equal to 1.

Substituting the second key estimate into~\eqref{reduction to max bound}, one obtains
  \begin{equation*}
      \Big\|\sum_{T \in \T} \chi_T\Big\|_{\BL{k,A}^{p}(\R^n)} \,\lessapprox\,  \Big(\prod_{i = m-1}^{n-1} \Big(\frac{\delta_{i}}{\delta}\Big)^{X_{i}}D_{i}^{Y_{i} + O(\varepsilon_{\;\!\!\circ})}\Big) \delta^{-(n-1-\frac{n}{p})}
      \Big(\sum_{T \in \T} |T| \Big)^{\frac{1}{p}} 
\end{equation*}
where
\begin{align*}
    X_{i}&:= \Theta_{i+1}-\Theta_{i}- \Big(\sum_{j=m}^i\gamma_{j}\Big)\Big(1-\frac{1}{p}\Big); \\
    Y_{i}&:= \Theta_{i+1} - \Big(1+i\big(1-\sum_{j=m}^{i} \gamma_j\big)\Big)\Big(1-\frac{1}{p}\Big).
\end{align*}
One now chooses the various exponents so that $X_{i}, Y_{i} = 0$ for all $m \leq i \leq n-1$ and $Y_{m-1} = 0$. This ensures that the $(\delta_i/\delta)^{X_i}$ and $D_i^{Y_i}$ factors in the above expression are admissible but does not allow one to control the $D_i^{O(\varepsilon_{\;\!\!\circ})}$ factors, which may still be non-admissible. To deal with the $D_i^{O(\varepsilon_{\;\!\!\circ})}$ one may perturb the~$p$ exponent which results under the conditions $X_{i}, Y_{i} = 0$, so that $Y_{i}$ becomes negative, and then choose $\varepsilon_{\;\!\!\circ}$ sufficiently small depending on the choice of perturbation. This yields an open range of $k$-broad estimates, which can then be trivially extended to a closed range via interpolation through logarithmic convexity (the interpolation argument relies on the fact that one is permitted an $\delta^{-\varepsilon}$-loss in the constants in the $k$-broad inequalities).

The condition $X_{i} = 0$ is equivalent to
\begin{equation}\label{our r constraint}
\Big(1 - \frac{1}{p_{i+1}}\Big)^{-1}-\Big(1 - \frac{1}{p_{i}}\Big)^{-1} =  \sum_{j=m}^i\gamma_{j}
\end{equation}
whilst the condition $Y_{i-1} = 0$ is equivalent to 
\begin{equation}\label{our D constraint}
    \Big(1 - \frac{1}{p_{i}}\Big)^{-1} = i - (i-1)\sum_{j=m}^{i-1} \gamma_j.
\end{equation}
Choose $p_m := \frac{m}{m -1}$ so that~\eqref{our D constraint} holds in the $i = m$ case. The remaining~$p_{i}$ are then defined in terms of the $\gamma_j$ by the equation
\begin{equation}\label{linear system 1} \Big(1 - \frac{1}{p_{i}}\Big)^{-1}=m + \sum_{j=m}^{i-1}(i-j)\gamma_j
\end{equation} 
so that each of the $n-m$ constraints in~\eqref{our r constraint} is met. 

It remains to solve for the $n-m+1$ variables $\gamma_m, \dots, \gamma_n$. By comparing the right-hand sides of~\eqref{our D constraint} and~\eqref{linear system 1}, it follows that
\begin{equation}\label{linear system 2}
  \sum_{j = m}^{i-1} (2i - j - 1)\gamma_j = i - m  \qquad \textrm{for $m+1 \leq i \leq n$.}
\end{equation}
To solve this linear system, let $\kappa_{i}$ denote the left-hand side of \eqref{linear system 2} and observe that 
\begin{equation*}
    \kappa_{i+1} + \kappa_{i-1} - 2\kappa_{i} =   (i+1)\gamma_{i}-(i - 2)\gamma_{i-1} \qquad\textrm{ for $m+1 \leq i \leq n-1$,}
\end{equation*}
where $\kappa_m := 0$. On the other hand, by considering the right-hand side of~\eqref{linear system 2}, it is clear that $\kappa_{i+1} + \kappa_{i-1} - 2\kappa_{i} = 0$. Combining these observations gives a recursive relation
\begin{equation*}
\gamma_m := \frac{1}{m+1}, \qquad \gamma_i = \Big(\frac{i-2}{i+1}\Big) \gamma_{i-1} \quad \textrm{for $m + 1 \leq i \leq n$} 
\end{equation*}
and from this one deduces that 
\begin{equation*}
    \gamma_j = \frac{1}{m+1} \prod_{i=m}^{j-1} \frac{i-1}{i+2} = \frac{(m-1)m}{(j-1)j(j+1)} \qquad \textrm{for $m \leq j \leq n-1$.}
\end{equation*}

It remains to check that these parameter values give the correct value of $p_n$, corresponding to the exponent featured in Theorem~\ref{main theorem}. It follows from~\eqref{our D constraint} that
\begin{align*}
    \Big(1 - \frac{1}{p_n}\Big)^{-1} &= n -(n-1) (m-1)m\sum_{j=m}^{n-1} \frac{1}{(j-1)j(j+1)} \\
     &= n - \frac{(n-1)n - (m-1)m}{2n}.
\end{align*}
This is smallest when $m=k$, which directly yields the desired range of $p$, as stated in Theorem~\ref{main theorem}, completing the proof. 




\section{Remarks on the numerology and related results}\label{final remarks}

In this section we discuss the relationship between the main result of this paper and the existing literature on the Kakeya set conjecture. The first step is to obtain a more explicit range of exponents for Theorem~\ref{linear theorem}, which is treated in the following subsection. Later, we also discuss applications of the method of this article to certain variants of the Kakeya maximal problem.




 \subsection{Numerology} Recall that the range of exponents in Theorem~\ref{linear theorem} is given by 
\begin{equation}\label{not neat}
   p \geq  1+\min_{2\leq k\leq n}\max\Big\{\, \frac{2n}{n(n-1)+k(k-1)}, \,\frac{1}{n-k+1}\,\Big\}.
\end{equation}
As claimed in the introduction, this guarantees that Conjecture~\ref{maximal conjecture} holds in the range
\begin{equation}\label{neat}
p\ge 1+\frac{1}{2-\sqrt{2}}\frac{1}{n-1}.
\end{equation} 
In fact, in many dimensions a somewhat better bound is obtained. To see this, allowing $k$ to be non-integer for a moment, one finds that the minimum value in \eqref{not neat} is attained when $k = k_1$ where $k_1 = k_1(n)$ is chosen so that
$$\frac{2n}{n(n-1)+k_1(k_1-1)}=\frac{1}{n- k_1+1}.$$
Solving the quadratic, one deduces that
\begin{align*}
 k_1 &= (\sqrt{2}-1)n+\frac{1}{2}+\sqrt{2}n\Big(\big(1+\frac{1}{n}+\frac{1}{8n^2}\big)^{1/2}-1\Big) \\
&\le (\sqrt{2}-1)n+\frac{1}{2}+\frac{1}{\sqrt{2}} +\frac{1}{8\sqrt{2}n},
\end{align*}
where the upper bound follows by Bernoulli's inequality. Let $\tilde{k}_1$ denote the expression appearing on the last line of the above display. Since the sequence $(\sqrt{2} - 1)n$ is equidistributed modulo 1, for any $\varepsilon > 0$ there exist infinitely many values of $n$ for which the interval $[\tilde{k}_1, \tilde{k}_1 + \varepsilon]$ contains an integer. For any such value of $n$ it follows that Conjecture~\ref{maximal conjecture} is true in the range 
$$
p \ge 1+\frac{1}{(2-\sqrt{2})n+\frac{1}{2}-\frac{1}{\sqrt{2}}  -\frac{1}{8\sqrt{2}n}} + \varepsilon.
$$
On the other hand, considering the worst case scenario, when $k$ is not close to an integer, we can at least find an integer in $[k_0,k_0+1]$, with $k_0  = k_0(n)$ chosen so that  
$$\frac{2n}{n(n-1)+k_0(k_0-1)}=\frac{1}{n-(k_0+1)+1}.$$
Calculating $k_0$ and bounding from above using Bernoulli's inequality as before, we find that, in any dimension, 
Conjecture \ref{maximal conjecture} is true in the range 
$$
p \ge 1+\frac{1}{(2-\sqrt{2})n-\frac{1}{2} -\frac{1}{8\sqrt{2}n}}.
$$
This range is always larger than the one stated in \eqref{neat}.




\subsection{Implications for the Kakeya set conjecture} As mentioned in the introduction, a maximal estimate of the form \eqref{maximal inequality} implies that the Hausdorff dimension of any Kakeya set must be greater than or equal to $p'$, where $1/p+1/p'=1$. It is instructive to compare the Hausdorff dimension bounds obtained from Theorem~\ref{linear theorem} with the current best known high dimensional results on the Kakeya set conjecture due to Katz and Tao~\cite{KT2002}. In particular, in \cite{KT2002} it was shown that Kakeya sets in $\mathbb{R}^n$ have Hausdorff dimension greater than or equal to $(2-\sqrt{2})(n-4)+3$.\footnote{This is an improved range over what can be obtained directly from the maximal estimate in \cite{KT2002}; an even larger bound for the Minkowski dimension is obtained in \cite{KT2002} for dimensions $n\ge 24$.} Considering the best case scenario from the previous section, we are able to obtain the following improvement. 

\begin{corollary} For every $\varepsilon > 0$ there exists an infinite sequence of dimensions $n$ such that every Kakeya set $K \subseteq  \R^n$ satisfies
\begin{equation*}
  \dim_H K \geq (2-\sqrt{2})n+\frac{3}{2}-\frac{1}{\sqrt{2}} - \varepsilon.
\end{equation*}
\end{corollary}
Provided $\varepsilon > 0$ is sufficiently small, this bound is stronger than that obtained by Katz--Tao \cite{KT2002}. On the other hand, the Hausdorff dimension bound provided by Theorem~\ref{linear theorem} is also weaker than the result \cite{KT2002} for infinitely many dimensions. In our worst case scenario, arguing as in the previous subsection, given any $\varepsilon > 0$ we can find infinitely many dimensions $n$ for which our results do not provide a better bound than
\begin{equation*}
   \dim_H K \geq (2-\sqrt{2})n+\frac{1}{2}  +  \varepsilon.
\end{equation*}
 Provided $\varepsilon > 0$ is sufficiently small, this is strictly worse than the Katz--Tao Hausdorff dimension estimate. See Figure~\ref{dimension} for the state-of-the-art in lower dimensions. It is perhaps interesting that the polynomial partitioning approach of this article yields the same $(2-\sqrt{2})n + O(1)$ numerology as the (completely different) sum-difference approach employed by Katz and Tao \cite{KT2002}.

\renewcommand{\arraystretch}{1.3}
 \begin{figure}
\begin{tabular}{ ||c|c|c||c|c|c|| } 
 \hline
  $n=$ & $\dim_H \ge$  &  & $n=$ & $\dim_H \ge$  &  \\ 
   \hline 
 2 &  \cellcolor{white} 2 &  \cellcolor{white} Davies~\cite{Davies1971}  &   9 &  \cellcolor{yellow!50} $6$ &  \cellcolor{yellow!50} Theorem~\ref{linear theorem}\\ 
  3 & \cellcolor{white} $5/2+\varepsilon$ &  \cellcolor{white} Katz--Zahl~\cite{KZ2019}  &  10 & \cellcolor{white}   $15-6\sqrt{2}$ &  \cellcolor{white} Katz--Tao~\cite{KT2002}  \\
 4 &   \cellcolor{white} $3.0858...$ &    \cellcolor{white} Katz--Zahl~\cite{KZ2}  & 11 &  \cellcolor{white} $17-7\sqrt{2}$ &  \cellcolor{white} Katz--Tao~\cite{KT2002}\\  
 5 & \cellcolor{yellow!50} $18/5$ & \cellcolor{yellow!50} Theorem~\ref{linear theorem}    &  12 & \cellcolor{yellow!50}   $31/4$ &  \cellcolor{yellow!50} Theorem~\ref{linear theorem}   \\ 
 6 &  \cellcolor{white} $7-2\sqrt{2}$ &  \cellcolor{white} Katz--Tao~\cite{KT2002}  &    13 &  \cellcolor{white} $21-9\sqrt{2}$ &  \cellcolor{white} Katz--Tao~\cite{KT2002}\\ 
 7 &  \cellcolor{yellow!50} $34/7$ &  \cellcolor{yellow!50} Theorem~\ref{linear theorem}  & 14 & \cellcolor{yellow!50}   $9$ &  \cellcolor{yellow!50} Theorem~\ref{linear theorem}   \\ 
 8 & \cellcolor{white}   $11-4\sqrt{2}$ &  \cellcolor{white} Katz--Tao~\cite{KT2002}   &   15 & \cellcolor{white}   $25-11\sqrt{2}$& \cellcolor{white} Katz--Tao~\cite{KT2002}\\ 
  \hline
\end{tabular}
    \caption{The state-of-the-art for the Kakeya set conjecture in low dimensions. New results are \colorbox{yellow!50}{highlighted.}
    }\label{dimension}\end{figure}




\subsection{Further variants of the Kakeya problem} It is an interesting problem to determine what can be said when the direction-separation hypothesis in Conjecture~\ref{maximal conjecture} is weakened; indeed, results of this kind have greatly influenced the current understanding of the Kakeya conjecture (see, for instance, \cite{Tao}). One classical theorem in this direction is due to Wolff \cite{Wolff1995} and considers families of tubes which satisfy the following hypothesis.

\begin{definition} Let $N \geq 1$ and $\T$ be a family of $\delta$-tubes in $\R^n$. We say that $\T$ satisfies the \emph{$(N)$-linear Wolff axiom} if 
\begin{equation*}
    \#\big\{ T \in \mathbb{T} :  T \subseteq E \big\} \leq N\delta^{-(n-1)} |E|
\end{equation*}
whenever $E \subseteq \R^n$ is a rectangular box of arbitrary dimensions.
\end{definition}

In \cite{Wolff1995}, Wolff showed that the maximal inequality 
\begin{equation}\label{recall maximal inequality}
\Big\| \sum_{T\in\mathbb{T}} \chi_{T}\Big\|_{L^p(\mathbb{R}^n)} \lesssim_{n,\varepsilon} N^{1-1/p} \delta^{-(n-1-\frac{n}{p})-\varepsilon} \Big(\sum_{T \in \T} |T| \Big)^{1/p}
\end{equation}
holds for the restricted range $p \geq \frac{n+2}{n}$ whenever $\T$ satisfies the $(N)$-linear Wolff axiom.\footnote{Strictly speaking, Wolff's theorem \cite{Wolff1995} holds under a slightly less restrictive condition referred to simply as \emph{the Wolff axiom}. See \cite{GZ2018} for a comparison of these conditions.} Furthermore, it is not difficult to see that any direction-separated $\T$ satisfies the $(N)$-linear Wolff axiom for some $N \sim 1$ and so his result provided similar progress for  Conjecture \ref{maximal conjecture}. 

Interestingly, there exist examples of tube families $\T$ in dimensions $n \geq 4$ that satisfy the $(N)$-linear Wolff axiom with $N \sim 1$, but for which \eqref{recall maximal inequality} fails to hold for the whole range $p \geq \frac{n}{n-1}$; see \cite{Tao2005}. In particular, when $n = 4$ one may construct such $\T$ for which~\eqref{recall maximal inequality} is only valid in Wolff's range $p \geq 3/2$. Examples of this kind are not direction-separated and therefore do not provide counterexamples to Conjecture~\ref{maximal conjecture}. 

To go beyond $p \geq 3/2$ in four dimensions, Guth and Zahl \cite{GZ2018} considered families of tubes which satisfy a more restrictive version of the $(N)$-linear Wolff axiom.

\begin{definition}
We say that $\mathbb{T}$ satisfies the {\it $(D,N)$-polynomial Wolff axiom} if  
\begin{equation*}
    \#\big\{ T \in \mathbb{T} :  |T \cap E| \geq \lambda |T| \big\} \leq N\delta^{-(n-1)} \lambda^{-n} |E|
\end{equation*}
whenever $\lambda\ge \delta$ and $E\subseteq \R^n$ is a semialgebraic set of complexity at most $D$.
\end{definition}

In this language, Theorem~\ref{PWA} states that for all $D\in\mathbb{N}$ and all $\varepsilon>0$ there is a constant $C(D,\varepsilon,n)$ such that any direction-separated family $\mathbb{T}$ satisfies the $(D,N)$-polynomial Wolff axiom with $N = C(D,\varepsilon,n)\delta^{-\varepsilon}$. Thus, the  following conjecture of Guth and Zahl~\cite[Conjecture 1.1]{GZ2018} is stronger than the Kakeya maximal conjecture.

\begin{conjecture}[Guth--Zahl \cite{GZ2018}]\label{PWAcon} Let   $p\ge \frac{n}{n-1}$. Then, for all $\varepsilon > 0$, there is a complexity $D = D_{\varepsilon,n} \in \mathbb{N}$ and a constant $C_{\varepsilon,n} >0$ such that
\begin{equation*}
\Big\| \sum_{T\in\mathbb{T}} \chi_T\Big\|_{L^p(\mathbb{R}^n)}\le C_{\varepsilon,n} N^{1-1/p}\delta^{-(n-1 - n/p)-\varepsilon}\Big(\sum_{T \in \T} |T| \Big)^{1/p}
\end{equation*}
whenever $0 < \delta < 1$, $N \geq 1$ and $\mathbb{T}$ satisfies the $(D,N)$-polynomial Wolff axiom.\footnote{Strictly speaking, the conjecture of \cite{GZ2018} is slightly weaker than Conjecture \ref{PWAcon} in some regards and stronger in others.}
\end{conjecture} 

It is easy to adapt C\'ordoba's $L^2$-argument \cite{Cordoba1977} to prove Conjecture \ref{PWAcon} for $n=2$. Guth and Zahl \cite{GZ2018} showed that in four dimensions, under the polynomial Wolff axioms, the $p \geq 3/2$ bound can be improved to $p\ge 121/81$.\footnote{The original paper \cite{GZ2018} claimed the range $p\ge 85/57$ but contained an arithmetic error, as highlighted in \cite{KZ2019}.}  Later, Katz and Zahl \cite{KZ2019} obtained a slight improvement over the Wolff bound $p \geq 5/3$ for Conjecture~\ref{PWAcon} in three dimensions. In all other dimensions the Wolff bound $p \geq \frac{n+2}{n}$ provides the previous best known result under the polynomial Wolff axioms alone. By carrying out the analysis of this paper, but only using the polynomial Wolff axiom rather than the nested estimates from Theorem \ref{PWA}, one obtains the following range of estimates.

\begin{theorem}\label{PW linear theorem} Conjecture \ref{PWAcon} is true in the range 
\begin{equation}\label{pwa linear exponent}
p \geq 1+\min_{2\le k\le n}\max\Big\{\, \Big(\frac{n}{n-1}\Big)^{n-k}, \,\frac{n-1}{n-k+1}\,\Big\}\frac{1}{n-1}.
\end{equation}
\end{theorem}

\renewcommand{\arraystretch}{1.3}
\renewcommand{\arraystretch}{1.3}
 \begin{figure}
\begin{tabular}{ ||c|c|c||c|c|c|| } 
 \hline
  $n=$ & $p \geq $  &  & $n=$ & $p \geq $  &  \\ 
   \hline 
 \cellcolor{white} 2 & 2 & C\'ordoba~\cite{Cordoba1977}  &  \cellcolor{yellow!50} 9 &  \cellcolor{yellow!50} $1 + 9^4/8^5$ &  \cellcolor{yellow!50} Theorem~\ref{PW linear theorem}\\ 
  \cellcolor{white} 3 & $5/3 - \varepsilon$ &  Katz--Zahl~\cite{KZ2019}  &\cellcolor{yellow!50}  10 & \cellcolor{yellow!50}   $1+10^5/9^6$ &  \cellcolor{yellow!50} Theorem~\ref{PW linear theorem}  \\
\cellcolor{white} 4 &   $121/81$ &   Guth--Zahl~\cite{GZ2018}  & \cellcolor{yellow!50} 11 &  \cellcolor{yellow!50} $7/6$ &  \cellcolor{yellow!50} Theorem~\ref{PW linear theorem}\\  
\cellcolor{yellow!50} 5 & \cellcolor{yellow!50} $1+5^2/4^3$ & \cellcolor{yellow!50} Theorem~\ref{PW linear theorem}    & \cellcolor{yellow!50}  12 & \cellcolor{yellow!50}   $1+12^6/11^7$ &  \cellcolor{yellow!50} Theorem~\ref{PW linear theorem}   \\ 
\cellcolor{white} 6 & 4/3 & Wolff~\cite{Wolff1995}  &   \cellcolor{yellow!50}  13 &  \cellcolor{yellow!50} $8/7$ &  \cellcolor{yellow!50} Theorem~\ref{PW linear theorem}\\ 
 \cellcolor{yellow!50} 7 &  \cellcolor{yellow!50} $1+7^3/6^4$ &  \cellcolor{yellow!50} Theorem~\ref{PW linear theorem}  &  \cellcolor{yellow!50}  14 & \cellcolor{yellow!50}   $1+14^7/13^8$ &  \cellcolor{yellow!50} Theorem~\ref{PW linear theorem}   \\ 
\cellcolor{yellow!50}  8 & \cellcolor{yellow!50}   $1+8^4/7^5$ &  \cellcolor{yellow!50} Theorem~\ref{PW linear theorem}   &  \cellcolor{yellow!50}  15 & \cellcolor{yellow!50}   $1+15^8/14^9$&  \cellcolor{yellow!50} Theorem~\ref{PW linear theorem}\\ 
  \hline
\end{tabular}
    \caption{The current state-of-the-art for Conjecture~\ref{PWAcon} in low dimensions. 
    }
    \label{exponent table}
\end{figure}

The above range of exponents is larger than Wolff's when $n=5$ or $n \geq 7$. To see this, note that for any $0 < r < 1$ there exists some integer $2 \leq k \leq n$ satisfying $k\in[r(n-1)+1,r(n-1)+2)$. Writing the endpoint in \eqref{pwa linear exponent} as $1 + \frac{\alpha_n}{n-1}$, it follows that
\begin{equation*}
\alpha_n < \inf_{0< r < 1}\max\Big\{\,\Big(1+\frac{1}{n-1}\Big)^{(n-1)(1-r)},\, \frac{1}{1-r}\,\Big\} \leq \Omega^{-1}=1.763....  
\end{equation*}
Here the omega constant $\Omega\in(1/2,1)$ is the solution to $e^{\Omega}=\Omega^{-1}$. In particular, Theorem~\ref{PW linear theorem} implies that Conjecture~\ref{PWAcon} is true in the range $p \geq 1+\frac{\Omega^{-1}}{n-1}$, yielding an improvement over Wolff's bound when $n \geq 9$. Calculating the precise value of $p_n$ for lower $n$, we find that Theorem~\ref{PW linear theorem} also improves the state-of-the-art for Conjecture~\ref{PWAcon} in dimensions $n=5, 7, 8$; see Figure~\ref{exponent table} for some explicit values for \eqref{pwa linear exponent}. 




\appendix

\section{Tools from real algebraic geometry}

For the reader's convenience, here we recall the definitions and results from real algebraic geometry that play a role in our arguments in Section~\ref{yes}.

\subsubsection*{Wongkew's theorem} We make considerable use of the following theorem of Wongkew \cite{Wongkew1993} (see also \cite{Guth2016, Zhang2017}), which bounds the volume of neighbourhoods of algebraic varieties.

\begin{theorem}[Wongkew \cite{Wongkew1993}]\label{Wongkew theorem} Suppose $\bZ$ is an $m$-dimensional variety in $\R^n$ with $\deg \bZ \leq d$. For any $0 < \rho \leq \lambda$ and $\lambda$-ball $B_{\lambda}$ the neighbourhood $N_{\!\rho}(\bZ\cap B_{\lambda})$ can be covered by $O_d( (\lambda/\rho)^{m})$ balls of radius $\rho$.
\end{theorem}

\subsubsection*{The Tarski--Seidenberg projection theorem} A fundamental result in the theory of semialgebraic sets is the Tarski--Seidenberg projection theorem, which is also referred to as ``quantifier elimination''. A useful reference for this material is~\cite{BPR}. 

\begin{theorem}[Tarski--Seidenberg]\label{Tarski}
Let $\Pi $ be the orthogonal projection of $\mathbb{R}^n$ into its first $n-1$ coordinates. Then for every $E\ge 1$, there is a constant $C(n,E) > 0$ so that, for every  semialgebraic $S \subset \mathbb{R}^n$ of complexity at most $E$, the projection $\Pi (S)$ has complexity at most $C(n,E)$. 
\end{theorem}

We repeatedly use Theorem~\ref{Tarski} to form semialgebraic \textit{sections} of semialgebraic sets.

\begin{corollary}\label{Tarski corollary} Let $S \subset \mathbb{R}^{2n}$ be a compact semialgebraic set of complexity at most $E$. Let~$\Pi$ be the orthogonal projection into the final $n$ coordinates 
$(\mathbf{a},\mathbf{d}) \mapsto \mathbf{d}.$
Then there is a constant $C(n,E)>0$, depending only on $n$ and $E$, and a semialgebraic set $Z$, of complexity at most~$C(n,E)$, so that
$$Z \subset S,\quad\quad
\Pi(Z)= \Pi(S),$$
and so that for each $\mathbf{d},$ there is at most one $\mathbf{a}$ with
$(\mathbf{a},\mathbf{d}) \in Z.$
\end{corollary}

This is Lemma 2.2 from \cite{KR2018}. It is a direct consequence of Theorem~\ref{Tarski}, as discussed in \cite{KR2018}. 

\subsubsection*{Gromov's algebraic lemma} The final key tool is the existence of useful parameterisations of semialgebraic sets, as guaranteed by the following lemma. 

\begin{lemma}[Gromov] \label{Gromov} 
For all integers $E, n, r \geq 1$, there exists $M(E, n, r) < \infty$ with the following
properties. For any compact semialgebraic  set $A \subset [0,1]^n$ of dimension $m$ and complexity at most~$E$, there exists
an integer $N\le M(E,n,r)$ and $C^r$ maps $\phi_1,\dots,\phi_N: [0,1]^m \longrightarrow [0,1]^n$ such that
$$\bigcup_{j=1}^N  \phi_j ([0,1]^m) =A\quad \text{and}\quad \|\phi_j\|_{C^r}:=\max_{|\alpha|\le r}\|\partial^\alpha\phi_j\|_{\infty} \leq 1.$$
\end{lemma}

 This result was originally stated by Gromov. Detailed proofs were later given by Pila and Wilkie~\cite{PW} and Burguet~\cite{Burguet}.




\end{document}